\documentclass[9pt]{amsart}
\usepackage{amsmath,amsfonts,amsthm,amssymb}
\usepackage{url}
\usepackage{verbatim}
\usepackage[usenames,dvips]{color}
\usepackage{graphicx}
\usepackage{xcolor}
\usepackage{mathrsfs,mathtools}
\usepackage{lmodern}
\usepackage{bm}
\usepackage[english]{babel}

\numberwithin{equation}{section}
\numberwithin{subsection}{section}
\newtheorem{theorem}{Theorem}
\newtheorem{lemma}{Lemma}[section]
\newtheorem{corollary}[lemma]{Corollary}
\newtheorem{definition}[lemma]{Definition}
\newtheorem{remark}[lemma]{Remark}

\newtheorem{proposition}[lemma]{Proposition}

\newcommand{\bfD}{{\bf D}}

\newcommand{\bmgamma}{\bm{\gamma}}

\newcommand{\bmxi}{\bm{\xi}}

\newcommand{\bmomega}{\bm{\omega}}

\newcommand{\bmA}{\bm{A}}
\newcommand{\bmb}{\bm{b}}
\newcommand{\bmB}{\bm{B}}

\newcommand{\bmC}{\bm{C}}

\newcommand{\bme}{\bm{e}}

\newcommand{\bmf}{\bm{f}}
\newcommand{\bmF}{\bm{F}}

\newcommand{\bmj}{\bm{j}}

\newcommand{\bmK}{\bm{K}}

\newcommand{\bmn}{\bm{n}}
\newcommand{\bmN}{\bm{N}}

\newcommand{\bmr}{\bm{r}}
\newcommand{\bmR}{\bm{R}}

\newcommand{\bmU}{\bm{U}}
\newcommand{\bmv}{\bm{v}}

\newcommand{\bmW}{\bm{W}}
\newcommand{\bmx}{\bm{x'}}

\newcommand{\bmzero}{\bm{0}}
\newcommand{\rmh}{\mathrm{h}}
\DeclareMathOperator{\curl}{\mathrm{curl}}
\DeclareMathOperator{\Div}{\mathrm{div}}
\DeclareMathOperator{\grad}{\mathrm{\nabla}}
\DeclareMathOperator{\gradp}{\nabla^{\perp}}
\DeclareMathOperator{\pd}{\partial}

\usepackage[colorlinks=true]{hyperref}

\title[]{  Analytical study of a generalized Dirichlet--Neumann operator for three-dimensional water waves with vorticity }

\author{S. Pasquali$^{(\ast)}$}
\address[$\ast$]{International School for Advanced Studied (SISSA), via Bonomea 265, 34136, Trieste, Italy}
\email[$\ast$]{stefano.pasquali@sissa.it}

\begin{document}

\begin{abstract}
In this paper we consider three-dimensional water waves with vorticity, under the action of gravity. We discuss a generalized Zakharov--Craig--Sulem formulation of the problem introduced by Castro and Lannes, which involves a generalized Dirichlet--Neumann operator. We study this operator in detail, extending some well-known results about the classical Dirichlet--Neumann operator for irrotational water waves, such as the Taylor expansion in homogeneous powers of the wave profile, the computation of its differential and a paralinearization result. We stress the fact that no geometric condition on either the velocity field or the vorticity is assumed.  
\end{abstract}

\maketitle

\tableofcontents

\section{Introduction} \label{sec:intro}

In this paper we study three-dimensional pure gravity waves on the free surface of a perfect fluid. We assume that the fluid has flat bottom with finite depth. We denote by $\bmU$ the velocity field of the fluid, and by $\bmomega \coloneqq \curl\bmU$ its vorticity.\\

The problem of existence of small-amplitude water waves has a long history. The first results for two-dimensional travelling water waves have been proved in the irrotational setting by Stokes \cite{stokes1880theory}, and later for periodic waves by Levi-Civita \cite{levi1925determination}, and in the case of nonzero vorticity by Gerstner \cite{gerstner1809theorie}. The problem of existence of two-dimensional water waves has been extensively studied, both in the irrotational setting (see \cite{iooss2005standing} \cite{wu2009almost} \cite{alazard2015sobolev} \cite{ionescu2015global} \cite{baldi2018time} \cite{bambusi2021hamiltonian} \cite{ai2022two} \cite{feola2024quasi} and references therein; see \cite{alazard2015gravity} \cite{berti2018amost} for gravity-capillary water waves), and in the case of nonzero vorticity (we mention \cite{wahlen2007hamiltonian} \cite{ifrim2018two}  \cite{berti2021traveling} \cite{berti2024pure} for the case of constant vorticity; see  \cite{constantin2004exact}  \cite{wahlen2006steady} \cite{ehrnstrom2022smooth} \cite{wahlen2023global} for the case of an arbitrary vorticity distribution). \\

On the other hand, there are fewer results about the existence of three-dimensional small-amplitude water waves. Regarding travelling waves in the irrotational case, we mention the results by Reeder and Shinbrot \cite{reeder1981three} and by Craig and Nicholls \cite{craig2000traveling,craig2002traveling} for \emph{gravity-capillary} water waves, and the results by Iooss and Plotnikov for \emph{pure gravity} water waves \cite{iooss2009small,iooss2011asymmetrical}; we also mention the survey \cite{haziot2022traveling} regarding travelling water waves. We also mention \cite{alazard2014cauchy} \cite{alazard2016cauchy} regarding the local Cauchy theory for gravity water waves, \cite{wu2011global} \cite{germain2012global} \cite{deng2017global} regarding global existence for small initial data, \cite{ionescu2019long} about long-time existence of gravity-capillary water waves in the periodic setting, \cite{buffoni2018variational} \cite{buffoni2022fully} regarding the existence of fully localised solitary waves for the three-dimensional gravity-capillary water waves system.

Recently there has been some interest in Beltrami fields, namely those fields for which $\curl \bmU = \alpha \, \bmU$, for some $\alpha \in \mathbb{R}$. We mention a variational formulation by Lokharu and Wahl\'en \cite{lokharu2019variational} for doubly periodic waves with relative velocities given by Beltrami fields, which is valid under general assumptions on the wave profile (admitting, for example, the case of overhanging wave profiles). More recently, Groves and Horn proved a more explicit variational formulation for Beltrami fields \cite{groves2020variational}, under the classical assumption that the free surface is given by the graph of an unknown function $\eta$ depending only on the horizontal directions: such formulation can be considered as a generalization of an alternative variational framework for three-dimensional irrotational water waves by Benjamin \cite{benjamin1984impulse}, and it allows to recover the classical  Zakharov--Craig--Sulem formulation in the irrotational case $\alpha =0$. Moreover, this formulation leads naturally to the definition of a generalized Dirichlet--Neumann operator, which reduces to the classical Dirichlet--Neumann operator in the irrotational case (see \cite{groves2024analytical} for more details). These formulations lead to existence results for small amplitude doubly periodic waves \cite{lokharu2020existence} \cite{groves2024analytical}. \\

Regarding three-dimensional water waves with general vorticity, we mention the paper by Castro and Lannes \cite{castro2015well}, in which the authors introduced a generalized Zakharov--Craig--Sulem formulation (widely used in the present work) and proved a local well-posedness result; moreover, the authors proved that their formulation is formally Hamiltonian, and that it reduces to the Zakharov--Craig--Sulem formulation in the irrotational case.  We mention \cite{wang2021local}, in which the authors establish local well-posedness for solutions with low regularity, using a different formulation compared to Castro--Lannes, see (6.1)-(6.8) in \cite{wang2021local}. More recently, in \cite{seth2024symmetric} the authors constructed symmetric doubly periodic gravity-capillary three-dimensional water waves with small vorticity. We mention a recent result by Ifrim, Pineau, Tataru and Taylor \cite{ifrim2025sharp}, in which the authors provide a complete local well-posedness theory in  Sobolev spaces for the free boundary incompressible Euler equations with zero surface tension on a connected fluid domain, working in an Eulerian framework which allow the authors to prove sharp results without using Nash--Moser theory. We also mention the papers \cite{christodoulou2000motion} \cite{coutand2007well} \cite{lindblad2005well} \cite{shatah2011local}, which deal with the local well-posedness theory of the free boundary Euler equations (and its variants) and take into account the effect of vorticity.

Finally, we mention a recent and important result by Ginsberg and Pusateri \cite{ginsberg2024long}, in which the authors proved long-time existence for small amplitude three-dimensional water waves with vorticity. From a technical point of view, the result by Ginsberg and Pusateri requires the following geometric assumption:
\begin{itemize}
\item[i.] the initial vorticity $\bmomega_0$ vanishes on the boundary of the fluid domain at time $t=0$, see statement of Theorem 2.2 of \cite{ginsberg2024long}.
\end{itemize}
As mentioned in pag.7 of \cite{ginsberg2024long}, condition $\mathrm{i.}$ is preserved in time. We point out that in \cite{ginsberg2024long} the authors decompose the velocity field such that its rotational part is tangential to the free surface, see (1.1d) and (2.1) of \cite{ginsberg2024long}; as a consequence, the system can be reduced to a system of two equations evaluated on the free boundary, like the Zakharov--Craig--Sulem formulation for the irrotational case (see Appendix B of \cite{ginsberg2024long}).\\

In the present paper we perform a detailed study of the water waves system with general vorticity, proving analytic properties of the generalized Dirichlet--Neumann operator $\mathcal{G}_{gen}[\eta]$, where $\eta$ is a function whose graph determines the free surface of the fluid domain.  Regarding the velocity field and the generalized Dirichlet--Neumann operator, we prove:
\begin{itemize}
\item[-] an estimate for the velocity field of the fluid;
\item[-] an analyticity result for the generalized Dirichlet--Neumann operator;
\item[-] a formula for the differential of the rotational part of the generalized Dirichlet--Neumann operator;
\item[-] an explicit (up to recursive formulae) Taylor expansion of the generalized Dirichlet--Neumann operator around $\eta=0$;
\item[-] a paralinearization result for the rotational part of the generalized Dirichlet--Neumann operator.
\end{itemize}

The assumptions used in the present paper involve only the regularity and the smallness of $(\eta,\Phi,\tilde{\bmomega})$ in suitable Sobolev and H\"older spaces, where $\Phi$ denotes the trace of the velocity potential at the surface, and $\tilde{\bmomega}$ denotes the straightened vorticity. We emphasise the fact that no geometric condition on either the velocity field or the vorticity is required; the lack of geometric assumptions leads to a highly non-trivial analysis of the rotational part of both the velocity field and of the generalized Dirichlet--Neumann operator. We also point out that the advantage of using the formulation by Castro and Lannes allows to describe very explicitly the contribution of the vorticity. 

We believe that the present work provides a very flexible framework which generalizes well-known results for the Zakharov--Craig--Sulem formulation in the irrotational case, and that it is suited to deal with many different problems, such as long-time existence of solutions, analysis of periodic wave patterns or solitary waves, and derivation of asymptotic models in various physical regimes.

\subsection{The Model} \label{subsec:model}

Consider an incompressible inviscid fluid occupying a three-dimensional domain with flat bottom, under the action of gravity. The fluid domain can be parametrized by
\begin{align} \label{eq:Domain}
D_\eta(t) &\coloneqq \{ (\bmx,z) \in \mathbb{R}^2 \times \mathbb{R}: -h < z < \eta(t,\bmx) \},
\end{align}
so that the free surface is given by the graph of an unknown function $\eta:\mathbb{R} \times \mathbb{R}^2 \to \mathbb{R}$, and $h>0$ is the depth of the fluid. We assume that the fluid has constant density equal to $1$.

The equations describing the flow are given by
\begin{align}
\Div \bmU &= 0, \; \; \text{in} \; \; D_\eta(t), \label{eq:Incompr} \\
\bmU_t + (\bmU \cdot \nabla_{\bmx,z})  \bmU &= - \nabla_{\bmx,z} \mathscr{P} - g \bme_3, \; \; \text{in} \; \; D_\eta(t), \label{eq:Euler} \\
\bmU \cdot \bme_3 &= 0 , \; \; \text{at} \; \; z=-h, \label{eq:Imperm} \\
\eta_t - \bmU \cdot \bmN &= 0, \; \; \text{at} \; \; z=\eta, \label{eq:KinFree} \\
\mathscr{P} &= \mathscr{P}_0, \; \; \text{at} \; \; z=\eta, \label{eq:PrBC}
\end{align}
where $\bmU: \overline{D_{\eta}(t)} \to \mathbb{R}^3$ is the velocity field, $\bme_3=(0,0,1)^T$, $\bmN$ denotes the outward normal vector, $g$ is the gravity constant, $\mathscr{P}_0$ is the atmospheric pressure (assumed to be constant).  The curl and divergence operator used above are defined as $\curl \bmU \coloneqq \nabla_{\bmx,z} \times \bmU$ and $\Div \bmU \coloneqq \nabla_{\bmx,z} \cdot \bmU$. 

Whenever it is not important, we omit the time-dependence of $D_\eta(t)$ and we will denote it by $D_\eta$. We also denote by $D_0$ the fixed domain $\mathbb{R}^2 \times (-h,0)$.

\begin{remark} \label{rem:Vorticity}

In addition to \eqref{eq:Incompr}-\eqref{eq:PrBC}, in the literature it is usually assumed that the vorticity of the fluid satisfies additional properties. Two examples are \emph{irrotational velocity fields}, which correspond to  
\begin{align} \label{eq:Irrot}
\curl \bmU &=  \bmzero, \; \; \text{in} \; \; D_\eta ,
\end{align}
and \emph{(strong) Beltrami fields}, which correspond to 
\begin{align} \label{eq:Beltrami}
\curl \bmU &=  \alpha \, \bmU, \; \; \text{in} \; \; D_\eta , \; \; \text{for some} \; \; \alpha \in \mathbb{R}.
\end{align}
\end{remark}

\subsection{Notation} \label{sec:notation}

Let $\bmF=(F_1,F_2,F_3)^T$ be a three-dimensional vector field, and denote by $\bmF_\rmh=(F_1,F_2)^T$ and by $\bmF_{\parallel}=\bmF_\rmh+F_3 \nabla \eta|_{z=\eta}$ the horizontal component of the tangential part of $\bmF$. In order to simplify the notation, we will sometimes indicate the evaluation at the free surface with an underscore. 

For the two-dimensional vector field $\bmf=(f_1,f_2)^T$ we denote $\bmf^\perp =(f_2,-f_1)^T$; we also introduce the following operators
\begin{align}
\grad = (\pd_x,\pd_y)^T, &\; \;  \gradp = (\pd_y,-\pd_x)^T, \label{eq:Op}
\end{align}
(this notation is consistent with \cite{groves2024analytical}, and differs from \cite{castro2015well}). According to the Hodge-Weyl decomposition for vector fields in $\mathbb{R}^2$, for any three-dimensional vector field $\bmF$ we have
\begin{align}
\bmF_{\parallel} &= \grad\Phi+\gradp\Psi, \label{eq:HWdec} \\
\Phi = \Delta^{-1}(\grad \cdot \bmF_{\parallel}), &\; \; \Psi=\Delta^{-1}(\gradp \cdot \bmF_{\parallel}), \label{eq:PhiPsi} 
\end{align}
where $\Delta^{-1}$ is the two-dimensional Newtonian potential. 

We denote by $\bmN \coloneqq (-\eta_x,-\eta_y,1)^T$ the normal vector at the free surface. Therefore for any vector field $\bmA: D_\eta \to \mathbb{R}^3$ we have
\begin{align*}
\underline{\bmA} \times \bmN &=
\begin{pmatrix}
\bmA_{\parallel}^{\perp} \\ \bmA_{\parallel}^{\perp} \cdot \nabla\eta
\end{pmatrix}
.
\end{align*}
We denote by $\bmv$ the horizontal component of the velocity field $\bmU$, and by $\mathrm{w}$ its vertical component, so that
\begin{align*}
\bmU &= 
\begin{pmatrix}
\bmv \\ \mathrm{w}
\end{pmatrix}
 .
\end{align*}

We denote by $\mathcal{S}(\mathbb{R}^2)$ the space of Schwartz functions on $\mathbb{R}^2$.

We set $\Lambda\coloneqq(1-\Delta)^{1/2}$, and we define for all $s \in \mathbb{R}$ the Sobolev space
\begin{align*}
H^{s}(\mathbb{R}^2) &\coloneqq \{ u \in \mathcal{S}'(\mathbb{R}^2) : \Lambda^s f \in L^2(\mathbb{R}^2) \} ,  \\
\| u \|_{H^{s}(\mathbb{R}^2)} &\coloneqq \| \Lambda^s \; u \|_{L^2(\mathbb{R}^2)} , \; \; \forall \, u \in H^{s}(\mathbb{R}^2) ,
\end{align*}
and for all $k \in \mathbb{N}$ the Sobolev space 
\begin{align*}
H^{k}(D_{\eta}) &\coloneqq \{ u \in \mathcal{S}'(D_{\eta}) :  \partial^{\alpha} f \in L^2(D_{\eta})  \;\; \forall  \alpha \in \mathbb{N}^3 : |\alpha| \leq k \} ,  \\
\| u \|_{H^{k}(D_{\eta})} &\coloneqq \sum_{ \alpha \in \mathbb{N}^3 : |\alpha| \leq k} \| \partial^{\alpha} f \|_{ L^2(D_{\eta}) }  , \; \; \forall \, u \in H^{k}(D_{\eta}) .
\end{align*}
Notice that by an extension argument standard Sobolev embedding theorems holds true also if we replace $\mathbb{R}^3$ by $D_0$ (see also proof of Theorem 4.10 in \cite{groves2020variational}).

We also introduce for all $s \in \mathbb{R}$ the space
\begin{align*}
H^s_0(\mathbb{R}^2) &\coloneqq \{ u \in H^s(\mathbb{R}^2) : \exists v \in H^{s+1}(\mathbb{R}^2) \; \; \mathrm{s.t.} \; \; u = |\bfD| v  \}, \;\; \bfD \coloneqq-\mathrm{i}\nabla , \\
\| u \|_{H^{s}_0(\mathbb{R}^2)} &\coloneqq \| \; |\bfD|^{-1} u \|_{ H^{s+1}(\mathbb{R}^2) }, \; \; \forall \, u \in H^{s}_0(\mathbb{R}^2) . 
\end{align*}

Moreover, we define for all $s \in \mathbb{R}$ and $k \in \mathbb{N}$ the space $H^{s,k}(D_0)$ by
\begin{align*}
H^{s,k}(D_0) &\coloneqq \bigcap_{j=0}^k H^j( (-h,0) ; H^{s-j}(\mathbb{R}^2) ),  \\
\| u \|_{H^{s,k}(D_0)} &\coloneqq \sum_{j=0}^k \| \Lambda^{s-j} \; \partial_z^j u \|_{L^2(D_0)} , \; \; \forall \, u \in H^{s,k}(D_0) ,
\end{align*}
and we recall the following embedding (see Proposition 2.12 in \cite{lannes2013water})
\begin{equation*}
H^{s+1/2,1}(D_0) \subset L^\infty([-h,0] ; H^s(\mathbb{R}^2) ) , \;\; \forall s \in \mathbb{R} ,
\end{equation*}
where the space $ L^\infty([-h,0] ; H^s(\mathbb{R}^2) )$ is endowed with the canonical norm
\begin{align*}
\| f \|_{ L^\infty([-h,0] ; H^s(\mathbb{R}^2) ) } &\coloneqq \sup_{w \in [-h,0]} \|f(\cdot,w)\|_{ H^s(\mathbb{R}^2) } , \;\; \forall f \in L^\infty([-h,0] ; H^s(\mathbb{R}^2) ) .
\end{align*}

We also recall the definition of H\"older and Zygmund spaces. Following Appendix A of \cite{alazard2015sobolev}, we denote by $C^0(\mathbb{R}^2)$ the space of \emph{bounded} continuous functions; for any $s \in \mathbb{N}$, we denote by $C^s(\mathbb{R}^2)$ the space of $C^0(\mathbb{R}^2)$-functions whose derivative of order less or equal to $s$ are in $C^0(\mathbb{R}^2)$. For any $s \in (0,+\infty) \setminus \mathbb{N}$, we denote by $C^s(\mathbb{R}^2)$ the space of \emph{bounded} functions on $\mathbb{R}^2$ whose derivatives of order $[s]$ are uniformly H\"older continuous with exponent $s-[s]$. 

Next, choose a function $\Gamma \in C^\infty_0(B(0,1])$ which is equal to $1$ on $B(0,1/2]$; set $\gamma(\xi)\coloneqq\Gamma(\xi/2)-\Gamma(\xi)$ ($\gamma$ is supported on the annulus $\{ 1/2 \leq |\xi| \leq 2\}$, then we have that for any $\xi \in \mathbb{R}^2$
\begin{align*}
1 & = \Gamma(\xi) + \sum_{j \in \mathbb{N}} \gamma(2^{-j} \xi),
\end{align*}
and, recalling that $\bfD =-\mathrm{i}\nabla$, we set $\Delta_j \coloneqq\gamma(2^{-j} \bfD)$ for any $j \in \mathbb{Z}$, and $S_0\coloneqq\Gamma(\bfD)$. Then for any $s \in \mathbb{R}$ we define the Zygmund space $C^s_{\ast}(\mathbb{R}^2)$  as the space of temperate distributions $f \in \mathcal{S}'(\mathbb{R}^2)$ such that
\begin{align*}
\| f \|_{ C^s_{\ast}(\mathbb{R}^2) } &\coloneqq \|S_0 f \|_{L^\infty(\mathbb{R}^2)} + \sup_{j \in \mathbb{N}} 2^{js} \|\Delta_j f \|_{L^\infty(\mathbb{R}^2)}  < + \infty .
\end{align*}
We mention that for $s \in (0,+\infty) \setminus \mathbb{N}$ we have that $C^s_{\ast}(\mathbb{R}^2) = C^s(\mathbb{R}^2)$ and that the norms $\|\cdot\|_{ C^s_{\ast}(\mathbb{R}^2) }$ and $\|\cdot\|_{ C^s(\mathbb{R}^2) }$ are equivalent (see also Proposition 41.16 of \cite{metivier2008paradifferential}). We also use the notations $C^0(D_0)$ and $C^s(D_0)$ to denote bounded continuous functions on $D_0$ and bounded functions on $D_0$ whose derivatives of order $[s]$ are uniformly H\"older continuous with exponent $s-[s]$, respectively. We also mention the following consequences of the Sobolev embedding theorem in Zygmund spaces (see Proposition 8.4 in \cite{taylor2011partial3}),
\begin{equation*}
H^s(\mathbb{R}^2) \subset C^{s-1}(\mathbb{R}^2), \;\; H^s(D_0) \subset C^{s-3/2}(D_0) \;\; \forall s \in \mathbb{R} \setminus \frac{1}{2} \mathbb{N}.
\end{equation*}

Next, we define for all $k \in \mathbb{N}$ the homogeneous Sobolev space
\begin{align*}
\dot{H}^{k+1}( D_\eta ) &\coloneqq \{ f \in L^2_{\mathrm{loc}}( D_{\eta} ) : \nabla_{\bmx,z} f  \in ( H^{k}( D_\eta ) )^3 \}, 
\end{align*}
endowed with the norm $\| f \|_{ \dot{H}^{k+1}( D_\eta ) } \coloneqq \| \nabla_{\bmx,z} f \|_{ H^k(D_{\eta}) }$, and for all $s \in \mathbb{R}$ 
\begin{align*}
\dot{H}^{s+1}( \mathbb{R}^2 ) &\coloneqq \{ f \in L^2_{\mathrm{loc}}( \mathbb{R}^2 ) : \nabla f  \in ( H^{s}( \mathbb{R}^2 ) )^2 \},
\end{align*}
endowed with the norm $\| f \|_{ \dot{H}^{s+1}( \mathbb{R}^2 ) } \coloneqq \| \nabla f \|_{ H^s(\mathbb{R}^2) }$.

By using the definitions of the above norms one can check the following embedding of homogeneous Sobolev space,
\begin{equation*}
\dot{H}^s(\mathbb{R}^2) \cap C^0(\mathbb{R}^2) \subset C^{s-1}_{\ast}(\mathbb{R}^2) , \;\; \forall s \in \mathbb{R},
\end{equation*}
see also Theorem 1.50 of \cite{bahouri2011fourier} for a closely related result.

\subsection{The generalized Zakharov--Craig--Sulem formulation} \label{subsec:genZCS}

We now introduce a recent formulation for water waves with vorticity by Castro and Lannes, which generalizes the well-known Zakharov--Craig--Sulem formulation for irrotational velocity fields.  We defer to Sec. 2 of \cite{castro2015well} for more details.

Denoting by $\Pi$ and $\Pi_{\perp}$ respectively the projection onto gradient vector fields and the projection onto orthogonal gradient vector fields, we obtain
\begin{align*}
\Pi = \grad \; \grad^T \; \Delta^{-1}, &\; \; \Pi_{\perp} = \gradp \; (\gradp)^T \; \Delta^{-1} .
\end{align*}
We can use the Hodge-Weyl decomposition \eqref{eq:HWdec} in order to decompose $\bmU_{\parallel}$ as follows,
\begin{align*}
\bmU_{\parallel} = \Pi \bmU_{\parallel} + \Pi_{\perp} \bmU_{\parallel}  &= \grad \Phi + \gradp\Psi, 
\end{align*}
for some scalar functions $\Phi$ and $\Psi$. Hence, by taking the trace of \eqref{eq:Euler} at the free surface, we obtain the following equation for $\Phi$,
\begin{align} \label{eq:Euler3}
\Phi_t + g \, \eta + \frac{1}{2}  |\bmU_{\parallel}|^2 - \frac{1}{2} ( (1+|\nabla\eta|^2) \underline{w}^2 ) - \nabla \cdot \Delta^{-1} \, ( \underline{\bmomega} \cdot \bmN \; \underline{\bmv}^{\perp} ) &= 0. 
\end{align}

We also observe that by taking the curl of \eqref{eq:Euler}, we obtain the vorticity equation
\begin{align} \label {eq:vorticity}
\bmomega_t + (\bmU \cdot \nabla_{\bmx,z}) \bmomega &= (\bmomega \cdot \nabla_{\bmx,z}) \bmU, \; \; \text{in} \; \; D_\eta .
\end{align}

\begin{remark} \label{rem:EqPsi}

We just point out that the function $\Psi$ in the decomposition of $\bmU_{\parallel}$ is completely determined by $\bmomega$ and $\eta$. Indeed, since for any three-dimensional vector field $\bmF$ the following identity holds true,
\begin{align*}
\underline{\curl \bmF} \cdot \bmN &= \nabla \cdot \bmF_{\parallel}^{\perp} , 
\end{align*}
we get
\begin{align*}
\underline{\bmomega} \cdot \bmN &= \nabla \cdot \bmU_{\parallel}^{\perp} , 
\end{align*}
so that, if $\bmomega \in ( L^{2}(D_\eta) )^3$ is divergence-free, we can obtain $\Psi$ as the unique solution in $\dot{H}^{3/2}(\mathbb{R}^2)$ of
\begin{align*}
-\Delta\Psi &= \underline{\bmomega} \cdot \bmN .
\end{align*}
\end{remark}

Now let us assume that $\eta \in H^{k}(\mathbb{R}^2)$ with $k > \frac{5}{2}$, and let us denote by $D_\eta$ the domain
\begin{align} \label{eq:TimeIndDom}
D_\eta &\coloneqq \{ (\bmx,z) \in \mathbb{R}^2 \times \mathbb{R}: -h < z < \eta(\bmx) \} .
\end{align}
Assume moreover that the above domain is \emph{strictly connected}, namely that there exists $h_0>0$ such that
\begin{align} \label{eq:StrConnected}
h + \eta(\bmx) &\geq h_0, \; \; \forall \; \bmx \in \mathbb{R}^2 .
\end{align}

Then, it is known that it is possible to reconstruct the velocity field $\bmU$ in terms of $\bmomega$, $\nabla\Phi$ and $\eta$ (see Theorem 2.4 in \cite{castro2015well}, which holds also for $\eta \in W^{2,\infty}(\mathbb{R}^2)$; see also Corollary 2.44 in \cite{lannes2013water} for the corresponding result for the irrotational case). 

In order to state our result, we have to introduce some auxiliary definitions. Indeed, let $\eta \in H^{k}(\mathbb{R}^2)$ with $k > \frac{5}{2}$ be such that \eqref{eq:StrConnected} holds true, and let us denote by $D_\eta$ the domain as in \eqref{eq:TimeIndDom}. We define the subspace of $( L^{2}(D_\eta) )^3$ of divergence-free vector fields as
\begin{align} \label{eq:DivFree}
H(\Div_0,D_\eta) &\coloneqq \{ \bmF \in ( L^{2}(D_\eta) )^3 : \Div \bmF = 0 \} .
\end{align}
Moreover, we introduce the following set 
\begin{align} \label{eq:DivFreeAux}
H_b(\Div_0,D_\eta) &\coloneqq \{ \bmF \in ( L^{2}(D_\eta) )^3 : \Div \bmF = 0 , \bmF \cdot \bmN|_{z=-h} \in H^{-1/2}_0(\mathbb{R}^2) \}  ,
\end{align}
which we endow with the norm
\begin{align*}
\| \bmF \|_{2,b} &\coloneqq \| \bmF \|_{ ( L^{2}(D_\eta) )^3} +  \| \bmF \cdot \bmN|_{z=-h} \|_{ H^{-1/2}_0(\mathbb{R}^2) } , \; \; \forall \; \bmF \in H_b(\Div_0,D_\eta) .
\end{align*}

\begin{proposition} \label{prop:CLSolBVP}

Let $k > \frac{5}{2}$, let $\eta \in H^{k}(\mathbb{R}^2)$ be such that \eqref{eq:StrConnected} holds true, and let us denote by $D_\eta$ the domain as in \eqref{eq:TimeIndDom}. Let also $\bmomega \in H_b(\Div_0,D_\eta) $ and $\Phi \in \dot{H}^{3/2}(\mathbb{R}^2)$.

Then there exists a unique $\bmU \in ( H^{1}( D_\eta ) )^3$ which solves the boundary value problem 
\begin{align} \label{eq:BVP}
\begin{cases}
\curl \bmU = \bmomega , & \text{in} \; \; D_\eta ,  \\
\Div \bmU = 0 , & \text{in} \; \; D_\eta ,  \\
\bmU \cdot \bme_3 = 0 , & \text{at} \; \; z=-h,  \\
\bmU_{\parallel} = \grad\Phi - \gradp \Delta^{-1} ( \bmomega \cdot \bmN ) , & \text{at} \; \; z=\eta . 
\end{cases}
\end{align}
Moreover, one has
\begin{align} \label{eq:bfuDecomp}
\bmU &= \curl \bmA + \nabla_{\bmx,z} \varphi ,
\end{align}
where $\bmA \in ( H^{2}(D_\eta) )^3$ solves
\begin{align} \label{eq:systA}
\begin{cases}
\curl \, \curl \bmA = \bmomega , & \text{in} \; \; D_\eta , \\
\Div \bmA = 0 , & \text{in} \; \; D_\eta ,  \\
\bmA \times \bme_3 = \bmzero , & \text{at} \; \; z=-h ,  \\
\curl \bmA \cdot \bme_3 = 0 , & \text{at} \; \; z=-h ,  \\
\bmA \cdot \bmN = 0 , & \text{at} \; \; z=\eta ,  \\
(\curl \bmA)_{\parallel} = -\gradp \Delta^{-1} (\bmomega \cdot \bmN) , & \text{at} \; \; z=\eta , 
\end{cases}
\end{align}
while $\varphi \in H^{2}(D_\eta)$ solves
\begin{align} \label{eq:systphi}
\begin{cases}
\Delta_{\bmx,z} \varphi = 0 , & \text{in} \; \; D_\eta ,  \\
\varphi = \Phi , & \text{at} \; \; z=\eta ,  \\
\partial_n \varphi =  0, & \text{at} \; \; z=-h . 
\end{cases}
\end{align}
Then
\begin{align}
& \| \bmU \|_{( L^{2}(D_\eta) )^3} + \| \nabla_{\bmx,z}\bmU \|_{( L^{2}(D_\eta) )^{3 \times 3} } \nonumber \\
&\leq C( h_0^{-1}, h, \|\eta\|_{H^{k}(\mathbb{R}^2)} ) \; \left(  \|\bmomega\|_{2,b} + \| \nabla\Phi \|_{( H^{1/2}(\mathbb{R}^2) )^2}  \right) .  \label{eq:EstH1u}
\end{align}
\end{proposition}

\begin{remark} \label{rem:CLdecomp}

We point out that the statement of Proposition \ref{prop:CLSolBVP} improves the corresponding result by Castro and Lannes (see Theorem 2.4 of \cite{castro2015well}); indeed, by Sobolev embedding we have that $\eta \in C^{3/2+\varepsilon}(\mathbb{R}^2)$, for some $\varepsilon>0$. 
\end{remark}

For completeness, we prove Proposition \ref{prop:CLSolBVP} in Appendix \ref{sec:proofDivCurl}. Moreover, the proof allows us to reformulate the boundary value problem \eqref{eq:systA} in an equivalent way (see Corollary \ref{cor:AltStrong}), so that its solution can be written explicitly in terms of a Green matrix (see Sec. \ref{sec:GDNOhom}; see also Sec. 4 of \cite{groves2020variational} for a similar formulation for Beltrami flows).

\begin{remark} \label{rem:bfuDecomp}

The first two equations in \eqref{eq:systA} imply that
\begin{align} 
\bmomega &= \curl \curl \bmA \, = \, \nabla_{\bmx,z}( \nabla_{\bmx,z} \cdot \bmA) - (\nabla_{\bmx,z} \cdot \nabla_{\bmx,z})\bmA \, = \, - \Delta_{\bmx,z} \bmA, \; \; \mathrm{in} \; \; D_\eta . \label{eq:1EqA}
\end{align}

\end{remark}

Let $k > \frac{5}{2}$, $\eta \in H^{k}(\mathbb{R}^2)$ be such that \eqref{eq:StrConnected} holds true, and let us denote by $D_\eta$ the domain as in \eqref{eq:TimeIndDom}. We now define the following linear maps: 
\begin{align*}
\mathbb{A}[\eta] &: H_b(\Div_0,D_\eta) \to ( H^{2}(D_\eta) )^3, \; \; \mathbb{A}[\eta]\bmomega \coloneqq \bmA, \nonumber \\
\mathbb{U}_I[\eta] &: \dot{H}^{3/2}(\mathbb{R}^2) \to ( H^{1}(D_\eta) )^3, \; \; \mathbb{U}_I[\eta]\Phi \coloneqq \nabla_{\bmx,z}\varphi, \nonumber \\
\mathbb{U}_{II}[\eta] &: H_b(\Div_0,D_\eta) \to ( H^{1}(D_\eta) )^3, \; \; \mathbb{U}_{II}[\eta]\bmomega \coloneqq \curl \bmA, \nonumber \\
\mathbb{U}[\eta] &: \dot{H}^{3/2}(\mathbb{R}^2) \times H_b(\Div_0,D_\eta) \to ( H^{1}(D_\eta) )^3, \\
\mathbb{U}[\eta](\Phi,\bmomega) &\coloneqq \mathbb{U}_{I}[\eta]\Phi + \mathbb{U}_{II}[\eta]\bmomega . \nonumber 
\end{align*}
where $\bmA$ and $\varphi$ are given by the decomposition \eqref{eq:bfuDecomp}. We also define the map which gives the tangential velocity field at the free surface,
\begin{align*}
\mathbb{U}_{\parallel}[\eta] &: \dot{H}^{3/2}(\mathbb{R}^2) \times H_b(\Div_0,D_\eta) \to ( H^{1}(D_\eta) )^2, \\\
\mathbb{U}_{\parallel}[\eta](\Phi,\bmomega) &\coloneqq \nabla\Phi - \gradp \Delta^{-1} ( \underline{\bmomega} \cdot \bmN ) . \nonumber 
\end{align*}
Similarly, if we denote by $\mathbb{V}[\eta]$ and $\mathbb{W}[\eta]$ the operators that reconstruct respectively the horizontal and the vertical part of the velocity field, we can define
\begin{align*}
\underline{\mathbb{V}}[\eta](\Phi,\bmomega) \coloneqq \mathbb{V}[\eta](\Phi,\bmomega)|_{z=\eta} , &\; \; \underline{\mathbb{W}}[\eta](\Phi,\bmomega) \coloneqq \mathbb{W}[\eta](\Phi,\bmomega)|_{z=\eta} ,
\end{align*}
so that
\begin{align*}
\mathbb{U}_{\parallel}[\eta](\Phi,\bmomega) &= \underline{\mathbb{V}}[\eta](\Phi,\bmomega) + \underline{\mathbb{W}}[\eta](\Phi,\bmomega) \, \nabla\eta .
\end{align*}

Moreover, under the same assumptions on $\eta$ we can also define a \emph{generalized Dirichlet--Neumann operator}
\begin{align} 
\mathcal{G}_{gen}[\eta] &: \dot{H}^{3/2}(\mathbb{R}^2) \times H_b(\Div_0,D_\eta) \to  H^{1/2}(\mathbb{R}^2) , \nonumber \\
\mathcal{G}_{gen}[\eta](\Phi,\bmomega) &\coloneqq \mathbb{U}[\eta](\Phi,\bmomega) \cdot \bmN |_{z=\eta} . \label{eq:GDNO}
\end{align}
With this definition we have
\begin{align*}
\underline{\mathbb{W}}[\eta](\Phi,\bmomega) &= \frac{ \mathcal{G}_{gen}[\eta](\Phi,\bmomega) + \nabla\eta \cdot \mathbb{U}_{\parallel}[\eta](\Phi,\bmomega) }{1+|\nabla\eta|^2} .
\end{align*}

Hence the system \eqref{eq:Incompr}-\eqref{eq:PrBC} is equivalent to
\begin{align}
\Div \bmomega &= 0 , \; \; \text{in} \; \; D_\eta . \label{eq:IncomprGZCS} \\
\eta_t - \mathcal{G}_{gen}[\eta](\Phi,\bmomega) &= 0 , \; \; \text{at} \; \; z=\eta , \label{eq:KinFreeGZCS}
\end{align}
\begin{align}
&\Phi_t + g \; \eta + \frac{1}{2} | \mathbb{U}_{\parallel}[\eta](\Phi,\bmomega) |^2 - \frac{ \left(  \mathcal{G}_{gen}[\eta](\Phi,\bmomega) + \nabla\eta \cdot \mathbb{U}_{\parallel}[\eta](\Phi,\bmomega)  \right)^2 }{ 2 (1+|\nabla\eta|^2)}  \nonumber \\
&\qquad + \gradp \cdot \Delta^{-1} \, \left( \underline{\bmomega} \cdot \bmN \; \underline{\mathbb{V}}[\eta](\Phi,\bmomega) \right) = 0 , \, \text{at} \; \; z=\eta , \label{eq:EulerGZCS} 
\end{align}
\begin{align}
\bmomega_t + \left( \mathbb{U}[\eta](\Phi,\bmomega) \cdot \nabla_{\bmx,z} \right) \, \bmomega &= \left( \bmomega \cdot \nabla_{\bmx,z} \right) \, \mathbb{U}[\eta](\Phi,\bmomega) , \; \; \text{in} \; \; D_\eta . \label{eq:VorticityGZCS} 
\end{align}

\begin{remark} \label{rem:Incompr}

Eq. \eqref{eq:IncomprGZCS} is automatically satisfied for all times if it is satisfied at time $t=0$.
\end{remark}

\begin{remark} \label{rem:Irrot}

In the irrotational case $\bmomega \equiv \bmzero$ we have that 
\begin{align*}
\mathbb{U}_{\parallel}[\eta](\Phi,\bmzero) = \nabla\Phi , &\quad \mathcal{G}_{gen}[\eta](\Phi,\bmzero) = \mathcal{G}[\eta]\Phi , 
\end{align*}
where $\mathcal{G}[\eta]$ denotes the classical Dirichlet--Neumann operator. Moreover, the system \eqref{eq:IncomprGZCS}-\eqref{eq:VorticityGZCS} reduces to the usual Zakharov--Craig--Sulem formulation, namely
\begin{align}
& \eta_t - \mathcal{G}[\eta]\Phi = 0 , \nonumber \\
& \Phi_t + g \, \eta + \frac{1}{2} |\nabla\Phi|^2 - \frac{1}{ 2 (1+|\nabla\eta|^2)} \left(  \mathcal{G}[\eta]\Phi + \nabla\eta \cdot \nabla\Phi  \right)^2 = 0 ., \label{eq:ZCS}
\end{align}
\end{remark}

Observe that in the irrotational case the system \eqref{eq:ZCS} involve functions depending only on the horizontal variables, while in the system \eqref{eq:IncomprGZCS}-\eqref{eq:VorticityGZCS} the vorticity $\bmomega$ is a function defined on the time-dependent fluid domain $D_\eta(t)$. In order to deal with this difficulty we introduce the following straightening diffeomorphism (also called admissible diffeomorphism in the literature, see Definition 2.13 in \cite{lannes2013water}) ,
\begin{align}
\Sigma(t): D_0 \to D_\eta(t) , &\quad \Sigma(t)(\bmx,w) \coloneqq ( \bmx, w+ \sigma(t,\bmx,w) ), \nonumber \\
\sigma(t,\bmx,-h)=0, &\quad \sigma(t,\bmx,0)=\eta(t,\bmx)\in H^{s}(\mathbb{R}^2), \; \; s > 2,  \label{eq:straight} 
\end{align}
where $\eta$ satisfies condition \eqref{eq:StrConnected}, the coefficients of the Jacobian matrix $J_{\Sigma} = \nabla_{\bmx,w}\Sigma$ belong to $L^\infty(D_0)$, and there exists a constant 
\begin{align*}
M_0 &= M_0( h_0^{-1}, \|\eta\|_{ H^{s}(\mathbb{R}^2)} , \|\eta\|_{ C^{1}(\mathbb{R}^2)} )>0
\end{align*}
such that
\begin{align*}
\| J_{\Sigma,i,j} \|_{L^\infty(D_0)} &\leq M_0 , \; \; \forall i,j=1,2,3,
\end{align*}
and such that the determinant of the Jacobian matrix $J_{\Sigma}$ is uniformly bounded from below on $D_0$ by $c_0>0$ satisfying $M_0 \geq c_0^{-1}$ (this last property implies that $1 + \partial_w \sigma(t) \neq 0$ in $D_0$). 

Moreover, we say that a straightening diffeomorphism is \emph{regularizing} if
\begin{align*}
\| \nabla_{\bmx,w}\sigma \|_{( H^{s-1/2,0}(D_0) )^3} &\leq M_0 \, \|\eta\|_{H^{s}(\mathbb{R}^2)}.
\end{align*}

We now present two common examples of straightening diffeomorphisms (see Sec. 2.2.2 of \cite{lannes2013water} for more details).

\begin{remark} \label{rem:TrivialDiffeo}

The so-called \emph{trivial diffeomorphism} between $D_0$ and $D_{\eta}(t)$ is given by \eqref{eq:straight}, where
\begin{align} \label{eq:TrivialDiffeo}
\sigma(t,\bmx,w) &\coloneqq \left( 1 + \frac{w}{h} \right) \eta(t,\bmx) .
\end{align}
\end{remark}

\begin{proposition} \label{prop:RegDiffeo}

Let $s>2$ and assume that $\eta \in H^{s}(\mathbb{R}^2)$ satisfies condition \eqref{eq:StrConnected}. Let $\chi \in C^{\infty}_c(\mathbb{R})$ be a positive function equal to $1$ in a neighborhood of the origin, and consider the diffeomorphism \eqref{eq:straight}, where
\begin{align}
\sigma(t,\bmx,w) &\coloneqq \left( 1 + \frac{w}{h} \right) \eta^{(\delta)}(t,\bmx,w) , \, \eta^{(\delta)}(t,\bmx,w) \coloneqq \chi(\delta \, w \, |\bfD|)\eta(t,\bmx) . \label{eq:straightReg}
\end{align}
Then
\begin{itemize}
\item[i.] if $\delta >0$ in \eqref{eq:straightReg} satisfies
\begin{align*}
\delta &< \frac{ h_0 }{C(\chi) \|\eta\|_{ H^{s}(\mathbb{R}^2)} } , \;\; C(\chi) \coloneqq \|\chi'\|_{L^\infty(\mathbb{R})} \, \left[ \int_{\mathbb{R}^2} (1+|\bmxi|^2)^{-(s-1)} \, \mathrm{d}\bmxi \right]^{1/2} < +\infty ,
\end{align*}
then the diffeomorphism defined in \eqref{eq:straightReg} is such that $J_{\Sigma} = \nabla_{\bmx,z}\Sigma$ belong to $( L^\infty(D_0) )^0$, and there exists $M_0>0$ such that
\begin{align*}
\| J_{\Sigma,i,j} \|_{L^\infty(D_0)} &\leq M_0 , \; \; \forall i,j=1,2,3.
\end{align*}
Moreover, the determinant of the Jacobian matrix $J_{\Sigma}$ is uniformly bounded from below on $D_0$ such that $M_0 \geq 1/c_0$, where $c_0 = h_0 - \delta \, C(\chi) \, \|\eta\|_{ H^{s}(\mathbb{R}^2) }$.
\item[ii.] if the assumption of i. holds true, then the diffeomorphism defined in \eqref{eq:straightReg} is regularizing.
\end{itemize}

\end{proposition}

We also recall the following smoothing property (see Lemma 2.20 in \cite{lannes2013water}): let $s \in \mathbb{R}$ and $\lambda_1>0$, then for all $f \in C_c(\mathbb{R})$ there exists a constant $C=C(f)>0$ such that
\begin{align} \label{eq:smoothing}
\left\| f(\lambda_1^{1/2} w |\bfD|) u \right\|_{H^s(D_0)}^2 &\leq C \, \left| (1+\lambda_1^{1/2} |\bfD|)^{-1/2} u \right|_{H^s(\mathbb{R}^2)}^2 .
\end{align}

Observe that by \eqref{eq:smoothing} (with $\lambda_1=\delta^2$) we have that 
\begin{align*}
\| \eta^{(\delta)} \|_{H^{s+1/2,2}(D_0)} &\leq C(1/\delta) \|\eta\|_{H^{s}(\mathbb{R}^2)} . 
\end{align*}
As a comparison, for the trivial diffeomorphism of Remark \ref{rem:TrivialDiffeo} one need to control $H^s(\mathbb{R}^2)$-norm of the surface parameterization $\eta$ in order to control the $H^s(D_0)$-norm of $\sigma$ in \eqref{eq:TrivialDiffeo}. 

By using straightening diffeomorphisms the role of the vorticity $\bmomega$ is replaced by the straightened vorticity $\tilde{\bmomega}\coloneqq\bmomega \circ \Sigma(t)$, and role of the velocity field $\bmU$ is replaced by the straightened velocity field $\tilde{\bmU}\coloneqq\bmU \circ \Sigma(t)$. As a consequence, we can define a straightened version of the generalized Zakharov--Craig--Sulem formulation \eqref{eq:IncomprGZCS}-\eqref{eq:VorticityGZCS}, where the unknowns $(\eta,\Phi,\tilde{\bmomega})$ are defined on the time-independent domain $D_0$.

\section{Main results} \label{sec:results}

Here we present the main results of the paper.

\subsection{Analytical results for the generalized Dirichlet--Neumann operator} \label{subsec:analGDNO}

Exploiting the generalized Zakharov--Craig--Sulem formulation of Sec. \ref{subsec:genZCS}, and requiring more regularity with respect to the assumptions of Proposition \ref{prop:CLSolBVP}, we first prove a regularity result on the velocity field of the fluid. Recall that we denote $\Lambda = (1-\Delta)^{1/2}$ and that the Sobolev norm $\|\cdot\|_{H^{k,\ell}(D_0)}$ controls at most $\ell$ vertical derivatives (when $\ell \leq k$). Introducing the operator $\partial_w^{\Sigma} = \frac{ \partial_w }{ 1+ \partial_w\sigma}$, we define the \emph{good unknowns}
\begin{align} 
\tilde{\varphi}_{(\bmj)} &:= \partial^{\bmj} \tilde{\varphi} - (\partial^{\bmj} \sigma ) \;  \partial_w^{\Sigma} \tilde{\varphi} , \; \; \bmj \in \mathbb{N}^2 \setminus \{\bmzero\} ,  \label{eq:IrrGUn} \\
\Phi_{(\bmj)} &:= \partial^{\bmj} \Phi - \underline{\tilde{\mathrm{w}}}_{I} \; \partial^{\bmj} \eta , \; \; \bmj \in \mathbb{N}^2 \setminus \{\bmzero\} , \; \; \underline{\tilde{\mathrm{w}}}_{I} := \tilde{\mathbb{U}}_{I}[\eta]\Phi |_{w=0}, \label{eq:TraceIrrGUn} \\
\tilde{\bmA}_{(\bmj)} &:= \partial^{\bmj}  \tilde{\bmA} - ( \partial^{\bmj} \sigma) \; \partial_w^{\Sigma} \tilde{\bmA} , \; \; \bmj \in \mathbb{N}^2 \setminus \{\bmzero\} . \label{eq:RotGUn} 
\end{align}

\begin{remark} \label{rem:GoodUnRem}

We mention that in (3.22) of \cite{castro2015well} the authors define for all $\bmj \in \mathbb{N}^2 \setminus \{ \bmzero \}$ the good unknowns $\Phi_{(\bmj)} := \partial^{\bmj} \Phi - \underline{\tilde{\mathrm{w}} } \; \partial^{\bmj} \eta$. On the other hand, we introduce the good unknowns \eqref{eq:IrrGUn}-\eqref{eq:RotGUn} in order to better highlight the roles of the rotational and irrotational parts of the velocity field.

\end{remark}

Moreover, we define
\begin{align*}
\mathfrak{P} &:= \frac{ |\bfD| }{ (1+|\bfD|)^{1/2} } ,
\end{align*}
which appears already in the irrotational case, see for example Lemma 2.34 in \cite{lannes2013water}.

\begin{theorem} \label{thm:mainCLHighReg}

Let $N > 3/2$, let $\eta \in C^{N}(\mathbb{R}^2)$ be such that \eqref{eq:StrConnected} holds true. Assume that the assumptions of Proposition \ref{prop:CLSolBVP} hold true, and that $\sigma$ is a regularizing diffeomorphism such that
\begin{align} \label{eq:CondDiffHigh}
\| \sigma \|_{C^N(D_0)} &\leq C \, \| \eta \|_{ C^N(\mathbb{R}^2) } 
\end{align}
holds true. Then there exists a unique solution $\tilde{\bmU} \in ( H^1(D_0) )^3$ which solves the boundary value problem \eqref{eq:BVPStraight}. If $0 \leq k < N-3/2$, $k \in \mathbb{N}$, $\Phi \in \dot{H}^{k+3/2}(\mathbb{R}^2)$, $\Lambda^k\tilde{\bmomega} \in ( L^2(D_0))^3$, then
\begin{align*}
\| \tilde{\bmU} \|_{(H^{k+1,1}(D_0))^3} &\leq C( h_0^{-1}, h, \|\eta\|_{C^{N}(\mathbb{R}^2)} ) \times \\
&\qquad \times \bigg[  \| \mathfrak{P}\Phi \|_{H^{1}(\mathbb{R}^2)}  + \sum_{ 1< |\bmj| \leq  k+1  }  \|  \mathfrak{P}\Phi_{(\bmj)} \|_{L^{2}(\mathbb{R}^2)} + \| \Lambda^k \tilde{\bmomega} \|_{ 2,b}   \bigg] .
\end{align*}

Moreover, if $N > 3$, $0 \leq \ell \leq k < N-3/2$  and $\tilde{\bmomega} \in ( H^{k,\ell}(D_0))^3$ , then the following higher-order estimate hold
\begin{align*}
\| \tilde{\bmU} \|_{(H^{k+1,\ell+1}(D_0))^3} &\leq C( h_0^{-1}, h, \|\eta\|_{C^{N}(\mathbb{R}^2)} ) \times \nonumber \\
&\qquad \times \bigg[  \| \mathfrak{P}\Phi \|_{H^{1}(\mathbb{R}^2)}  + \sum_{1 < |\bmj| \leq k+1 }  \|  \mathfrak{P}\Phi_{(\bmj)} \|_{L^{2}(\mathbb{R}^2)}   \nonumber \\
&\qquad \qquad + \| \tilde{\bmomega} \|_{ (H^{k,\ell}(D_0))^3}  + \| \Lambda^k( \tilde{\bmomega} \cdot \bme_3)|_{w=-h} \|_{H_0^{-1/2}} \bigg] .
\end{align*}
\end{theorem}

The above result is proved in Sec. \ref{sec:HighOrdEst}. From Theorem \ref{thm:mainCLHighReg} we can immediately derive the following higher order estimates on the generalized Dirichlet--Neumann operator defined in \eqref{eq:GDNO}. More precisely, we consider the straightened version of the generalized Dirichlet--Neumann operator, given by
\begin{align*}
\tilde{\mathcal{G}}_{gen}[\eta](\Phi,\tilde{\bmomega}) &= \tilde{\mathbb{U}}[\eta](\Phi,\tilde{\bmomega}) \cdot \bmN |_{w=0} .
\end{align*}

\begin{corollary} \label{cor:mainGDNOhigh}

Let $N > 3/2$, let $\eta \in C^{N}(\mathbb{R}^2)$ be such that \eqref{eq:StrConnected} holds true. Assume that $0 \leq k < N-3/2$, $k \in \mathbb{N}$, $\Phi \in \dot{H}^{k+3/2}(\mathbb{R}^2)$, $\tilde{\bmomega} \in ( H^{k,0}(D_0))^3$ and that $\sigma$ is a regularizing diffeomorphism such that \eqref{eq:CondDiffHigh} holds true. Then
\begin{equation*}
\tilde{\mathcal{G}}_{gen}[\eta](\Phi,\tilde{\bmomega})  \in  H^{k+1/2}(\mathbb{R}^2) ,
\end{equation*}
and the following higher-order estimate holds true,
\begin{align*} 
\left\| \tilde{\mathcal{G}}_{gen}[\eta](\Phi,\tilde{\bmomega}) \right\|_{ H^{k+1/2}(\mathbb{R}^2) }  &\leq C( h_0^{-1}, h, \|\eta\|_{C^{N}(\mathbb{R}^2)} ) \times \nonumber \\
&\qquad \times \bigg[ \| \mathfrak{P}\Phi \|_{H^{1}(\mathbb{R}^2)}  + \sum_{1 < |\bmj| \leq k+1 }  \|  \mathfrak{P}\Phi_{(\bmj)} \|_{L^{2}(\mathbb{R}^2)} \\
&\qquad \qquad + \| \tilde{\bmomega} \|_{ (H^{k,0}(D_0))^3}  + \| \Lambda^k( \tilde{\bmomega} \cdot \bme_3)|_{w=-h} \|_{H_0^{-1/2}} \bigg] . 
\end{align*}

\end{corollary}

Next, we give an analyticity result for the generalized Dirichlet--Neumann operator (see also Theorem \ref{thm:ExSolFBVP} and Corollary \ref{cor:GDNOAnal}).

\begin{theorem} \label{thm:mainGDNOAnal}

Let $s \geq 2$, and let $\tilde{\bmomega} \in (H^{s-2}(D_0))^3$, $\Phi \in \dot{H}^{s-1/2}(\mathbb{R}^2)$. Assume that the condition
\begin{align} \label{eq:CondAnalGDNO}
\| \sigma \|_{H^{s+1/2}(D_0)} &\leq C \, \| \eta \|_{ H^{s+1/2}(\mathbb{R}^2) } 
\end{align}
holds true. Then there exists an open neighbourhood $V$ of the origin in $H^{s+1/2}(\mathbb{R}^2)$ such that formulae
\begin{align*}
\tilde{\mathcal{G}}_{gen,I}[\eta]\Phi &= \tilde{\mathbb{U}}_{I}[\eta]\Phi \cdot \bmN |_{w=0} , \\
\tilde{\mathcal{G}}_{gen,II}[\eta]\tilde{\bmomega} &= \tilde{\mathbb{U}}_{II}[\eta]\tilde{\bmomega} \cdot \bmN |_{w=0} 
\end{align*}
define analytic maps 
\begin{align*}
\tilde{\mathcal{G}}_{gen,I}[\cdot] &: V \to L \left( \dot{H}^{s-1/2}(\mathbb{R}^2) , H^{s-3/2}(\mathbb{R}^2)  \right) , \\
\tilde{\mathcal{G}}_{gen,II}[\cdot] &: V \to L \left( ( H^{s-2}(D_0) )^3 , H^{s-3/2}(\mathbb{R}^2)  \right) .
\end{align*}

\end{theorem}

We point out that the technical condition \eqref{eq:CondAnalGDNO} is satisfied both by the trivial diffeomorphism \eqref{eq:TrivialDiffeo} and by the regularizing diffeomorphism \eqref{eq:straightReg}. We mention that in Sec. \ref{sec:GDNOhom} we also compute the differential of the rotational part of the generalized Dirichlet--Neumann operator, which is a result that could be interesting in itself (see Theorem 3.21 in \cite{lannes2013water} and Theorem 1.7 of \cite{groves2024analytical} for corresponding results for the classical Dirichlet--Neumann operator and for the generalized Dirichlet--Neumann operator for Beltrami flows, respectively). 

\begin{proposition} \label{prop:DiffRotGDNOmain}

Let $s \geq 3$, and let $\tilde{\bmomega} \in (H^{s-2}(D_0))^3$, $\Phi \in \dot{H}^{s-1/2}(\mathbb{R}^2)$. Assume that the condition \eqref{eq:CondAnalGDNO} holds true. Then there exists an open neighbourhood $V$ of the origin in $H^{s+1/2}(\mathbb{R}^2)$ such that the differential of the operator 
\begin{equation*}
\tilde{\mathcal{G}}_{gen,II}[\cdot] : V \to L \left( ( H^{s-2}(D_0) )^3 , H^{s-3/2}(\mathbb{R}^2)  \right)
\end{equation*}
is given by
\begin{align}
\mathrm{d}\tilde{\mathcal{G}}_{gen,II}[\eta](\delta\eta)\tilde{\bmomega} &= - \tilde{\mathcal{G}}_{gen,II}[\eta] \tilde{\bmgamma} - \nabla \cdot \left[ ( \underline{\tilde{\mathbb{K}}}_{II}[\eta]\tilde{\bmomega} - \underline{\tilde{\mathbb{W}}}_{II}[\eta]\tilde{\bmomega} \; \nabla\eta ) \, \delta\eta \right] , \label{eq:DiffGgenIImain} \\
\tilde{\bmgamma} &\coloneqq \tilde{\bmgamma}[\eta](\delta\eta)\tilde{\bmomega} \, = \, \partial_w^{\Sigma} \tilde{\bmomega} \, \mathrm{d}\sigma(\delta\eta) , \label{eq:bmgammaMain}
\end{align}
\begin{align*}
\underline{\tilde{\mathbb{K}}}_{II}[\eta]\tilde{\bmomega} &\coloneqq -\gradp \, \Delta^{-1} \, ( \tilde{\bmomega} \cdot \bmN |_{w=0} ) , \\
\underline{\tilde{\mathbb{W}}}_{II}[\eta]\tilde{\bmomega} &\coloneqq  \frac{1}{1+|\eta|^2} \, \left[ \underline{\tilde{\mathbb{K}}}_{II}[\eta]\tilde{\bmomega} \cdot \nabla\eta + \tilde{\mathcal{G}}_{gen,II}[\eta]\tilde{\bmomega}   \right]. \nonumber
\end{align*}

\end{proposition}

Next, we study the Taylor expansion of the generalized Dirichlet--Neumann operator \eqref{eq:GDNO} around $\eta=0$. Recall that from \eqref{eq:GDNO} and from Remark \ref{rem:EqPsi} we have 
\begin{align*}
\mathcal{G}_{gen}[\eta](\Phi,\bmomega) &=  \nabla \cdot \bmA_{\parallel}^{\perp} + \underline{ \nabla_{\bmx,z}\varphi } \cdot \bmN .
\end{align*}

From \eqref{eq:GDNO} we recall the decomposition
\begin{align*}
\mathcal{G}_{gen}[\eta](\Phi,\bmomega) &= \mathcal{G}_{gen,I}[\eta]\Phi + \mathcal{G}_{gen,II}[\eta]\bmomega , \\
\mathcal{G}_{gen,I}[\eta] \Phi &:= \mathbb{U}_I[\eta]\Phi \cdot \bmN |_{z=\eta} \, = \, \underline{\nabla_{\bmx,z}\varphi} \cdot \bmN , \\
\mathcal{G}_{gen,II}[\eta] \bmomega &:= \mathbb{U}_{II}[\eta]\bmomega \cdot \bmN |_{z=\eta} \, = \, \nabla \cdot \bmA_{\parallel}^{\perp} ,
\end{align*}
together with its straightened version
\begin{align*}
\tilde{\mathcal{G}}_{gen}[\eta](\Phi,\tilde{\bmomega}) &= \tilde{\mathcal{G}}_{gen,I}[\eta]\Phi + \tilde{\mathcal{G}}_{gen,II}[\eta]\tilde{\bmomega} , \\
\tilde{\mathcal{G}}_{gen,I}[\eta] \Phi &:= \tilde{\mathbb{U}}_I[\eta]\Phi \cdot \bmN |_{w=0} \, = \, \nabla^{\Sigma}\tilde{\varphi} \cdot \bmN |_{w=0} , \\
\tilde{\mathcal{G}}_{gen,II}[\eta] \tilde{\bmomega} &:= \tilde{\mathbb{U}}_{II}[\eta]\tilde{\bmomega} \cdot \bmN |_{w=0} \, = \, \curl^{\Sigma}\tilde{\bmA} \cdot \bmN |_{w=0} ,
\end{align*}
and its associated expansion
\begin{align} \label{eq:mainTaylorGeta}
\tilde{\mathcal{G}}_{gen}[\eta] &= \sum_{j=0}^\infty \tilde{\mathcal{G}}_j[\eta] , \quad
\tilde{\mathcal{G}}_{gen,I}[\eta] = \sum_{j=0}^\infty \tilde{\mathcal{G}}_{j,I}[\eta] , \quad \tilde{\mathcal{G}}_{gen,II}[\eta] = \sum_{j=0}^\infty \tilde{\mathcal{G}}_{j,II}[\eta] .
\end{align}
where $\mathcal{G}_{j,I}[\eta]$ and $\mathcal{G}_{j,II}[\eta]$ are homogeneous of degree $j$ in $\eta$. It is also useful to consider the corresponding expansion for the quantity $\tilde{\bmgamma}$ defined in \eqref{eq:bmgammaMain} evaluated at $\delta\eta=\eta$, namely
\begin{align*}
\tilde{\bmgamma} |_{\delta\eta=\eta} &= \sum_{ j \geq 0 } \tilde{\bmgamma}_j ,
\end{align*}
where $\tilde{\bmgamma}_j$ is homogeneous of degree $j$ in $\eta$.

\begin{theorem} \label{thm:mainHomExpGDNO}
 
Under the assumptions of Proposition \ref{prop:DiffRotGDNOmain}, the term $\tilde{\mathcal{G}}_0$ in the homogeneous expansion \eqref{eq:mainTaylorGeta} of the operator $\tilde{\mathcal{G}}_{gen}[\eta]$ is given by 
\begin{align*}
\tilde{\mathcal{G}}_{0} &= \tilde{\mathcal{G}}_{0,I} + \tilde{\mathcal{G}}_{0,II} ,
\end{align*}
where $\tilde{\mathcal{G}}_{0,I}$ and $\tilde{\mathcal{G}}_{0,II}$ are given by 
\begin{align*}
\tilde{\mathcal{G}}_{0,I} \, \Phi &=  \mathscr{F}^{-1} \, \bigg[ |\bmxi| \, \tanh(h \, |\bmxi|)  \, (\mathscr{F} \Phi) \, \bigg] , \\
\tilde{\mathcal{G}}_{0,II} \,\tilde{\bmomega} &= \nabla \cdot \mathscr{F}^{-1} \, \bigg[ \, \int_{-h}^0 \frac{ \sinh(|\bmxi|(\zeta+h)) }{ |\bmxi| \, \cosh(h\, |\bmxi|)}  \,  ( \mathscr{F}\tilde{\bmomega} )_{\rmh}^{\perp}(\cdot ,\zeta) \, \mathrm{d}\zeta  \\
&\qquad \qquad - \frac{ \sinh(h \, |\bmxi|) }{ |\bmxi| \, \cosh(h\, |\bmxi|)} \, \mathscr{F} \left[ \gradp\Delta^{-1}(\tilde{\omega}_3|_{w=0}) \right]
\bigg] , 
\end{align*}
respectively. Moreover, if $\tilde{\bmgamma}_0 = \bmzero$, then for any $j \geq 1$ the term $\tilde{\mathcal{G}}_j[\eta]$ are given by 
\begin{align*}
\tilde{\mathcal{G}}_{j}[\eta] &= \tilde{\mathcal{G}}_{j,I}[\eta] + \tilde{\mathcal{G}}_{j,II}[\eta] ,
\end{align*}
where $\tilde{\mathcal{G}}_{j,I}$ and $\tilde{\mathcal{G}}_{j,II}$ are given by \eqref{eq:GjI}  and \eqref{eq:GjII}, respectively.
\end{theorem}

The above result is proved in Sec. \ref{sec:GDNOhom}. As mentioned in Sec. 1.3 of \cite{groves2024analytical}, individual terms appearing in recursion formulae for $\mathcal{G}_{j}[\eta]$ involve more and more derivatives of $\eta$ as $j$ increases; the overall validity of the formulae arises from cancellations between the terms (for the irrotational case see Sec. 2.2 of \cite{nicholls2001new}).

We just mention that the condition $\tilde{\bmgamma}_0 = \bmzero$ depends both on the vorticity and on the choice of the straightening diffeomorphism in \eqref{eq:straight}: it is trivially true in the irrotational case, but it is true also in more general cases (for example if we choose $\sigma$ as the trivial diffeomorphism \eqref{eq:TrivialDiffeo}, and the vorticity $\tilde{\bmomega}$ in a open neighbourhood of the origin in $( H^{s-2}(D_0) )^3$; see also Sec. 2.4 of \cite{groves2024analytical} for a similar approach for Beltrami flows).

\subsection{Paralinearization of the generalized Dirichlet--Neumann operator} \label{subsec:paralinGDNO}

In this paper we also prove a paralinearization result for the rotational part of the generalized Dirichlet--Neumann operator. We refer to Appendix \ref{sec:Paradiff}, and in particular to \eqref{eq:ParadiffOp}, for the definition of paradifferential operators.

By following the approach of Sec.3 of \cite{alazard2011water}, we introduce the following localizing transform: we choose $\delta >0$ such that the fluid domain $D_\eta$ contains the strip
\begin{align} \label{eq:strip}
\Omega_{\delta} &:= \{ (\bmx,z) \in \mathbb{R}^2 \times \mathbb{R}| \eta(\bmx)-\delta h \leq z < \eta(\bmx) \} .
\end{align}

\begin{remark} \label{rem:locTrans}

We notice that if \eqref{eq:StrConnected} holds true, then we can take any $\delta \in (0, h_0/h)$ in \eqref{eq:strip}: indeed, $z \geq \eta(\bmx)-\delta h > - h$ implies $\eta(\bmx)>-h(1-\delta)$.

\end{remark}

We define $\hat\Sigma:D_0 \to \Omega_{\delta}$,
\begin{align*} 
	\hat\Sigma: (\bmx,w) \mapsto (\bmx,\varrho(\bmx,w)), &\; \; \varrho(\bmx,w):=\delta w + \eta(\bmx),
\end{align*}
and write $\hat\bmA(\bmx,w) := \bmA \left( \bmx, \varrho(\bmx,w) \right)$, where $\bmA$ is the solution of the boundary value problem \eqref{eq:systA} restricted to $\Omega_{\delta}$. We denote by $\grad_{\bmx,z}=(\pd_x,\pd_y,\pd_z)^T$, and
\begin{align*}
	\grad_{\bmx,w}^\varrho := \left( J_{\hat\Sigma}^{-1} \right)^T \grad_{\bmx,z}, &\; \; \left( J_{\hat\Sigma}^{-1} \right)^T :=
	\begin{pmatrix}
		\mathbb{I}_2 & - \frac{\grad \eta}{\delta} \\
		0 & \frac{1}{\delta} \\
	\end{pmatrix}, \\
\pd_i^\varrho := \pd_i - \frac{ \pd_i\eta }{\delta} \pd_w, &\; \; i=x,y,\quad \pd_w^\varrho := \frac{ \pd_w }{\delta} ,\\
\curl^\varrho \hat{f}(\bmx,w) :=(\curl \, f)\circ \hat{\Sigma}(\bmx,w), &\; \; \Div^\varrho \hat{f}(\bmx,w) := (\Div \, f) \circ \hat{\Sigma}(\bmx,w).
\end{align*}

The flattened version of the operators $\curl$ and $\Delta$ applied respectively to $\hat\bmF(x,y,w)=\bmF(x,y,z)$ and to $\tilde{f}(x,y,w)=f(x,y,z)$ are given by
	\begin{align*}
\curl^\varrho \hat\bmF &= \curl \hat\bmF - \frac{1}{\delta} 
\begin{pmatrix}
\pd_y\eta \; \pd_w\hat{F}_3 \\
-\pd_x\eta \; \pd_w\hat{F}_3 \\
\pd_x\eta \; \pd_w\hat{F}_2 - \pd_y\eta \; \pd_w\hat{F}_1
\end{pmatrix}
 - \left( \frac{1}{\delta} -1 \right) 
\begin{pmatrix}
\pd_w\hat{F}_2 \\
- \pd_w\hat{F}_1 \\
0
\end{pmatrix}
 , \\
-\Delta^\varrho \hat{f} &= -\Delta \hat{f} + \frac{2}{\delta} \left( \pd_x\eta \; \pd_{xw}^2\hat{f} + \pd_y\eta \; \pd_{yw}^2\hat{f} \right) - \left( \frac{1}{\delta^2} - 1\right) \pd_w^2\hat{f} \\
&\qquad + \frac{1}{\delta} \, (\Delta\eta) \, \pd_w\hat{f} - \frac{1}{\delta^2} \, |\nabla\eta|^2 \, \pd_w^2\hat{f} . 
	\end{align*}

\begin{theorem} \label{thm:mainParalinGDNO}

Let $n \geq 4$ be such that $n \notin \mathbb{N}$, $\varepsilon >0$ and assume that $\eta \in H^n(\mathbb{R}^2)$ is such that \eqref{eq:StrConnected} holds true. If 
\begin{align} 
\partial_w^k \hat{\bmA} &\in ( C^0([-h,0];H^{n-k}(\mathbb{R}^2)) )^3 , \;\; k=0,1, \quad \hat{\bmomega} \in ( C^0([-h,0];H^{n-2}(\mathbb{R}^2)) )^3 ,  \label{eq:mainAssParalinGDNOp}
\end{align}
then the operator 
\begin{align*}
\hat{\mathcal{G}}_{gen,II}[\eta]\hat{\bmomega} &\coloneqq \curl^{\varrho}\hat{\bmA} \cdot \bmN |_{w=0} ,
\end{align*}
is given by
\begin{align*}
\hat{\mathcal{G}}_{gen,II}[\eta]\hat{\bmomega} &=  T_{\lambda_{II}} (  \hat{\bmA} - T_{\partial_w^{\varrho} \hat{\bmA}} \eta ) + \frac{1}{\delta} T_{1+|\nabla\eta|^2} \, T_{ \gradp\eta } \cdot \hat{\bmf}(\eta,\hat{\bmomega}) \nonumber \\ 
&\qquad - T_{\nabla\eta} \cdot ( \underline{ \hat{\mathbb{V}} }_{II}[\eta]\hat{\bmomega} ) - T_{ \underline{ \hat{\mathbb{V}} }_{II}[\eta]\hat{\bmomega} } \cdot \nabla\eta - T_{ |\nabla\eta|^2 } \, ( \underline{ \hat{\mathbb{W}} }_{II}[\eta]\hat{\bmomega} )  \nonumber \\
&\qquad + T_{1+|\nabla\eta|^2} \left( T_{\partial_w^{\varrho} \hat{A}_{2x}}\eta -  T_{ \partial_w^{\varrho} \hat{A}_{1y}}\eta -   T_{\eta_x} T_{(\partial_w^{\varrho})^2 \hat{A}_{2}}\eta + T_{\eta_y} T_{(\partial_w^{\varrho})^2 \hat{A}_{1}}\eta \right)  \\
&\qquad + \mathcal{R}_{II}, 
\end{align*}
where the right-hand side is evaluated at $w=0$; moreover, $\lambda_{II} \in M_{1,3}( \Sigma^1_{n-2}(\mathbb{R}^2) )$ is given by \eqref{eq:lambdaII}, $\hat{\bmf}(\eta,\hat{\bmomega}) \in ( H^{n-1}(\mathbb{R}^2) )^2$,
\begin{align*}
\hat{\underline{\mathbb{V}}}_{II}[\eta]\hat{\bmomega} &= (\curl^{\varrho} \hat{\bmA})_{\rmh}|_{w=0} \in ( H^{n-1}(\mathbb{R}^2) )^2 , \\
\hat{\underline{\mathbb{W}}}_{II}[\eta]\hat{\bmomega} &= (\curl^{\varrho} \hat{\bmA})_3|_{w=0} \in  H^{n-1}(\mathbb{R}^2) ,
\end{align*}
and $\mathcal{R}_{II}(\eta,\hat{\bmomega}) \in H^{2n-3-\varepsilon}(\mathbb{R}^2)$.
\end{theorem}

The above result is proved in Sec. \ref{sec:GDNOpara}.

\section{The straightened boundary value problem} \label{sec:strBVP}

Recalling the straightening diffeomorphism \eqref{eq:straight}, we want to study the straightened version of the boundary value problem \eqref{eq:BVP}. 

For simplicity, in the following discussion we assume that $\eta$, the function $\sigma$ in \eqref{eq:straight}, the function $f: D_{\eta} \to \mathbb{R}$ and the vector field $\bmF: D_{\eta} \to \mathbb{R}^3$ are sufficiently smooth (either in Sobolev spaces or in H\"older spaces), so that the quantities we introduce below are well-defined. We define $\tilde{\bmF}(x,y,w) \coloneqq \bmF \circ \Sigma(x,y,w)$ and $\tilde{f}(x,y,w)=f \circ \Sigma(x,y,w)$; moreover, we write
\begin{align*}
\partial_i^\Sigma \tilde{\bmF} &\coloneqq ( \partial_i \bmF ) \circ \Sigma, \;\; i=t,x,y,w ,
\end{align*}
and similarly for $\partial_i^\Sigma \tilde{f}$. Notice that from \eqref{eq:straight} we have that
\begin{align*}
\partial_j^\Sigma = \partial_j - \frac{ \partial_j \sigma }{ 1+\partial_w \sigma} \partial_w , \; \; j =x,y,t, &\; \; \partial_w^\Sigma = \frac{ \partial_w }{ 1+\partial_w \sigma } .
\end{align*}

\begin{remark} \label{rem:flatOp}

Similarly, we can define
\begin{align*}
\curl^{\Sigma} \tilde{\bmF} = \nabla^\Sigma \times \tilde{\bmF} &\coloneqq (\curl \bmF) \circ \Sigma , \; \; \Div^{\Sigma} \tilde{\bmF} = \nabla^\Sigma \cdot \tilde{\bmF} \coloneqq (\Div \bmF) \circ \Sigma , \\
\Delta^\Sigma \tilde{f} &\coloneqq (\Delta f) \circ \Sigma .
\end{align*}

The flattened version of the operators $\curl$ and $\Delta$ applied to $\tilde{\bmF}$ and to $\tilde{f}$ are given by
\begin{align*}
\Div^\Sigma \tilde{\bmF} &= \partial_x^\Sigma \tilde{F}_1 + \partial_y^\Sigma \tilde{F}_2 + \partial_w^\Sigma \tilde{F}_3 \\
&= \Div \tilde{\bmF} - \frac{1}{1+\partial_w \sigma} \left( \partial_x \sigma \; \partial_w \tilde{F}_1 + \partial_y \sigma \; \partial_w \tilde{F}_2 \right) - \frac{ \partial_w \sigma }{ 1+\partial_w \sigma}  \; \partial_w \tilde{F}_3 , 
\end{align*}
\begin{align*}
\curl^\Sigma \tilde{\bmF} &= \left( \partial_y^\Sigma \tilde{F}_3 - \partial_w^\Sigma\tilde{F}_2, -\partial_x^\Sigma \tilde{F}_3 + \partial_w^\Sigma\tilde{F}_1, \partial_x^\Sigma \tilde{F}_2 - \partial_y^\Sigma\tilde{F}_1  \right)^T \\
&= \curl \tilde{\bmF} -  \frac{ \partial_w \sigma }{ 1+\partial_w \sigma}  \; \left( -\partial_w \tilde{F}_2 , \partial_w \tilde{F}_1,0 \right)^T \\
&\qquad - \frac{1}{1+\partial_w \sigma} \left( \partial_y \sigma \; \partial_w \tilde{F}_3 , - \partial_x \sigma \; \partial_w \tilde{F}_3 , \partial_x \sigma \; \partial_w \tilde{F}_2 - \partial_y \sigma \; \partial_w \tilde{F}_1 \right)^T ,\\
\Delta^\Sigma \tilde{f} &= (\partial_x^\Sigma)^2 \tilde{f} + (\partial_y^\Sigma)^2 \tilde{f} + (\partial_w^\Sigma)^2 \tilde{f} = \left[ \mathtt{a} \, \partial_w^2 + \Delta + \bmb \cdot \nabla \partial_w - \mathtt{c} \, \partial_w \right] \tilde{f},
\end{align*}
where
\begin{align} 
\mathtt{a} &\coloneqq \frac{ 1+ |\nabla_{\bmx} \sigma|^2 }{ (1+\partial_w \sigma)^2 } , \; \; \bmb \coloneqq -2 \frac{ \nabla_{\bmx} \sigma }{ 1+\partial_w \sigma } , \nonumber \\
\mathtt{c} &\coloneqq \frac{1}{ 1+\partial_w \sigma } \left[  \Delta_{\bmx} \sigma  + \bmb  \cdot \nabla_{\bmx} \; \partial_w \sigma  +  \mathtt{a} \,  \partial_w^2 \sigma \right] . \label{eq:CoeffFlatLap}
\end{align}

If we denote by $J_\Sigma$ the Jacobian matrix of $\Sigma$, we denote by $P(\Sigma)$ the symmetric matrix 
\begin{align}
P(\Sigma) &\coloneqq |\det(J_\Sigma)| \, J_{\Sigma}^{-1}  J_\Sigma^{-T}  = (1+\partial_w \sigma) \,
\begin{pmatrix}
\mathbb{I}_2 & \frac{- \nabla_{\bmx} \sigma }{ 1+ \partial_w \sigma} \\
& \\
\frac{- ( \nabla_{\bmx} \sigma )^T }{ 1+\partial_w \sigma} & \frac{ 1+ |\nabla_{\bmx} \sigma|^2 }{ (1+\partial_w \sigma)^2 } 
\end{pmatrix}
, \label{eq:PSigma} 
\end{align}
so that 
\begin{align*}
\Delta^{\Sigma} \tilde{f} &= (1 + \partial_w\sigma)^{-1} \, \nabla_{\bmx,w} \cdot P(\Sigma)\nabla_{\bmx,w} \tilde{f} ,
\end{align*}
see Lemma 2.5 in \cite{lannes2005well}.

\end{remark}

From the above definitions and from Proposition \ref{prop:CLSolBVP} it follows that

\begin{proposition} \label{prop:BVPStraight}

Let $k>\frac{5}{2}$, let $\eta \in H^{k}(\mathbb{R}^2)$ and let $\sigma$ be a regularizing diffeomorphism be such that \eqref{eq:StrConnected} holds true, and let us denote by $D_\eta$ the domain as in \eqref{eq:TimeIndDom}. Let also $\bmomega \in H_b(\Div_0,D_\eta) $ and $\Phi \in \dot{H}^{3/2}(\mathbb{R}^2)$.

Then there exists a unique $\bmU \in ( H^{1}( D_\eta ) )^3$ which solves the boundary value problem \eqref{eq:BVP}. Moreover, $\tilde{\bmU} = \bmU \circ \Sigma \in ( H^{1}( D_0 ) )^3$ solves
\begin{equation} \label{eq:BVPStraight}
\begin{cases}
\curl^{\Sigma} \tilde{\bmU} = \tilde{\bmomega}, & \text{in} \; \; D_0, \\
\Div^{\Sigma} \tilde{\bmU} = 0 & \text{in} \; \; D_0,  \\
\tilde{\bmU} \cdot \bme_3 = 0, & \text{at} \; \; w = -h ,  \\
\tilde{\bmU}_{\parallel} = \nabla\Phi - \gradp \Delta^{-1} (\tilde{\bmomega} \cdot \bmN) , & \text{at} \; \; w=0. 
\end{cases}
\end{equation}
Conversely, if $\tilde{\bmU} \in ( H^{1}( D_0 ) )^3$ is the unique solution to the boundary value problem \eqref{eq:BVPStraight}, then $\bmU = \tilde{\bmU} \circ \Sigma^{-1} \in ( H^{1}( D_{\eta} ) )^3$ is the unique solution to the boundary value problem \eqref{eq:BVP}.

\end{proposition}

From Proposition \ref{prop:CLSolBVP} one obtains the following estimate on $\tilde{\bmU}$,
\begin{align} 
& \| \tilde{\bmU} \|_{( L^{2}(D_0) )^3} + \| \nabla^{\Sigma}\tilde{\bmU} \|_{( L^{2}(D_0) )^{3 \times 3}} \nonumber \\
&\leq C( h_0^{-1}, h, \|\eta\|_{H^{k}(\mathbb{R}^2)} ) \; \left(  \| \tilde{\bmomega} \|_{2,b} + \| \nabla\Phi \|_{( H^{1/2}(\mathbb{R}^2) )^2}  \right) , \label{eq:EstH1str}
\end{align}
where we point out that $\Sigma$ is a regularizing diffeomorphism (see Appendix \ref{sec:proofDivCurl}; see also Sec. 3.1.1 of \cite{alazard2014cauchy} and Sec.4.2 of \cite{wang2021local}).

\section[Corollaries of Proposition $\ref{prop:CLSolBVP}$]{Corollaries of Proposition \ref{prop:CLSolBVP} } \label{sec:corProp}

We now discuss some consequences of Proposition \ref{prop:CLSolBVP}. Observe that Eq. \eqref{eq:bfuDecomp} provides a decomposition of the velocity field into a rotational part and an irrotational part. First we mention a corollary of the above Proposition \ref{prop:CLSolBVP}, which provides an inverse of the curl operator on the space of divergence-free vector fields (see also Corollary 3.9  in \cite{castro2015well}).

\begin{corollary} \label{cor:InvCurl}

Let $k>\frac{5}{2}$, let $\eta \in H^{k}(\mathbb{R}^2)$ be such that \eqref{eq:StrConnected} holds true, and let us denote by $D_\eta$ the domain as in \eqref{eq:TimeIndDom}.  Let also $\bmC \in H(\Div_0,D_\eta)$ be such that  $\bmC \cdot \bmN |_{z=-h} = 0$. 

Then there exists a unique $\bmB \in ( H^{1}(D_\eta) )^3$ which solves the boundary value problem
\begin{align*}
\curl \bmB &= \bmC , \; \; \text{in} \; \; D_\eta , \\
\Div \bmB &= 0 , \; \; \text{in} \; \; D_\eta , \\
\bmB &= \bmzero ,\; \; \text{at} \; \; z=-h , \\
\bmB_{\parallel} &= -\gradp \Delta^{-1} (\bmC \cdot \bmN) ,\; \; \text{at} \; \; z=\eta , 
\end{align*}
and we denote this solution by $\bmB = \curl^{-1} \bmC$.
\end{corollary}

From Proposition \ref{prop:CLSolBVP} and Remark \ref{rem:ProductRem} we can immediately derive the following result on the generalized Dirichlet--Neumann operator defined in \eqref{eq:GDNO}.

\begin{corollary} \label{cor:GDNO}

Let $k>\frac{5}{2}$, let $\eta \in H^{k}(\mathbb{R}^2)$ be such that \eqref{eq:StrConnected} holds true. Then for any $\Phi \in \dot{H}^{3/2}(\mathbb{R}^2)$ and for any $\bmomega \in H_b(\Div_0,D_\eta) $ we have
\begin{equation*}
\mathcal{G}_{gen}[\eta](\Phi,\bmomega)  \in  H^{1/2}(\mathbb{R}^2) ,
\end{equation*}
and the following estimate holds true,
\begin{align} 
\left\| \mathcal{G}_{gen}[\eta](\Phi,\bmomega) \right\|_{ H^{1/2}(\mathbb{R}^2) }  &\leq C( h_0^{-1}, h, \|\eta\|_{H^{k}(\mathbb{R}^2)} ) \, \left(  \|\bmomega\|_{2,b} + \| \nabla\Phi \|_{H^{1/2}(\mathbb{R}^2)}  \right) . \label{eq:EstGDNO}
\end{align}

\end{corollary}

We also reformulate the boundary value problem \eqref{eq:systA} in an equivalent way; this allows us to obtain a solution of \eqref{eq:systA} by solving (the flattened version of) an equivalent problem, whose solution can be explicitly written in terms of a Green matrix (see Sec. \ref{sec:GDNOhom}). \\

\begin{corollary} \label{cor:AltStrong}

Let $k>\frac{5}{2}$, let $\eta \in H^{k}(\mathbb{R}^2)$ be such that \eqref{eq:StrConnected} holds true. Let $\bmomega \in H_b( \Div_0,D_{\eta})$ and $\Phi \in \dot{H}^{3/2}(\mathbb{R}^2)$, then the strong solutions of the boundary value problem \eqref{eq:systA}  and of the boundary value problem
\begin{equation} \label{eq:systAalt}
\begin{cases}
-\Delta \bmA = \bmomega , &\; \; \text{in} \; \; D_{\eta}, \\
A_{3z} = 0 , &\; \; \text{at} \; \; z=-h, \\
\bmA \times \bme_3 = \bmzero, &\; \; \text{at} \; \; z=-h, \\
\bmA \cdot \bmn = 0, &\; \; \text{at} \; \; z=\eta, \\
(\curl\bmA)_{\parallel} = - \gradp \, \Delta^{-1} (\bmomega \cdot \bmN) , &\; \; \text{at} \; \; z=\eta,
\end{cases}
\end{equation}
coincide. 
\end{corollary}

The above result can be proved with an argument similar to the proof of Proposition 4.8 of \cite{groves2020variational}.

\section{Regularity of the velocity field } \label{sec:HighOrdEst}

All results for the velocity field $\bmU$ have an analogue for the straightened field
\begin{align*}
\tilde{\mathbb{U}}[\eta](\Phi,\tilde{\bmomega}) &:= \mathbb{U}[\eta](\Phi,\bmomega) \circ \Sigma ,
\end{align*}
where we introduce the decomposition
\begin{align*}
\tilde{\mathbb{U}}[\eta](\Phi,\tilde{\bmomega}) &:= \tilde{\mathbb{U}}_{I}[\eta]\Phi + \tilde{\mathbb{U}}_{II}[\eta]\tilde{\bmomega}  , \\
\tilde{\mathbb{U}}_{I}[\eta]\Phi := \mathbb{U}_{I}[\eta]\Phi \circ \Sigma , &\;\; \tilde{\mathbb{U}}_{II}[\eta]\tilde{\bmomega} := \mathbb{U}_{II}[\eta]\bmomega \circ \Sigma .
\end{align*}
Similarly, we define the straightened version of the generalized Dirichlet--Neumann operator defined in \eqref{eq:GDNO},
\begin{align*}
\tilde{\mathcal{G}}_{gen}[\eta](\Phi,\tilde{\bmomega}) &:= \tilde{\mathcal{G}}_{gen,I}[\eta]\Phi + \tilde{\mathcal{G}}_{gen,II}[\eta]\tilde{\bmomega}, \\
\tilde{\mathcal{G}}_{gen,I}[\eta]\Phi := \tilde{\mathbb{U}}_{I}[\eta]\Phi \cdot \bmN |_{w=0} , &\;\; \tilde{\mathcal{G}}_{gen,II}[\eta]\tilde{\bmomega} := \tilde{\mathbb{U}}_{II}[\eta]\tilde{\bmomega} \cdot \bmN |_{w=0} .
\end{align*}

We have that
\begin{align} 
\int_{D_0} \nabla^{\Sigma}\tilde{\bmU} \cdot P(\Sigma)\nabla^{\Sigma}\tilde{\bmC} &= \int_{D_0} (1+\partial_w\sigma) \tilde{\bmomega} \cdot \curl^{\Sigma}\tilde{\bmC} + \int_{\mathbb{R}^2} \bmF \cdot \underline{\bmC} ,  \label{eq:IntPart} \\
\bmF &:= \bmN \cdot \underline{ \nabla_{\bmx,z}\bmU } + \bmN \times \underline{ \curl \bmU }  , \nonumber 
\end{align}
for all $\tilde{\bmC} \in ( H^1(D_0) )^3$, where we recall that $P(\Sigma) = (1+\partial_w \sigma) J_\Sigma^{-1} (J_\Sigma^{-1})^T$. 

Recalling that $\Lambda = (1-\Delta)^{1/2}$ and that the Sobolev norm $\|\cdot\|_{H^{k,\ell}(D_0)}$ controls at most $\ell$ vertical derivatives ($\ell \leq k$), we recall the definition of the \emph{good unknowns} \eqref{eq:IrrGUn}-\eqref{eq:RotGUn},
\begin{align*} 
\tilde{\varphi}_{(\bmj)} &= \partial^{\bmj} \tilde{\varphi} - (\partial^{\bmj} \sigma ) \;  \partial_w^{\Sigma} \tilde{\varphi} , \; \; \bmj \in \mathbb{N}^2 \setminus \{\bmzero\},  \\ 
\Phi_{(\bmj)} &= \partial^{\bmj} \Phi - \underline{\tilde{\mathrm{w}}}_{I} \; \partial^{\bmj} \eta , \; \; \bmj \in \mathbb{N}^2 \setminus \{\bmzero\}, \; \; \underline{\tilde{\mathrm{w}}}_{I} := \tilde{\mathbb{W}}_{I}[\eta]\Phi |_{w=0},  \\
\tilde{\bmA}_{(\bmj)} &= \partial^{\bmj} \tilde{\bmA} - ( \partial^{\bmj} \sigma) \; \partial_w^{\Sigma} \tilde{\bmA} , \; \; \bmj \in \mathbb{N}^2 \setminus \{\bmzero\} .
\end{align*}

Moreover, we define
\begin{align*}
\mathfrak{P} &:= \frac{ |\bfD| }{ (1+|\bfD|)^{1/2} } ;
\end{align*}
the operator $\mathfrak{P}$ appears already in the irrotational case, see for example Lemma 2.34 in \cite{lannes2013water}. We state the following result (see Proposition 3.12 in \cite{castro2015well}), which is a generalization of Proposition \ref{prop:CLSolBVP}. In the rest of the section, we use the convention that a summation over an empty set is equal to zero.

\begin{theorem} \label{thm:CLHighReg}

Let $N > 3/2$, let $\eta \in C^{N}(\mathbb{R}^2)$ be such that \eqref{eq:StrConnected} holds true. Assume that the assumptions of Proposition \ref{prop:CLSolBVP} hold true, and that $\sigma$ is a regularizing diffeomorphism such that \eqref{eq:CondDiffHigh} holds true. Then there exists a unique solution $\tilde{\bmU} \in ( H^1(D_0) )^3$ which solves the boundary value problem \eqref{eq:BVPStraight}. If $0 \leq k < N-3/2$, $k \in \mathbb{N}$, $\Phi \in \dot{H}^{k+3/2}(\mathbb{R}^2)$, $\Lambda^k\tilde{\bmomega} \in ( L^2(D_0))^3$, then
\begin{align}
\| \tilde{\bmU} \|_{(H^{k+1,1}(D_0))^3} &\leq C( h_0^{-1}, h, \|\eta\|_{C^{N}(\mathbb{R}^2)} ) \times \nonumber \\
&\qquad \times \, \bigg[  \| \mathfrak{P}\Phi \|_{H^{1}(\mathbb{R}^2)}  + \sum_{ 1 < |\bmj| \leq  k+1  }  \|  \mathfrak{P}\Phi_{(\bmj)} \|_{L^{2}(\mathbb{R}^2)}  + \| \Lambda^k \tilde{\bmomega} \|_{ 2,b}   \bigg] . \label{eq:HighOrdEst}
\end{align}

If, moreover, $N >3$, $0 \leq \ell \leq k < N-3/2$,  and $\tilde{\bmomega} \in ( H^{k,\ell}(D_0))^3$ , then the following higher-order estimate hold
\begin{align}
\| \tilde{\bmU} \|_{(H^{k+1,\ell+1}(D_0))^3} &\leq C( h_0^{-1}, h, \|\eta\|_{C^{N}(\mathbb{R}^2)} ) \times \nonumber \\
& \times \bigg[  \| \mathfrak{P}\Phi \|_{H^{1}(\mathbb{R}^2)}  + \sum_{1 < |\bmj| \leq  k+1 }  \|  \mathfrak{P}\Phi_{(\bmj)} \|_{L^{2}(\mathbb{R}^2)}   \nonumber \\
&\qquad  \qquad + \| \tilde{\bmomega} \|_{ (H^{k,\ell}(D_0))^3}  + \| \Lambda^k( \tilde{\bmomega} \cdot \bme_3)|_{w=-h} \|_{H_0^{-1/2}(\mathbb{R}^2) } \bigg] .
 \label{eq:HighOrdkEst}
\end{align}
\end{theorem}

\begin{proof}

We start by taking $\tilde{\bmC} = \partial^{2 \bmj}\tilde{\bmU}$ in \eqref{eq:IntPart}, with $|\bmj|=k < N-3/2$; this is a slight abuse of notation, as mentioned in the proof of Proposition 3.12 in \cite{castro2015well}, because one should actually take $\tilde{\bmC} = \chi(\delta \, w |\bfD|) \partial^{2\bmj}\tilde{\bmU}$ where $\delta>0$ and $\chi \in C^\infty_c(\mathbb{R}^2)$, $\chi \equiv 1$ in a neighbourhood of the origin, and then deduce that the estimates are uniform as $\delta \to 0$, see for example Lemma 2.34 in \cite{lannes2013water}. We obtain
\begin{align*}
\int_{D_0} \nabla^{\Sigma}\partial^{\bmj}\tilde{\bmU} \cdot P(\Sigma)\nabla^{\Sigma}\partial^{\bmj}\tilde{\bmU} &= I_1 + I_2 + I_3, \\
I_1 &:= \int_{D_0} \nabla^{\Sigma}\partial^{\bmj}\tilde{\bmU} \cdot [ \partial^{\bmj} , P(\Sigma) ] \nabla^{\Sigma}\tilde{\bmU} , \\
I_2 &:= \int_{D_0} \Lambda^k \tilde{\bmomega} \cdot \Lambda^{-k} \left( (1+\partial_w \sigma) \curl^{\Sigma} \partial^{2\bmj} \tilde{\bmU}  \right) , \\
I_3 &:= \int_{\mathbb{R}^2} \partial^{\bmj} \tilde{\bmF} \cdot \partial^{\bmj} \tilde{\bmU}|_{w=0} .
\end{align*}
By Lemma \ref{lem:Product} and by condition \eqref{eq:CondDiffHigh} we have that for any $0 < \varepsilon \ll 1$ the following estimate holds true,
\begin{align*}
\left\| \partial^{\bmj} \, \nabla_{\bmx,w} \tilde{\bmU} \right\|_{( L^2(D_0) )^3}^2 &\leq C \left(h_0^{-1}, \|\eta\|_{C^{1+\varepsilon}(\mathbb{R}^2)} \right) \, ( |I_1| + |I_2| + |I_3| ) .
\end{align*}

Regarding $I_1$, if we denote by $Q(\Sigma):=P(\Sigma) - \mathbb{I}_3$, we have again by Lemma \ref{lem:Product} and by \eqref{eq:CondDiffHigh} that for any $\varepsilon>0$ such that $N \geq k+1+\varepsilon$
\begin{align}
|I_1| &\leq \| \nabla^{\Sigma}\partial^{\bmj}\tilde{\bmU} \|_{ ( L^2(D_0) )^3} \, \|  [ \partial^{\bmj} , Q(\Sigma) ] \nabla^{\Sigma}\tilde{\bmU} \|_{ ( L^2(D_0) )^3} \nonumber \\
&\leq  C \, \| \nabla^{\Sigma}\partial^{\bmj}\tilde{\bmU} \|_{ ( L^2(D_0) )^3} \, \| Q(\Sigma) \|_{M_3( C^{k+\varepsilon}(D_0) ) }    \|  \Lambda^{k-1}\nabla^{\Sigma}\tilde{\bmU} \|_{ ( L^2(D_0) )^3} \nonumber \\
&\leq  C(h_0^{-1},h, \|\eta\|_{C^{N}(\mathbb{R}^2) }  ) \, \| \nabla^{\Sigma}\partial^{\bmj}\tilde{\bmU} \|_{ ( L^2(D_0) )^3} \,    \|  \Lambda^{k-1}\nabla^{\Sigma}\tilde{\bmU} \|_{ ( L^2(D_0) )^3} . \label{eq:I1est}
\end{align}

Next, we estimate $I_2$. By \eqref{eq:CondDiffHigh} and by Lemma \ref{lem:Product} we have that for any $\varepsilon>0$ such that $N \geq k+1+\varepsilon$
\begin{align} \label{eq:I2est}
| I_2 | &\leq C \left( h_0^{-1} , \| \eta \|_{ C^{N}(\mathbb{R}^2) }  \right) \, \| \Lambda^k \tilde{\bmomega} \|_{ ( L^2(D_0) )^3} \, \| \partial^{\bmj} \nabla^{\Sigma} \tilde{\bmU} \|_{( L^2(D_0) )^3 } .
\end{align}

Moreover, arguing as in the proof of Proposition 3.12 in \cite{castro2015well} we can deduce that
\begin{align}
|I_3| &\leq C( h_0^{-1}, \|\eta\|_{C^{N}(\mathbb{R}^2)} ) \, \|\Lambda^{k}\nabla^\Sigma \tilde{\bmU} \|_{( L^2(D_0) )^3 } \times \nonumber \\
&\qquad \times  \bigg[ \| \mathfrak{P} \left( \nabla \partial^{\bmj} \Phi - \underline{\tilde{\mathrm{w}}} \, \nabla \partial^{\bmj} \eta \right) \|_{L^2(\mathbb{R}^2)}  + \| \Lambda^k\tilde{\bmomega} \|_{2,b}  + \|\Lambda^{k-1}\nabla^\Sigma \tilde{\bmU} \|_{( L^2(D_0) )^3 } \bigg]  . \label{eq:I3est}
\end{align}
Combining \eqref{eq:I1est}, \eqref{eq:I2est} and \eqref{eq:I3est}, and using \eqref{eq:IntPart} and \eqref{eq:EstH1str}, we have that for any $\varepsilon >0$ such that $N \geq k+3/2+\varepsilon$ the bound \eqref{eq:HighOrdEst} holds true. 

Finally, observe that for all $N >3$ and $1 \leq m < N-1$ there exists $C>0$ such that for any $\varepsilon >0$ such that $m+\varepsilon \leq N-1$
\begin{align} \label{eq:CLproduct}
\| \Lambda^{m-j} \; \partial_z^j (f  \, g) \|_{L^2(D_0)} &\leq C \; \|f\|_{C^{m+\varepsilon}(D_0)} \; \|g\|_{H^{m,j}(D_0)} ,  \; \; \forall j \in \{ 0,\ldots,m \} , \\
&\; \; \forall f \in C^{N-1}(D_0), \; \forall g \in H^{m,j}(D_0), \nonumber
\end{align}
(this is a consequence of Lemma \ref{lem:Product}; see Lemma 3.15 of \cite{castro2015well}). Let $1 \leq \ell \leq k$: combining the identity 
\begin{align*}
\partial_w^\Sigma \tilde{\bmU} &= \frac{1}{ 1+|\nabla \eta|^2 } \left[ ( \bmN \cdot \partial_w^\Sigma \tilde{\bmU} ) \, \bmN + (\bmN \times \partial_w^\Sigma \tilde{\bmU} ) \times \bmN \right] 
\end{align*}
with the above product rule \eqref{eq:CLproduct} and with condition \eqref{eq:CondDiffHigh}  we obtain
\begin{align*}
\| \partial_w \tilde{\bmU} \|_{( H^{k,\ell}(D_0) )^3 } &\leq C( \|\eta\|_{C^N(\mathbb{R}^2)} ) \, \left[ \| \tilde{\bmU} \|_{( H^{k+1,\ell}(D_0) )^3 } + \| \tilde{\bmomega} \|_{( H^{k,\ell}(D_0) )^3 } \right] ,
\end{align*}
hence
\begin{align*}
\| \tilde{\bmU} \|_{( H^{k+1,\ell+1}(D_0) )^3 } &\leq C( \|\eta\|_{C^N(\mathbb{R}^2)} ) \, \left[ \| \tilde{\bmU} \|_{( H^{k+1,\ell}(D_0) )^3 } + \| \tilde{\bmomega} \|_{( H^{k,\ell}(D_0) )^3 } \right] ,
\end{align*}
and by finite induction on $\ell$ we get
\begin{align*}
\| \tilde{\bmU} \|_{( H^{k+1,\ell+1}(D_0) )^3 } &\leq C( \|\eta\|_{C^N(\mathbb{R}^2)} ) \, \left[ \| \tilde{\bmU} \|_{( H^{k+1,1}(D_0) )^3 } + \| \tilde{\bmomega} \|_{( H^{k,\ell}(D_0) )^3 } \right] ,
\end{align*}
so that we can deduce \eqref{eq:HighOrdkEst}.

\end{proof}

\begin{remark}\label{rem:ImprEstPhi}

We point out that the assumption $\Phi \in \dot{H}^{k+3/2}(\mathbb{R}^2)$ can be weakened, see Chap. 2 of \cite{lannes2013water} and Theorem 2.4 of \cite{castro2015well}.

\end{remark}

From Theorem \ref{thm:CLHighReg} and Remark \ref{rem:ProductRem} we can immediately derive the following higher order estimates on the generalized Dirichlet--Neumann operator defined in \eqref{eq:GDNO}.

\begin{corollary} \label{cor:GDNOhigh}

Let $N > 3/2$, let $\eta \in C^{N}(\mathbb{R}^2)$ be such that \eqref{eq:StrConnected} holds true. Assume that $0 \leq k < N-3/2$, $k \in \mathbb{N}$, $\Phi \in \dot{H}^{k+3/2}(\mathbb{R}^2)$, $\tilde{\bmomega} \in ( H^{k,0}(D_0))^3$ and that $\sigma$ is a regularizing diffeomorphism such that \eqref{eq:CondDiffHigh} holds true. Then
\begin{equation*}
\tilde{\mathcal{G}}_{gen}[\eta](\Phi,\tilde{\bmomega})  \in  H^{k+1/2}(\mathbb{R}^2) ,
\end{equation*}
and the following higher-order estimate holds true,
\begin{align} 
\left\| \tilde{\mathcal{G}}_{gen}[\eta](\Phi,\tilde{\bmomega}) \right\|_{ H^{k+1/2}(\mathbb{R}^2) }  &\leq C( h_0^{-1}, h, \|\eta\|_{C^{N}(\mathbb{R}^2)} ) \times \nonumber \\
& \times \bigg[ \| \mathfrak{P}\Phi \|_{H^{1}(\mathbb{R}^2)}  + \sum_{1 < |\bmj| \leq  k+1 }  \|  \mathfrak{P}\Phi_{(\bmj)} \|_{L^{2}(\mathbb{R}^2)}   \nonumber \\
&\qquad + \| \tilde{\bmomega} \|_{ (H^{k,0}(D_0))^3}  + \| \Lambda^k( \tilde{\bmomega} \cdot \bme_3)|_{w=-h} \|_{H_0^{-1/2}(\mathbb{R}^2)} \bigg] . \label{eq:EstGDNOhigh}
\end{align}

\end{corollary}

Finally, we mention that from Corollary 3.14 in \cite{castro2015well} one can derive higher order estimates for the generalized Dirichlet--Neumann operator as in Corollary \ref{cor:GDNOhigh}.

\section{Analyticity of the generalized Dirichlet--Neumann operator} \label{sec:GDNOhom}

In this section we prove an analyticity result for the generalized Dirichlet--Neumann operator; we also show a formula for the differential of its rotational part. As for the classical Dirichlet--Neumann operator for irrotational flows (see \cite{craig1993numerical} \cite{craig1994hamiltonian}) and for Beltrami flows (see \cite{groves2024analytical}), one can write the Taylor expansion of $\tilde{\mathcal{G}}_{gen}[\eta]$ around $\eta=0$,
\begin{align} \label{eq:TaylorGeta}
\tilde{\mathcal{G}}_{gen}[\eta] &= \sum_{j=0}^\infty \tilde{\mathcal{G}}_j[\eta] ,
\end{align}
where $\tilde{\mathcal{G}}_j[\eta]$ is homogeneous of degree $j$ in $\eta$, and for $j \geq 1$ the terms $\tilde{\mathcal{G}}_j[\eta]$ may be computed systematically by using a recursion formula. 

Indeed, from \eqref{eq:GDNO} and from Remark \ref{rem:EqPsi} we have 
\begin{align*}
\mathcal{G}_{gen}[\eta](\Phi,\bmomega) &=  \nabla \cdot \bmA_{\parallel}^{\perp} + \underline{ \nabla_{\bmx,z}\varphi } \cdot \bmN .
\end{align*}

We look for a scalar function $\varphi$ and for a vector potential $\bmA$ which solve the systems \eqref{eq:systA} and \eqref{eq:systphi} respectively. While the uniqueness of the solution of the boundary value problem \eqref{eq:systphi} is well-known (see Proposition 2.9 and Corollary 2.44 in \cite{lannes2013water}), we now exploit Corollary \ref{cor:AltStrong} in order to obtain a solution of \eqref{eq:systA} by solving (the flattened version of) an equivalent problem, whose solution can be explicitly written in terms of a Green matrix (see Sec.4(c) of \cite{groves2020variational} for a similar approach). We recall the decomposition
\begin{align*}
\tilde{\mathcal{G}}_{gen}[\eta](\Phi,\tilde{\bmomega}) &= \tilde{\mathcal{G}}_{gen,I}[\eta]\Phi + \tilde{\mathcal{G}}_{gen,II}[\eta]\tilde{\bmomega} , \\
\tilde{\mathcal{G}}_{gen,I}[\eta] \Phi &\coloneqq \tilde{\mathbb{U}}_I[\eta]\Phi \cdot \bmN |_{w=0} \, = \, \nabla^{\Sigma}\tilde{\varphi} \cdot \bmN |_{w=0} , \\
\tilde{\mathcal{G}}_{gen,II}[\eta] \tilde{\bmomega} &\coloneqq \tilde{\mathbb{U}}_{II}[\eta]\tilde{\bmomega} \cdot \bmN |_{w=0} \, = \, \curl^{\Sigma}\tilde{\bmA} \cdot \bmN |_{w=0} ,
\end{align*}
and the associated expansions
\begin{align*}
\tilde{\mathcal{G}}_{gen,I}[\eta] = \sum_{j=0}^\infty \tilde{\mathcal{G}}_{j,I}[\eta] , &\; \; \tilde{\mathcal{G}}_{gen,II}[\eta] = \sum_{j=0}^\infty \tilde{\mathcal{G}}_{j,II}[\eta] .
\end{align*}

We consider the flattened version of \eqref{eq:systA} and \eqref{eq:systphi}; by exploiting Corollary \ref{cor:AltStrong}, we obtain that $\tilde{\bmA}$ and $\tilde{\varphi}$ satisfy
\begin{align}
(|\bfD|^2 - \partial_w^2) \, \tilde{\bmA}(\bmx,w) &= \tilde{\bmomega}(\bmx,w)  + \bmR_{11}[\sigma,\tilde{\bmA}](\bmx,w) , \; \; \text{in} \; \; D_0, \label{eq:systAeq1} \\
\partial_w \tilde{A}_3(\cdot,-h) &= 0, \label{eq:systAeq2} \\
\tilde{A}_1(\cdot,-h) = \tilde{A}_2(\cdot,-h) &= 0, \label{eq:systAeq3} \\
\tilde{A}_3(\cdot,0) &= r_{21}[\sigma,\tilde{\bmA}](\bmx) , \label{eq:systAeq4} \\
\mathrm{i} \, \bfD^{\perp} \, \tilde{A}_3(\cdot,0) - \partial_w \, \tilde{\bmA}_{\rmh}(\cdot,0)^{\perp} &= -\gradp \Delta^{-1}  \left( \tilde{\omega}_3|_{w=0} \right) + \bmr_{31}[\sigma,\tilde{\bmA},\tilde{\bmomega}] , \label{eq:systAeq5}
\end{align} 
\begin{align}
(|\bfD|^2 - \partial_w^2) \tilde{\varphi}(\bmx,w) &= r_{12}[\sigma,\tilde{\varphi}](\bmx,w), \; \; \text{in} \; \; D_0, \label{eq:systphieq1} \\
\partial_w \tilde{\varphi}(\cdot,-h) &= r_{22}[\sigma,\tilde{\varphi}] , \label{eq:systphieq2} \\
\tilde{\varphi}(\cdot,0) &=  \Phi , \label{eq:systphieq3}
\end{align}
where
\begin{align*}
\bmR_{11}[\sigma,\tilde{\bmA}](\bmx,w) &\coloneqq   \Delta^{\Sigma} \tilde{\bmA} (\bmx,w) - \Delta_{\bmx,w} \tilde{\bmA} (\bmx,w)  , \\
r_{21}[\sigma,\tilde{\bmA}](\bmx) &\coloneqq \tilde{\bmA}_{\rmh}(\bmx,0) \cdot \nabla\eta(\bmx) , \\
\bmr_{31}[\sigma,\tilde{\bmA},\tilde{\bmomega}](\bmx) &\coloneqq  -  \left( ( \curl^{\Sigma}   \tilde{\bmA} ) |_{w=0} \right)_{\rmh} +  \left( \curl \tilde{\bmA}   \right)_{\rmh}(\bmx,0) \\
&\qquad  -  \left( ( \curl^{\Sigma} \tilde{\bmA} ) |_{w=0} \right)_{3} \, \nabla\eta  \\
&\qquad  + \left( \gradp \Delta^{-1} ( \tilde{\bmomega}(\bmx,0) \cdot \nabla\eta)  \right) , 
\end{align*}
and
\begin{align*}
r_{12}[\sigma,\tilde{\varphi}](\bmx,w) &\coloneqq   \Delta^{\Sigma} \tilde{\varphi}(\bmx,w) - \Delta_{\bmx,w}\tilde{\varphi}(\bmx,w) , \\
r_{22}[\sigma,\tilde{\varphi}](\bmx) &\coloneqq -( \partial^{\Sigma}_w \tilde{\varphi} ) (\bmx,-h) + \partial_w \, \tilde{\varphi}(\bmx,-h).
\end{align*}

In the rest of the section we denote by $\mathscr{F}$ the Fourier transform with respect to $\bmx$, and we denote by $\tilde{f}'$ the partial derivative $\partial_w \, \tilde{f}$.

\begin{lemma} \label{lem:SolFlatBVP}

Assume that $s \geq 2$, and let $\tilde{\bmomega} \in ( H^{s-2}(D_0) )^3$, $\Phi \in \dot{H}^{s-1/2}(\mathbb{R}^2)$. Then the boundary value problems
\begin{align*}
\begin{cases}
(|\bfD|^2 - \partial_w^2)  \, \tilde{\bmA} &= \tilde{\bmomega}  + \bmR_{1} , \; \; \text{in} \; \; D_0,  \\
\partial_w \, \tilde{A}_3(\cdot,-h) &= 0, \\
\tilde{A}_1(\cdot,-h) = \tilde{A}_2(\cdot,-h) &= 0,  \\
\tilde{A}_3(\cdot,0) &= r_{2} , \\
\mathrm{i} \, \bfD^{\perp} \, \tilde{A}_3(\cdot,0) - \partial_w \, \tilde{\bmA}_{\rmh}(\cdot,0)^{\perp} &= - \, \gradp \Delta^{-1} \left( \tilde{\omega}_3|_{w=0} \right) + \bmr_{3} ,  
\end{cases}
\end{align*}
 
\begin{align*}
\begin{cases}
(|\bfD|^2 - \partial_w^2) \tilde{\varphi} &= r_{4} , \; \; \text{in} \; \; D_0, \\
\partial_w \, \tilde{\varphi}(\cdot,-h) &= r_{5} ,  \\
\tilde{\varphi}(\cdot,0) &=  \Phi , 
\end{cases}
\end{align*}
admit unique solutions $\tilde{\bmA} \in (  H^s(D_0)  )^3$ and $\tilde{\varphi} \in  \dot{H}^s( D_0 )$ for any $\bmR_1 \in ( H^{s-2}(D_0) )^3$, $r_2 \in H^{s-1/2}(\mathbb{R}^2)$, $\bmr_3 \in ( H^{s-3/2}(\mathbb{R}^2) )^2$, $r_4 \in \dot{H}^{s-2}(D_0)$ and $r_5 \in \dot{H}^{s-3/2}(\mathbb{R}^2)$.
\end{lemma}

\begin{proof}

We rewrite the first boundary value problem for $\tilde{\bmF} \coloneqq \mathscr{F}\tilde{\bmA}$ in the following way,
\begin{align*}
\mathtt{M} \tilde{\bmF} &= \mathscr{F} \tilde{\bmomega} + \mathscr{F} \bmR_1 \; \; \text{in} \; \; D_0, \\
\mathtt{L}_{-h}\tilde{\bmF} &= \bmzero, \; \; \text{at} \; \; w = -h, \\ 
\mathtt{L}_{0}\tilde{\bmF} &= 
\begin{pmatrix}
\mathscr{F} \left[ \nabla\Delta^{-1}(\tilde{\omega}_3|_{w=0}) \right] + \mathscr{F} \bmr_3^{\perp} + \mathrm{i} \, \bmxi \, \mathscr{F} r_2 \\ \mathscr{F} r_2
\end{pmatrix}
,\; \; \text{at} \; \; w = 0,
\end{align*}
where
\begin{align*}
\mathtt{M} = (|\bmxi|^2-\partial_w^2) \, \mathbb{I}_3, \; \; \mathtt{L}_{-h} &= \text{diag}(1,1,\partial_w), \; \; \mathtt{L}_{0} = \text{diag}(\partial_w,\partial_w,1) .
\end{align*}
Let us define the Green's matrix
\begin{align*}
\mathtt{G}_{\tilde{\bmA}}(w,\zeta) &=
\begin{cases}
\mathtt{U}(w+h) \, \overline{\mathtt{C}^{-1}}^T \, \overline{\mathtt{W}(\zeta)}^T, \; \; \text{for} \; \; -h \leq w \leq \zeta \leq 0, \\
\mathtt{W}(w) \, \mathtt{C}^{-1} \, \overline{\mathtt{U}(\zeta+h)}^T, \; \; \text{for} \; \; -h \leq \zeta \leq w \leq 0, 
\end{cases}
\end{align*}
where $\mathtt{U}(\cdot+h)$ (resp. $\mathtt{W}$) is a $3 \times 3$ matrix whose columns $\mathtt{U}_j(\cdot+h)$ (resp. $\mathtt{W}_j$), $j=1,2,3$, are given by
\begin{align*}
\mathtt{U}_j(w) = \sinh( |\bmxi| w ) \, \bme_j, \; \; j=1,2, &\; \; \mathtt{U}_3(w) = \cosh( |\bmxi| w ) \, \bme_3, \\
\mathtt{W}_j(w) = \cosh( |\bmxi| w ) \, \bme_j, \; \; j=1,2, &\; \; \mathtt{W}_3(w) = \sinh( |\bmxi| w ) \, \bme_3, 
\end{align*}
and where
\begin{align*}
\mathtt{C} &= |\bmxi| \, \cosh(h \, |\bmxi| ) \,  \text{diag}(1,1,-1) .
\end{align*}

Hence, the solution $\tilde{\bmA}$ of the boundary value problem \eqref{eq:systAeq1}-\eqref{eq:systAeq5} is given by
\begin{align}
\tilde{\bmA}(\cdot,w) &= \mathscr{F}^{-1} \, \bigg[ \, \int_{-h}^0 \mathtt{G}_{\tilde{\bmA}}(w,\zeta) \, \left[ ( \mathscr{F}\tilde{\bmomega} )(\cdot ,\zeta) + \mathscr{F}\bmR_1(\cdot,\zeta) \right] \, \mathrm{d}\zeta \nonumber \\
&\qquad - \mathtt{G}_{\tilde{\bmA}}(w,0)
\begin{pmatrix}
\mathscr{F} \left[ \nabla\Delta^{-1}(\tilde{\omega}_3|_{w=0}) \right] + \mathscr{F}\bmr_3^{\perp} + \mathrm{i} \, \bmxi \, \mathscr{F}r_2   \\ 0
\end{pmatrix}
\nonumber \\
&\qquad - \mathtt{G}_{\tilde{\bmA},\zeta}(w,0)
\begin{pmatrix}
0 \\ 0 \\ \mathscr{F} \, r_2
\end{pmatrix}
\, \bigg] . \label{eq:BVPsolA}
\end{align}

Similarly, we write the Green's matrix for the boundary value problem for $\tilde{g} \coloneqq \mathscr{F}\tilde{\varphi}$,
\begin{align*}
\mathtt{G}_{\tilde{\varphi}}(w,\zeta) &=
\begin{cases}
- \frac{1}{ |\bmxi| \, \cosh(h \, |\bmxi|)} \, \cosh(|\bmxi|(w+h)) \, \sinh(|\bmxi| \zeta), \\
\qquad \qquad \text{for} \; \; -h \leq w \leq \zeta \leq 0, \\
\\
- \frac{1}{ |\bmxi| \, \cosh(h \, |\bmxi|)} \, \cosh( |\bmxi|(\zeta+h) ) \, \sinh(|\bmxi| w) , \\
\qquad \qquad \text{for} \; \; -h \leq \zeta \leq w \leq 0.
\end{cases}
\end{align*}

Hence the solution $\tilde{\varphi}$ of the boundary value problem \eqref{eq:systphieq1}-\eqref{eq:systphieq3} is given by
\begin{align}
\tilde{\varphi}(\cdot,w) &= \mathscr{F}^{-1} \, \bigg[ 
\int_{-h}^0 \mathtt{G}_{\tilde{\varphi}}(w,\zeta) \, \mathscr{F}r_4(\cdot,\zeta) \, \mathrm{d}\zeta - \mathtt{G}_{\tilde{\varphi}}(w,-h) \, \mathscr{F}r_5 - \mathtt{G}_{\tilde{\varphi},\zeta}(w,0) \, (\mathscr{F} \Phi) \, \bigg] . \nonumber \\
&\phantom{} \label{eq:BVPsolphi}
\end{align}

\end{proof}

Following an argument analogous to the proof of Theorem 4.10 in \cite{groves2020variational}, one can deduce from Lemma \ref{lem:SolFlatBVP} the following result.

\begin{theorem} \label{thm:ExSolFBVP}

Suppose that $s \geq 2$, and let $\tilde{\bmomega} \in (H^{s-2}(D_0))^3$, $\Phi \in \dot{H}^{s-1/2}(\mathbb{R}^2)$. There exists an open neighbourhood $W$ of the origin in $H^{s+1/2}(D_0)$ such that the boundary value problems \eqref{eq:systAeq1}-\eqref{eq:systAeq5} and \eqref{eq:systphieq1}-\eqref{eq:systphieq3} admit unique solutions $\tilde{\bmA}=\tilde{\bmA}(\sigma,\tilde{\bmomega}) \in (H^s(D_0))^3$ and $\tilde{\varphi}=\tilde{\varphi}(\sigma,\Phi) \in H^s(D_0)$ for each $\sigma \in W$.  Furthermore, $\tilde{\bmA}(\sigma,\tilde{\bmomega})$ depends analytically upon $\sigma$ and $\tilde{\bmomega}$ and linearly upon $\tilde{\bmomega}$, while $\tilde{\varphi}(\sigma,\Phi)$ depends analytically upon $\sigma$ and $\Phi$ and linearly upon $\Phi$.
\end{theorem}

From Theorem \ref{thm:ExSolFBVP} we can deduce the following analyticity result.

\begin{corollary} \label{cor:GDNOAnal}

Let $s \geq 2$, and let $\tilde{\bmomega} \in (H^{s-2}(D_0))^3$, $\Phi \in \dot{H}^{s-1/2}(\mathbb{R}^2)$. Assume that the condition \eqref{eq:CondAnalGDNO} holds true. Then there exists an open neighbourhood $V$ of the origin in $H^{s+1/2}(\mathbb{R}^2)$ such that formulae
\begin{align*}
\tilde{\mathcal{G}}_{gen,I}[\eta](\Phi,\tilde{\bmomega}) = \tilde{\mathbb{U}}_{I}[\eta]\Phi \cdot \bmN |_{w=0} , &\quad
\tilde{\mathcal{G}}_{gen,II}[\eta](\Phi,\tilde{\bmomega}) = \tilde{\mathbb{U}}_{II}[\eta]\tilde{\bmomega} \cdot \bmN |_{w=0} ,
\end{align*}
define analytic maps 
\begin{align*}
\tilde{\mathcal{G}}_{gen,I}[\cdot] &: V \to L \left( \dot{H}^{s-1/2}(\mathbb{R}^2) , H^{s-3/2}(\mathbb{R}^2)  \right) , \\
\tilde{\mathcal{G}}_{gen,II}[\cdot] &: V \to L \left( ( H^{s-2}(D_0) )^3 , H^{s-3/2}(\mathbb{R}^2)  \right) .
\end{align*}

\end{corollary}

We notice that the formula for $\tilde{\mathcal{G}}_{gen,II}[\eta] \tilde{\bmomega}$ can be rewritten as
\begin{align*}
\tilde{\mathcal{G}}_{gen,II}[\eta] \tilde{\bmomega} &=  \tilde{A}_{2x} + \eta_y  \tilde{A}_{3x} -  \tilde{A}_{1y} - \eta_x \tilde{A}_{3y}|_{w=0} .
\end{align*}

We compute the term $\tilde{\mathcal{G}}_{0}$ in the Taylor expansion \eqref{eq:TaylorGeta}, where we write
\begin{align*}
\tilde{\mathcal{G}}_{0}(\Phi,\tilde{\bmomega}) &= \tilde{\mathcal{G}}_{0,I}\Phi + \tilde{\mathcal{G}}_{0,II}\tilde{\bmomega} .
\end{align*}

From formulae \eqref{eq:BVPsolA} and \eqref{eq:BVPsolphi} we obtain that for $\eta=0$
\begin{align*}
\tilde{\bmA}_0(\cdot,w) &= \tilde{\bmA}(\cdot,w)|_{\bmR_{11}=\bmzero,r_{21}=0,\bmr_{31}=\bmzero} \\
&= \mathscr{F}^{-1} \, \bigg[ \, \int_{-h}^0 \mathtt{G}_{\tilde{\bmA}}(w,\zeta) \,  ( \mathscr{F}\tilde{\bmomega} )(\cdot ,\zeta) \, \mathrm{d}\zeta  - \mathtt{G}_{\tilde{\bmA}}(w,0)
\begin{pmatrix}
\mathscr{F} \left[ \nabla\Delta^{-1}(\tilde{\omega}_3|_{w=0}) \right]   \\ 0
\end{pmatrix}
\bigg] , \\
\tilde{\varphi}_0(\cdot,w) &= \tilde{\varphi}(\cdot,w)|_{r_{12}=r_{22}=0}  \\
&= \mathscr{F}^{-1} \, \bigg[ - \mathtt{G}_{\tilde{\varphi},\zeta}(w,0) \, (\mathscr{F} \Phi) \, \bigg] \, = \, \mathscr{F}^{-1} \, \bigg[ \frac{ \cosh(|\bmxi| (w+h)) }{ \cosh(h |\bmxi|) } \, (\mathscr{F} \Phi) \, \bigg] ,
\end{align*}
so that 
\begin{align}
\tilde{\mathcal{G}}_{0,I} \, \Phi =  \partial_{w}\tilde{\varphi}_0|_{w=0}  &=  \mathscr{F}^{-1} \, \bigg[ |\bmxi| \, \tanh(h \, |\bmxi|)  \, (\mathscr{F} \Phi) \, \bigg] ,  \label{eq:G0I}
\end{align}

\begin{align} 
\tilde{\mathcal{G}}_{0,II} \,\tilde{\bmomega} &=  \tilde{A}_{0,2x} -  \tilde{A}_{0,1y} |_{w=0} \nonumber \\
&= \nabla \cdot \mathscr{F}^{-1} \, \bigg[ \, \int_{-h}^0 \frac{ \sinh(|\bmxi|(\zeta+h)) }{ |\bmxi| \, \cosh(h\, |\bmxi|)}  \,  ( \mathscr{F}\tilde{\bmomega} )_{\rmh}^{\perp}(\cdot ,\zeta) \, \mathrm{d}\zeta  \nonumber \\
&\qquad \qquad \qquad - \frac{ \sinh(h \, |\bmxi|) }{ |\bmxi| \, \cosh(h\, |\bmxi|)} \, \mathscr{F} \left[ \gradp\Delta^{-1}(\tilde{\omega}_3|_{w=0}) \right]
\bigg] , \label{eq:G0II} 
\end{align}
since
\begin{align*}
\mathtt{G}_{\tilde{\bmA}}(0,\zeta) &= \frac{ \sinh(|\bmxi|(\zeta+h)) }{ |\bmxi| \, \cosh(h\, |\bmxi|)} \, \text{diag}( 1 , 1 , 0 ),  \; \; -h \leq \zeta \leq 0 .
\end{align*}

From \eqref{eq:BVPsolphi} we can deduce a formula for the expansion of the operator $\tilde{\mathcal{G}}_{I}[\eta]$ around $\eta=0$ (see Sec. 2.4 of \cite{groves2024analytical}, where such expansion is proved in the more general case of Beltrami flows; see also the argument by Craig-Sulem in \cite{craig1993numerical})

\begin{align}
\tilde{\mathcal{G}}_{1,I}[\eta]\Phi &= - \tilde{\mathcal{G}}_{0,I} \bigg(   \eta \, \tilde{\mathcal{G}}_{0,I}\Phi  \bigg)  - \nabla \cdot \left( \eta \, \nabla\Phi \right) , \nonumber \\
\tilde{\mathcal{G}}_{j,I}[\eta]\Phi &= \frac{1}{j} \, \bigg\{ - \tilde{\mathcal{G}}_{j-1,I}[\eta] \left( \eta \, \tilde{\mathcal{G}}_{0,I}\Phi  \right) \nonumber \\
&\qquad - \sum_{k=0}^{j-2} \tilde{\mathcal{G}}_{k,I}[\eta] \bigg(  u_{j-1-k}(\eta)(\Phi)\eta  \bigg) \nonumber \\
&\qquad + \nabla \cdot \big( \left(   u_{j-2}(\eta)(\Phi) \, \nabla\eta \right) \, \eta \big) \; \bigg\} , \;\; j \geq 2, \label{eq:GjI}
\end{align}
where
\begin{align*}
u_0(\Phi) &=  \tilde{\mathcal{G}}_{0,I}\Phi , \\
u_k(\eta)(\Phi) &=
\begin{cases}
(-1)^{k/2} |\nabla\eta|^{k} \, \tilde{\mathcal{G}}_{0,I}\Phi \\
\qquad + \sum_{i+2j=k} \left(  \bmK_{i-1}\Phi \cdot \nabla\eta + \tilde{\mathcal{G}}_{i,I}[\eta]\Phi \right) \, (-1)^j |\nabla\eta|^{2j},
& \text{if} \;\; k \in 2 \mathbb{N} , \\ 
\\
\sum_{i+2j=k} \left(  \bmK_{i-1}\Phi \cdot \nabla\eta + \tilde{\mathcal{G}}_{i,I}[\eta]\Phi \right) \, (-1)^j |\nabla\eta|^{2j},
& \text{if} \;\; k \not\in 2 \mathbb{N} , \\
\end{cases}
\end{align*}
for $k \geq 1$, with $\bmK_{0}\Phi = \nabla\Phi$ and $\bmK_{i}\Phi = \bmzero$ $\forall i \geq 1$.\\

Next, we obtain the terms $\tilde{\mathcal{G}}_{j,II}[\eta]$ in the Taylor expansion of the operator $\tilde{\mathcal{G}}_{II}[\eta]$ by deducing a formula for the differential of $\mathrm{d}\tilde{\mathcal{G}}_{gen,II}[\eta](\delta\eta)$, and then by evaluating it at $\delta\eta=\eta$.

\begin{proposition} \label{prop:DiffRotGDNO}

Let $s \geq 3$, and let $\tilde{\bmomega} \in (H^{s-2}(D_0))^3$, $\Phi \in \dot{H}^{s-1/2}(\mathbb{R}^2)$. Assume that the condition \eqref{eq:CondAnalGDNO} holds true. Then there exists an open neighbourhood $V$ of the origin in $H^{s+1/2}(\mathbb{R}^2)$ such that the differential of the operator 
\begin{equation*}
\tilde{\mathcal{G}}_{gen,II}[\cdot] : V \to L \left( ( H^{s-2}(D_0) )^3 , H^{s-3/2}(\mathbb{R}^2)  \right)
\end{equation*}
is given by
\begin{align}
\mathrm{d}\tilde{\mathcal{G}}_{gen,II}[\eta](\delta\eta)\tilde{\bmomega} &= - \tilde{\mathcal{G}}_{gen,II}[\eta] \tilde{\bmgamma} - \nabla \cdot \left[ ( \underline{\tilde{\mathbb{K}}}_{II}[\eta]\tilde{\bmomega} - \underline{\tilde{\mathbb{W}}}_{II}[\eta]\tilde{\bmomega} \; \nabla\eta ) \, \delta\eta \right] , \;\; \tilde{\bmgamma} = \partial_w^{\Sigma} \tilde{\bmomega} \, \mathrm{d}\sigma , \label{eq:DiffGgenII} 
\end{align}
\begin{align*}
\underline{\tilde{\mathbb{K}}}_{II}[\eta]\tilde{\bmomega} = -\gradp \, \Delta^{-1} \, ( \tilde{\bmomega} \cdot \bmN |_{w=0} ) , &\;\; \underline{\tilde{\mathbb{W}}}_{II}[\eta]\tilde{\bmomega} =  \frac{1}{1+|\eta|^2} \, \left[ \underline{\tilde{\mathbb{K}}}_{II}[\eta]\tilde{\bmomega} \cdot \nabla\eta + \tilde{\mathcal{G}}_{gen,II}[\eta]\tilde{\bmomega}   \right]. \nonumber
\end{align*}

\end{proposition}

\begin{remark} \label{rem:DiffGDNO}

Due to the decomposition \eqref{eq:bfuDecomp} of the velocity field, we have that 
\begin{equation*}
\underline{\tilde{\mathbb{K}}}_{II}[\eta]\tilde{\bmomega} = ( \curl^{\Sigma} \tilde{\bmA} )_{\parallel} , \;\; \underline{\tilde{\mathbb{W}}}_{II}[\eta]\tilde{\bmomega} = ( \curl^{\Sigma} \tilde{\bmA} )_{3} |_{w=0} .
\end{equation*}
\end{remark}

\begin{proof}[Proof of Proposition \ref{prop:DiffRotGDNO}]

Recall that $\tilde{\mathcal{G}}_{gen,II}[\eta]\tilde{\bmomega} = ( \curl^{\Sigma}\tilde{\bmA} ) \cdot \bmN |_{w=0}$, where $\tilde{\bmA}$ solves 
\begin{equation} \label{eq:strBVPA}
\begin{cases}
-\Delta^{\Sigma} \tilde{\bmA} = \tilde{\bmomega}, & \mathrm{in} \; \; D_0,\\
\partial_{w}^{\Sigma} \tilde{A}_3 = 0 , & \mathrm{at} \; \; w=-h, \\
\tilde{\bmA} \times \bme_3  = \bmzero, & \mathrm{at} \; \; w=-h, \\
\tilde{\bmA} \cdot \bmN = 0, & \mathrm{at} \; \; w=0, \\
(\curl^\Sigma \tilde{\bmA})_{\parallel} = -\gradp \Delta^{-1}(\tilde{\bmomega} \cdot \bmN) , & \mathrm{at} \; \; w=0. 
\end{cases}
\end{equation}

Let us recall the following identity (see (3.41) in \cite{castro2015well})
\begin{align} \label{eq:CLid}
\delta( \partial_j^{\Sigma} f ) &= \partial_j^{\Sigma}(\delta f - \delta\sigma \, \partial_w^{\Sigma} f ) + \delta\sigma \, \partial_w^{\Sigma}\partial_j^{\Sigma} f , \,\,\,\, j=x,y,w,t,
\end{align}
where $\delta$ is any linearization operator. If we introduce
\begin{align*}
\tilde{\bmC} &\coloneqq \mathrm{d}\tilde{\bmA} - \mathrm{d}\sigma \, \partial_w^{\Sigma}\tilde{\bmA},
\end{align*}
where $\mathrm{d}f = \mathrm{d}f[\eta](\delta\eta)$, and we apply \eqref{eq:CLid}, we have that
\begin{align*}
\mathrm{d}\tilde{\mathcal{G}}_{gen,II}[\eta](\delta\eta)\tilde{\bmomega} &= ( \curl^{\Sigma}\tilde{\bmC} ) \cdot \bmN |_{w=0}  - \nabla \cdot \left( (\curl^{\Sigma}\tilde{\bmA})_{\rmh}|_{w=0} \, \delta\eta \right) .
\end{align*}

Since $\partial_w^{\Sigma}\tilde{\bmA}|_{w=0}\cdot \bmN = - \nabla \cdot  \tilde{\bmA}_{\rmh}|_{w=0}$, we have
\begin{equation*}
\begin{cases}
-\Delta^{\Sigma} \tilde{\bmC} = - \partial_w^{\Sigma}\tilde{\bmomega} \, \mathrm{d}\sigma  , & \mathrm{in} \; \; D_0,\\
\partial_{w}^{\Sigma} \tilde{C}_3 = 0 , & \mathrm{at} \; \; w=-h, \\
\tilde{\bmC} \times \bme_3  = \bmzero, & \mathrm{at} \; \; w=-h, \\
\tilde{\bmC} \cdot \bmN = \nabla \cdot ( \tilde{\bmA}_{\rmh}|_{w=0} \, \delta\eta ), & \mathrm{at} \; \; w=0, \\
(\curl^\Sigma \tilde{\bmC})_{\parallel} = \gradp \Delta^{-1}\left( \tilde{\bmomega}_{\rmh} |_{w=0} \cdot \nabla\delta\eta \right)  \\
\qquad \qquad \qquad - \delta\eta \, \partial_w^{\Sigma} (\curl^{\Sigma}\tilde{\bmA})_{\rmh}|_{w=0} \\
\qquad \qquad \qquad  - \delta\eta \, \partial_w^{\Sigma}(\curl\tilde{\bmA})_3|_{w=0} \, \nabla\eta \\
\qquad \qquad \qquad - (\curl^{\Sigma} \tilde{\bmA})_3|_{w=0} \nabla\delta\eta , & \mathrm{at} \; \; w=0 .
\end{cases}
\end{equation*}

Now we rewrite the last boundary condition using that
\begin{align*}
\partial_w^{\Sigma}(\curl^{\Sigma}\tilde{\bmA})_{\rmh}|_{w=0} + \partial_w^{\Sigma}(\curl^{\Sigma}\tilde{\bmA})_{3}|_{w=0} \nabla\eta &= \nabla(\curl^{\Sigma}\tilde{\bmA})_3|_{w=0} + \tilde{\bmomega}_{\rmh}^{\perp}|_{w=0} , 
\end{align*}
hence
\begin{align*}
 (\curl^\Sigma \tilde{\bmC})_{\parallel} &=  - \gradp \Delta^{-1} \left[  - \partial_w^{\Sigma}\tilde{\bmomega}|_{w=0}\cdot \bmN \, \delta\eta  \right] \\
&\qquad + \nabla \, \left[ - \Delta^{-1} \, \nabla \cdot ( \tilde{\bmomega}_{\rmh}^{\perp}|_{w=0} \, \delta\eta  )   - (\curl^{\Sigma} \tilde{\bmA})_3|_{w=0} \, \delta\eta \right]
\end{align*}
and we replace the interior equation and the last two boundary conditions with
\begin{align*}
-\Delta^{\Sigma}\tilde{\bmC} &= - \tilde{\bmgamma}, \; \; \mathrm{in} \; \; D_0, \; \; \tilde{\bmgamma} \coloneqq  \partial_w^{\Sigma} \tilde{\bmomega} \, \mathrm{d}\sigma , \\
\tilde{\bmC} \cdot \bmN &= 0, \; \; \mathrm{at} \; \; w=0, \\
(\curl^\Sigma \tilde{\bmC})_{\parallel} &= -\gradp \, \Delta^{-1} \left( - \tilde{\bmgamma} \cdot \bmN \right) , \; \; \mathrm{at} \; \; w=0, 
\end{align*}
by replacing $\tilde{\bmC}$ with $\tilde{\bmC} + \grad^{\Sigma}\tilde{\phi}_1 + \grad^{\Sigma}\tilde{\phi}_2 $, where
\begin{equation*}
\begin{cases}
\Delta^{\Sigma}\tilde{\phi}_1 = 0 , &\; \; \mathrm{in} \; \; D_0, \\
\grad^{\Sigma}\tilde{\phi}_1 \cdot \bmN = \nabla \cdot (\tilde{\bmA}_{\rmh}|_{w=0} \, \delta\eta) , &\; \; \mathrm{at} \; \; w=0 , \\
\tilde{\phi}_1 = 0, &\; \; \mathrm{at} \; \; w=-h.
\end{cases}
\end{equation*}
and 
\begin{equation*}
\begin{cases}
\Delta^{\Sigma}\tilde{\phi}_2 = 0 , &\; \; \mathrm{in} \; \; D_0, \\
\tilde{\phi}_2  =  - \Delta^{-1} \, \nabla \cdot ( \tilde{\bmomega}_{\rmh}^{\perp}|_{w=0} \, \delta\eta  )   - (\curl^{\Sigma} \tilde{\bmA})_3|_{w=0} \, \delta\eta , &\; \; \mathrm{at} \; \; w=0 , \\
\partial_n \, \tilde{\phi}_2 = 0, &\; \; \mathrm{at} \; \; w=-h.
\end{cases}
\end{equation*}

Therefore
\begin{align*}
\mathrm{d}\tilde{\mathcal{G}}_{gen,II}[\eta](\delta\eta)\tilde{\bmomega} &= - \tilde{\mathcal{G}}_{gen,II}[\eta] \tilde{\bmgamma} - \nabla \cdot \left[ ( \underline{\tilde{\mathbb{K}}}_{II}[\eta]\tilde{\bmomega} - \underline{\tilde{\mathbb{W}}}_{II}[\eta]\tilde{\bmomega} \; \nabla\eta ) \, \delta\eta \right] , \nonumber 
\end{align*}
where $\tilde{\bmgamma}$, $\underline{\tilde{\mathbb{K}}}_{II}[\eta]\tilde{\bmomega}$ and $\underline{\tilde{\mathbb{W}}}_{II}[\eta]\tilde{\bmomega}$ are as in the statement.

\end{proof}

In order to compute the terms in the expansion
\begin{align*}
\tilde{\mathcal{G}}_{gen,II}[\eta] &= \sum_{j=0}^\infty \tilde{\mathcal{G}}_{j,II}[\eta], \; \; \tilde{\mathcal{G}}_{j,II}[\eta] = \frac{1}{j!} \, \mathrm{d}^j \tilde{\mathcal{G}}_{gen,II}[0][\eta,\ldots,\eta] , 
\end{align*}
where $\tilde{\mathcal{G}}_{j,II}(\eta)$ is homogeneous of degree $j$ in $\eta$, we first evaluate formula \eqref{eq:DiffGgenII} at $\delta\eta=\eta$. Then we expand the different terms appearing in \eqref{eq:DiffGgenII}: if $\tilde{\bmgamma}|_{\delta\eta=\eta}$ has no homogeneous term of degree $0$ in $\eta$ we have $\tilde{\bmgamma}_0 = \bmzero$, therefore 
\begin{align*}
\tilde{\bmgamma}|_{\delta\eta=\eta} &= \sum_{j \geq 1} \tilde{\bmgamma}_j , \; \; \tilde{\bmgamma}_j \coloneqq   \frac{1}{j!} \, \mathrm{d}^{j}(\tilde{\bmgamma}[0]\tilde{\bmomega})[\eta,\ldots,\eta], \\
\tilde{\nu} \coloneqq \tilde{\nu}[\eta]\tilde{\bmomega} &\coloneqq \tilde{\bmomega} \cdot \bmN|_{w=0} = \sum_{j \geq 0} \tilde{\nu}_j , \; \; \tilde{\nu}_j \coloneqq   \frac{1}{j!} \, \mathrm{d}^{j}(\tilde{\nu}[0]\tilde{\bmomega})[\eta,\ldots,\eta]  , 
\end{align*}
\begin{align*}
\underline{ \tilde{\mathbb{K}} }_{II}[\eta]\tilde{\bmomega} = \sum_{k =0}^\infty \underline{ \tilde{\mathbb{K}} }_{k,II}[\eta]\tilde{\bmomega} , &\;\; \underline{ \tilde{\mathbb{W}} }_{II}[\eta]\tilde{\bmomega} = \sum_{k=0}^\infty \underline{ \tilde{\mathbb{W}} }_{k,II}[\eta]\tilde{\bmomega} ,
\end{align*}
where
\begin{align*}
\underline{ \tilde{\mathbb{K}} }_{0,II}\tilde{\bmomega} &= - \gradp \, \Delta^{-1} \tilde{\nu}_0 , \;\; \underline{ \tilde{\mathbb{K}} }_{k,II}[\eta]\tilde{\bmomega} = -  \gradp \, \Delta^{-1} \tilde{\nu}_k , \;\; \forall \, k \geq 1, \\
\underline{ \tilde{\mathbb{W}} }_{0,II}\tilde{\bmomega} &=  \tilde{\mathcal{G}}_{0,II}\tilde{\bmomega} , \\
\underline{ \tilde{\mathbb{W}} }_{k,II}[\eta]\tilde{\bmomega} &= \underline{ \tilde{\mathbb{Z}} }_k[\eta]\tilde{\bmomega} + \sum_{i+2j=k} \left(  \underline{ \tilde{\mathbb{K}} }_{i-1,II}[\eta]\tilde{\bmomega} \cdot \nabla\eta + \tilde{\mathcal{G}}_{i,II}[\eta]\tilde{\bmomega} \right) \, (-1)^j |\nabla\eta|^{2j} , \;\; \forall \, k \geq 1, \\
\underline{ \tilde{\mathbb{Z}} }_k[\eta]\tilde{\bmomega} &=
\begin{cases}
(-1)^{k/2}|\nabla\eta|^{k} \, \tilde{\mathcal{G}}_{0,II}\tilde{\bmomega} ,  & \text{if}\;\; k \in 2 \mathbb{N}, \\
0 , & \text{if}\;\; k \notin 2 \mathbb{N}
\end{cases}
.
\end{align*}

Therefore, we obtain 
\begin{align*}
\sum_{j \geq 0} j \, \tilde{\mathcal{G}}_{j,II}[\eta]\tilde{\bmomega} &= - \left[ \sum_{h \geq 0} \tilde{\mathcal{G}}_{h,II}[\eta] \right] \, \left[ \sum_{k \geq 1} \tilde{\bmgamma}_k \right] -   \nabla \cdot \left[ \underline{ \tilde{\mathbb{K}} }_{0,II}\tilde{\bmomega} \; \eta \right] \\
&\qquad - \sum_{k \geq 1} \nabla \cdot \left[ \left( \underline{ \tilde{\mathbb{K}} }_{k,II}[\eta]\tilde{\bmomega} - \underline{ \tilde{\mathbb{W}} }_{k-1,II}[\eta]\tilde{\bmomega} \; \nabla\eta \right) \, \eta \right] ,
\end{align*}
hence
\begin{align} \label{eq:G1II} 
\tilde{\mathcal{G}}_{1,II}[\eta]\tilde{\bmomega} &= - \tilde{\mathcal{G}}_{0,II}\tilde{\bmgamma}_1 - \nabla \cdot \left[ \underline{ \tilde{\mathbb{K}} }_{0,II}\tilde{\bmomega} \; \eta \right] ,  
\end{align}
and for any $j \geq 2$
\begin{align}
\tilde{\mathcal{G}}_{j,II}[\eta]\tilde{\bmomega} &= \frac{1}{j} \left\{ - \sum_{h =0}^{j-1} \tilde{\mathcal{G}}_{h,II}[\eta] \tilde{\bmgamma}_{j-h}  - \nabla \cdot \left[ \left( \underline{ \tilde{\mathbb{K}} }_{j-1,II}[\eta]\tilde{\bmomega} - \underline{ \tilde{\mathbb{W}} }_{j-2,II}[\eta]\tilde{\bmomega} \; \nabla\eta \right) \, \eta \right] \right\} . \label{eq:GjII} 
\end{align}

By summarising, we have shown the following result.

\begin{theorem} \label{thm:HomExpGDNO}
	 
Under the assumptions of Proposition \ref{prop:DiffRotGDNO}, the term $\tilde{\mathcal{G}}_0$ in the homogeneous expansion \eqref{eq:TaylorGeta} of the operator $\tilde{\mathcal{G}}_{gen}[\eta]$ is given by 
\begin{align*}
\tilde{\mathcal{G}}_{0} &= \tilde{\mathcal{G}}_{0,I} + \tilde{\mathcal{G}}_{0,II} ,
\end{align*}
where $\tilde{\mathcal{G}}_{0,I}$ and $\tilde{\mathcal{G}}_{0,II}$ are given by by \eqref{eq:G0I} and \eqref{eq:G0II}, respectively. Moreover, if $\tilde{\bmgamma}_0 = \bmzero$, then for any $j \geq 1$ the term $\tilde{\mathcal{G}}_j[\eta]$ are given by 
\begin{align*}
\tilde{\mathcal{G}}_{j}[\eta] &= \tilde{\mathcal{G}}_{j,I}[\eta] + \tilde{\mathcal{G}}_{j,II}[\eta] ,
\end{align*}
where $\tilde{\mathcal{G}}_{j,I}$ and $\tilde{\mathcal{G}}_{j,II}$ are given by \eqref{eq:GjI}  and \eqref{eq:GjII}, respectively.
\end{theorem}

\section[Paralinearization of generalized Dirichlet--Neumann]{Paralinearization of the generalized Dirichlet--Neumann operator} \label{sec:GDNOpara}

Now we study the paralinearization of the generalized Dirichlet--Neumann operator $\mathcal{G}_{gen}[\eta]$ defined in \eqref{eq:GDNO}. Recalling the decomposition 
\begin{align*}
\mathcal{G}_{gen}[\eta](\Phi,\bmomega) &= \mathcal{G}_{gen,I}[\eta]\Phi + \mathcal{G}_{gen,II}[\eta]\bmomega , \\
\mathcal{G}_{gen,I}[\eta]\Phi &= \mathcal{G}[\eta]\Phi \, = \, \mathbb{U}_{I}[\eta]\Phi \cdot \bmN \, = \, \underline{ \nabla_{\bmx,z} \varphi } \cdot \bmN , \\
\mathcal{G}_{gen,II}[\eta]\bmomega &= \mathbb{U}_{II}[\eta]\bmomega \cdot \bmN \, = \, \nabla \cdot \bmA_{\parallel}^{\perp} ,
\end{align*}
used in Section \ref{sec:GDNOhom}, we paralinearize $\mathcal{G}_{gen,I}[\eta]$ and $\mathcal{G}_{gen,II}[\eta]$. Recall  that given a symbol $a \in \Sigma^{m}_{n-2}(\mathbb{R}^2)$, the corresponding paradifferential operator $T_a$ is given by \eqref{eq:ParadiffOp}, namely
\begin{align*}
\mathscr{F}(T_a u)(\bmxi) &= (2\pi)^{-2} \int_{\mathbb{R}^2} \chi(\bmxi - \bmxi') \, \mathscr{F}a(\bmxi-\bmxi') \, \mathscr{F} u(\bmxi') \, \mathrm{d}\bmxi' .
\end{align*}

Regarding the irrotational Dirichlet--Neumann operator $\mathcal{G}_{gen,I}[\eta]$, we recall the following result by Alazard and M\'etivier (see Theorem 2.12 of \cite{alazard2009paralinearization}; as mentioned in Remark 2.14 of the same paper, the result is valid also in the non periodic case).

\begin{proposition} \label{prop:DNOpara}

Let $n \geq 4$ be such that $n \notin \mathbb{N}$, $\varepsilon>0$. If $\eta \in H^{n+1/2}(\mathbb{R}^2)$, $\Phi \in \dot{H}^n(\mathbb{R}^2) $, then  
\begin{align} \label{eq:DNOpara}
\mathcal{G}_{gen,I}[\eta]\Phi &= T_{\lambda_{I}}( \Phi - T_{ \underline{\mathbb{W}}_{I}[\eta]\Phi }\eta ) - T_{ \underline{\mathbb{V}}[\eta]\Phi } \cdot \nabla\eta - T_{ \Div \underline{\mathbb{V}}[\eta]\Phi }\eta + \mathcal{R}_{I}(\eta,\Phi) ,
\end{align}
where the symbol $\lambda_{I} \in \Sigma^1_{n-2}(\mathbb{R}^2)$. Moreover, we have that 
\begin{equation*}
\underline{\mathbb{W}}_{I}[\eta]\Phi \in H^{n-1}(\mathbb{R}^2), \;\; \underline{\mathbb{V}}[\eta]\Phi \in ( H^{n-1}(\mathbb{R}^2) )^2 ,
\end{equation*}
and are given by
\begin{align*}
\underline{\mathbb{W}}_{I}[\eta]\Phi = \frac{ \mathcal{G}[\eta]\Phi + \nabla\eta \cdot \nabla\Phi  }{ 1+|\nabla\eta|^2 } , &\; \; \underline{\mathbb{V}}_{I}[\eta]\Phi = \nabla\Phi - ( \underline{\mathbb{W}}_{I}[\eta]\Phi ) \nabla\eta ,
\end{align*}
where $\mathcal{G}[\eta]$ denotes the classical Dirichlet--Neumann operator, and $\mathcal{R}_{I}(\eta,\Phi) \in H^{ 2n - 3 - \varepsilon}(\mathbb{R}^2)$ for any $\varepsilon>0$.

\end{proposition}

Next, we proceed to study the rotational part $\mathcal{G}_{gen,II}[\eta]$ of the generalized Dirichlet--Neumann operator.

\begin{theorem} \label{thm:ParalinGDNOp}

Let $n \geq 4$ be such that $n \notin \mathbb{N}$, $\varepsilon >0$ and assume that $\eta \in H^n(\mathbb{R}^2)$ is such that \eqref{eq:StrConnected} holds true. If
\begin{align} 
\partial_w^k \hat{\bmA} \in ( C^0([-h,0];H^{n-k}(\mathbb{R}^2)) )^3 , \;\; k=0,1, &\;\; \hat{\bmomega} \in ( C^0([-h,0];H^{n-2}(\mathbb{R}^2)) )^3 ,  \label{eq:AssParalinGDNOp}
\end{align}
 then
\begin{align}
\hat{\mathcal{G}}_{gen,II}[\eta]\hat{\bmomega} &=  T_{\lambda_{II}} (  \hat{\bmA} - T_{\partial_w^{\varrho} \hat{\bmA}} \eta ) + \frac{1}{\delta} T_{1+|\nabla\eta|^2} \, T_{ \gradp\eta } \cdot \hat{\bmf}(\eta,\hat{\bmomega}) \nonumber \\ 
&\qquad + T_{1+|\nabla\eta|^2} \left( \frac{1}{\delta} T_{\hat{A}_{2xw}}\eta - \frac{1}{\delta} T_{\hat{A}_{1yw}}\eta - \frac{1}{\delta^2}  T_{\eta_x} T_{\hat{A}_{2ww}}\eta + \frac{1}{\delta^2} T_{\eta_y} T_{\hat{A}_{1ww}}\eta \right) \nonumber \\
&\qquad - T_{\nabla\eta} \cdot ( \underline{ \hat{\mathbb{V}} }_{II}[\eta]\hat{\bmomega} ) - T_{ \underline{ \hat{\mathbb{V}} }_{II}[\eta]\hat{\bmomega} } \cdot \nabla\eta - T_{ |\nabla\eta|^2 } \, ( \underline{ \hat{\mathbb{W}} }_{II}[\eta]\hat{\bmomega} )  + \mathcal{R}_{II}, \label{eq:paralinGDNOrot}
\end{align}
where the right-hand side is evaluated at $w=0$; moreover, $\lambda_{II} \in M_{1,3}( \Sigma^1_{n-2}(\mathbb{R}^2) )$ is given by \eqref{eq:lambdaII}, $\hat{\bmf}(\eta,\hat{\bmomega}) \in ( H^{n-1}(\mathbb{R}^2) )^2$,
\begin{align*}
\hat{\underline{\mathbb{V}}}_{II}[\eta]\hat{\bmomega} &= (\curl^{\varrho} \hat{\bmA})_{\rmh}|_{w=0} \in ( H^{n-1}(\mathbb{R}^2) )^2 , \\
 \hat{\underline{\mathbb{W}}}_{II}[\eta]\hat{\bmomega} &= (\curl^{\varrho} \hat{\bmA})_3|_{w=0} \in  H^{n-1}(\mathbb{R}^2) ,
\end{align*}
and $\mathcal{R}_{II}(\eta,\hat{\bmomega}) \in H^{2n-3-\varepsilon}(\mathbb{R}^2)$ for any $\varepsilon>0$.

\end{theorem}

In the rest of the section we prove the above Theorem \ref{thm:ParalinGDNOp}, by  following the approach adopted in Sec. 4 of \cite{alazard2009paralinearization}. 

By flattening \eqref{eq:systAalt} we find that $\hat{\bmA}$ satisfies
\begin{align} 
-\Delta^{\varrho} \hat{\bmA} - \hat{\bmomega} &= 0 \; \; \text{in} \; \; D_0, \label{eq:IntEqRho}
\end{align}
with boundary conditions
\begin{align}
\hat{\bmA}\cdot \bmN&=0,\quad w=0, \label{eq:bdry1_vec}\\
(\text{curl}^{\varrho}\hat{\bmA})_{\vert\vert}&= - \gradp \Delta^{-1} ( \hat{\bmomega} \cdot \bmN),\quad w=0. \label{eq:bdry2_vec}
\end{align}

Moreover, in the new coordinates the rotational part of the straightened generalized Dirichlet--Neumann operator can be written as
\begin{align}
\hat{\mathcal{G}}_{gen,II}[\eta]\hat{\bmomega} &= \hat{A}_{2x}+\eta_y\hat{A}_{3x}-\hat{A}_{1y}-\eta_x\hat{A}_{3y}\vert_{w=0}. \label{eq:rotGDNOdef}
\end{align}

Equation \eqref{eq:IntEqRho} is equivalent to the system
\begin{align} \label{eq:main_eq_sys}
E \, \hat{\bmA} &= - \check{\mathtt{a}} \, \hat{\bmomega} , \;\; E \coloneqq \partial_w^2 + \check{\mathtt{a}} \, \Delta + \check{\bmb} \cdot \nabla \partial_w - \check{\mathtt{c}} \, \partial_w, 
\end{align}
where
\begin{align*}
\check{\mathtt{a}} = \mathtt{a}_0^{-1}, \;\; \check{\bmb} = \bmb_0/\mathtt{a}_0, \;\; \check{\mathtt{c}} = \mathtt{c}_0/\mathtt{a}_0, &\quad
\mathtt{a}_0 = \frac{ 1+ |\nabla\eta|^2 }{\delta^2}, \;\; \mathtt{b}_0 = - \frac{2\nabla\eta}{\delta}, \;\; \mathtt{c}_0 = \frac{\Delta\eta}{\delta} .
\end{align*}

\begin{proposition} \label{prop:paraGood}

Set
\begin{align} \label{eq:paraGood}
\hat{\bmB} &\coloneqq \hat{\bmA} - T_{\partial_w^{\varrho} \hat{\bmA}} \eta, 
\end{align}
and let $T_E\coloneqq \partial_w^2 +  ( Id + T_{ \check{\mathtt{a}} -1 } ) \Delta + T_{\check{\bmb}} \cdot \nabla \, \partial_w - T_{ \check{\mathtt{c}} } \,\partial_w$, then $\hat{\bmB}$ satisfies 
\begin{align} \label{eq:ParaIntEq}
T_{E} \hat{\bmB} &= \hat{\bmF}_0 + \hat{\bmF}_1 ,
\end{align}
where 
\begin{align}
\hat{\bmF}_0 \coloneqq \hat{\bmF}_0(\hat{\bmomega},\eta) &=  - T_{ \check{\mathtt{a}} -1} \hat{\bmomega} - T_{\hat{\bmomega}} ( \check{\mathtt{a}} -1) + T_{  \partial_w^{\varrho}( \check{\mathtt{a}} \, \hat{\bmomega} ) } \eta + T_{ [\partial_w^{\varrho} ,E] \hat{\bmA} } \eta \nonumber \\
&\qquad - ( Id + T_{ \check{\mathtt{a}} -1 } ) \, \left[ T_{\Delta \partial_w^{\varrho} \hat{\bmA}} \eta + 2 T_{ \nabla \partial_w^{\varrho} \hat{\bmA}} \cdot \nabla\eta   \right]  - T_{ \check{\bmb} } \cdot T_{  \partial_w \partial_w^{\varrho} \hat{\bmA}} \nabla\eta \nonumber \\
&\qquad - T_{ \Delta \hat{\bmA} } \, ( \check{\mathtt{a}} -1)  - T_{ \nabla \, \partial_w \hat{\bmA} } \cdot \check{\bmb} + T_{ \partial_w \hat{\bmA} } \, \check{\mathtt{c}} , \label{eq:F0formula}
\end{align}
and $\hat{\bmF}_1 \in ( C^0( [-h,0];H^{2n-4}(\mathbb{R}^2) ) )^3$ . 

\end{proposition}

\begin{proof}

We follow an argument similar to the proof of Proposition 4.12 in \cite{alazard2009paralinearization}. First, recall that by assumption \eqref{eq:AssParalinGDNOp} we have that
\begin{align*}
\partial_w^k \hat{\bmA} &\in ( C^0( [-h,0] ; H^{n-k}(\mathbb{R}^2) ) )^3 , \;\; k=0,1, \;\; \hat{\bmomega} \in ( C^0( [-h,0] ; H^{n-2}(\mathbb{R}^2) ) )^3 ,
\end{align*}
and that by the regularity assumption of $\eta$ we have
\begin{align*}
\check{\mathtt{a}}-1 \in H^{n-1}(\mathbb{R}^2)  , \quad \check{\bmb} &\in (  H^{n-1}(\mathbb{R}^2)  )^2 , \quad \check{\mathtt{c}} \in  H^{n-2}(\mathbb{R}^2)  .
\end{align*}

By Lemma \ref{lem:Paraprod2} we have that if we set
\begin{align*}
\mathcal{T}(\hat{\bmA}) &\coloneqq T_{ \Delta \hat{\bmA} } \, (\check{\mathtt{a}}-1)  + T_{ \nabla \, \partial_w \hat{\bmA} } \cdot \check{\bmb} - T_{ \partial_w \hat{\bmA} } \, \check{\mathtt{c}} ,
\end{align*}
then
\begin{align*}
& E \hat{\bmA} +  \hat{\bmomega} + (\check{\mathtt{a}} -1) \, \hat{\bmomega} - T_{E} \hat{\bmA} - \mathcal{T}(\hat{\bmA}) - T_{ \check{\mathtt{a}} -1} \hat{\bmomega} - T_{ \hat{\bmomega} } (\check{\mathtt{a}}-1)  
\end{align*}
belongs to $( C^0([-h,0];H^{2n-4}(\mathbb{R}^2)) )^3$.

Next, we substitute $\hat{\bmA} = \hat{\bmB} + T_{\partial_w^{\varrho} \hat{\bmA}} \eta$ in the above interior equation. Recall that $E \hat{\bmA} = -\check{\mathtt{a}} \, \hat{\bmomega}$, and that $\partial_w^k \hat{\bmA} \in ( C^0([-h,0];H^{n-k}(\mathbb{R}^2) )^3$ for $k \leq 3$. 

From the definition of $\hat{\bmB}$, we get (up to a smooth remainder)
\begin{align*}
\partial_w^2 \hat{\bmB} &= \partial_w^2 \hat{\bmA} - \partial_w^2 \left( T_{\partial_w^{\varrho} \hat{\bmA}} \eta \right) =  \partial_w^2 \hat{\bmA} - T_{\partial_w^2 \partial_w^{\varrho} \hat{\bmA}} \eta , \\
\Delta \hat{\bmB} &= \Delta \hat{\bmA} - T_{\Delta \partial_w^{\varrho} \hat{\bmA}} \eta -2 T_{ \nabla \partial_w^{\varrho} \hat{\bmA}} \cdot \nabla\eta - T_{\partial_w^{\varrho} \hat{\bmA}} \Delta\eta , \\
\nabla \partial_w \hat{\bmB} &= \nabla \partial_w \hat{\bmA} - T_{\nabla \partial_w \partial_w^{\varrho} \hat{\bmA}} \eta - T_{  \partial_w \partial_w^{\varrho} \hat{\bmA}} \nabla\eta  , \\
\partial_w \hat{\bmB} &= \partial_w \hat{\bmA} - T_{\partial_w \partial_w^{\varrho} \hat{\bmA}} \eta .
\end{align*}
By Lemma \ref{lem:Paraprod2} we can deduce from the interior equation $E\hat{\bmA} + \check{\mathtt{a}} \, \hat{\bmomega} = \bmzero$ that
\begin{align*}
& T_{ E \, \partial_w^{\varrho} \hat{\bmA}  } \eta + T_{  \partial_w^{\varrho}( \check{\mathtt{a}} \, \hat{\bmomega} ) } \eta + T_{  [\partial_w^{\varrho} ,E] \hat{\bmA} } \eta \in ( C^0( [-h,0]; H^{2n-4}(\mathbb{R}^2) ) )^3 .
\end{align*}

Therefore, we have (up to a smooth remainder)
\begin{align*}
T_{E}\hat{\bmA} - T_{E}\hat{\bmB} &= T_{\partial_w^2 \partial_w^{\varrho} \hat{\bmA}} \eta  + ( Id + T_{ \check{\mathtt{a}} -1 } ) \, \left[ T_{\Delta \partial_w^{\varrho} \hat{\bmA}} \eta + 2 T_{ \nabla \partial_w^{\varrho} \hat{\bmA}} \cdot \nabla\eta + T_{\partial_w^{\varrho} \hat{\bmA}} \Delta\eta  \right] \\
&\qquad + T_{ \check{\bmb} } \cdot \left[  T_{\nabla \partial_w \partial_w^{\varrho} \hat{\bmA}} \eta + T_{  \partial_w \partial_w^{\varrho} \hat{\bmA}} \nabla\eta  \right] - T_{ \check{\mathtt{c}} } \,  T_{\partial_w \partial_w^{\varrho} \hat{\bmA}} \eta  ,
\end{align*}
which, up to a remainder in $( C^0( [-h,0]; H^{2n-4}(\mathbb{R}^2) ) )^3$, it implies that
\begin{align*}
 T_{E}\hat{\bmA} - T_{E}\hat{\bmB} &= - T_{  \partial_w^{\varrho}( \check{\mathtt{a}} \, \hat{\bmomega} ) } \eta - T_{ [\partial_w^{\varrho} ,E] \hat{\bmA} } \eta  + ( Id + T_{ \check{\mathtt{a}} -1 } ) \, \left[ T_{\Delta \partial_w^{\varrho} \hat{\bmA}} \eta + 2 T_{ \nabla \partial_w^{\varrho} \hat{\bmA}} \cdot \nabla\eta   \right] \\
&\qquad + T_{ \check{\bmb} } \cdot T_{  \partial_w \partial_w^{\varrho} \hat{\bmA}} \nabla\eta .
\end{align*}

If we plug in the above formula in the paralinearized interior equation we obtain that, up to a remainder in $( C^0( [-h,0]; H^{2n-4}(\mathbb{R}^2) ) )^3$, 
\begin{align*}
 T_{E}\hat{\bmB} &=   T_{  \partial_w^{\varrho}( \check{\mathtt{a}} \, \hat{\bmomega} ) } \eta + T_{ [\partial_w^{\varrho} ,E] \hat{\bmA} } \eta  - ( Id + T_{ \check{\mathtt{a}} -1 } ) \, \left[ T_{\Delta \partial_w^{\varrho} \hat{\bmA}} \eta + 2 T_{ \nabla \partial_w^{\varrho} \hat{\bmA}} \cdot \nabla\eta   \right] \\
&\qquad - T_{ \check{\bmb} } \cdot T_{  \partial_w \partial_w^{\varrho} \hat{\bmA}} \nabla\eta - \mathcal{T}(\hat{\bmA}) - T_{ \check{\mathtt{a}}-1 } \hat{\bmomega} - T_{\hat{\bmomega}} (\check{\mathtt{a}}-1) .
\end{align*}

\end{proof}

Next we consider the reduction to the boundary. We want to find two matrix-valued symbols $\mathtt{M}$, $\mathtt{N}$ such that (up to a smoothing remainder)
\begin{align*}
T_{E} \, \mathbb{I}_3 &= \left( \partial_w \mathbb{I}_3 - T_{ \mathtt{N} } \right) \left( \partial_w \mathbb{I}_3 - T_{ \mathtt{M} } \right),
\end{align*}
see Sec. 3.4.2 of \cite{alazard2011water} and Sec. 4.3 of \cite{alazard2009paralinearization} for a corresponding argument for the irrotational case, and Sec. 3.1 of \cite{groves2024analytical} for the case of Beltrami flows.

\begin{lemma} \label{lem:factorisation}
There exist operators $M, N$ with associated symbols 
\begin{equation*}
\mathtt{M},\mathtt{N} \in M_3 \left(  \Sigma^1_{n-2}({\mathbb R}^2)  \right)
\end{equation*}
such that
\begin{itemize}
\item[(i)]
$T_{E} \, \mathbb{I}_3 - (\partial_w \mathbb{I}_3 - T_{ \mathtt{N} } )(\partial_w \mathbb{I}_3 - T_{ \mathtt{M} } ) \in M_3 \left(  \Sigma^{1-(n-4)}_{n-2}(\mathbb{R}^2)  \right)$,
\item[(ii)]
the principal symbols $\mathtt{M}^{(1)}$, $\mathtt{N}^{(1)}$ of $M$, $N$ take the form $\mathtt{M}^{(1)} =  \mathtt{m}^{(1)} \mathbb{I}_3$, $\mathtt{N}^{(1)} = \mathtt{n}^{(1)} \mathbb{I}_3$, where the scalar-valued symbols $\mathtt{m}^{(1)}$, $-\mathtt{n}^{(1)} \in  \Sigma^1_{n-2}({\mathbb R}^2) $ are strongly elliptic.
\end{itemize}
\end{lemma}

\begin{proof}

We first set $ L_1 \coloneqq [ \check{ \bmb } \cdot \nabla -  \check{\mathtt{c}}  ] \, \mathbb{I}_3$, $L_0 \coloneqq  \check{\mathtt{a}}   \, \Delta$, so that $E = \partial_w^2 \, \mathbb{I}_3 + L_1 \, \partial_w + L_0$.  Recall that $\eta \in H^n(\mathbb{R}^2)$. Since
\begin{align*}
E -  (\partial_w \mathbb{I}_3 - N)(\partial_w \mathbb{I}_3 - M ) &= (L_1+  (M +N)) \partial_w + (L_0 -  N M)
\end{align*}
we set
\begin{align} \label{eq:Defn of N}
N &= -  L_1 - M
\end{align}
and, by following Proposition \ref{prop:AsympParadiffOp}, we seek $M$ such that
\begin{align*}
L_0 + L_1 M +  M^2 &= \mathbb{O}_3 
\end{align*}
by constructing a symbol $\mathtt{M} \in M_3( \Sigma^1_{n-2}(\mathbb{R}^2) )$ such that 
\begin{align*}
&  - \check{\mathtt{a}} \, |\bmxi|^2 \, \mathbb{I}_3 - \sum_{\alpha=(\alpha_1,\alpha_2) \in \mathbb{N}_0^2 } \left[  \partial_{\xi_1}^{\alpha_1} \partial_{\xi_2}^{\alpha_2} |\bmxi|^2 \, D_{x}^{\alpha_1} D_{y}^{\alpha_2}  \check{\mathtt{a}}  \right] \, \mathbb{I}_3 + \left[ \mathrm{i} \, \check{\bmb}  \cdot \bmxi - \check{\mathtt{c}} \right] \, \mathtt{M}   \\
&\qquad + ( \check{\bmb} \cdot \nabla ) \, \mathtt{M}  + \sum_{\alpha=(\alpha_1,\alpha_2) \in \mathbb{N}^2 } \partial_{\xi_1}^{\alpha_1} \partial_{\xi_2}^{\alpha_2} \mathtt{M} \, D_{x}^{\alpha_1} D_{y}^{\alpha_2} \mathtt{M} = \mathbb{O}_3 ,
\end{align*}
up to a remainder in $M_3\left(  \Sigma^{1-(n-4)}_{n-2}(\mathbb{R}^2)  \right)$, and such that
\begin{align*}
\mathtt{M} &= \sum_{ 1-(n-4) < j \leq 1} \mathtt{M}^{(j)}, \; \; \mathtt{M}^{(j)} \in M_3 \left( \Gamma^j_{n-3+j}({\mathbb R}^2)  \right) .
\end{align*}

We proceed by computing the terms in the aysmptotic expansion of $\mathtt{M}$ inductively.

\begin{itemize}
\item
At leading order we have 
\begin{align*}
- \check{\mathtt{a}} \, |\bmxi|^2 {\mathbb I}_3 + \mathrm{i} \check{\bmb} \cdot \bmxi \mathtt{M}^{(1)} +  (\mathtt{M}^{(1)})^2 &= \mathbb{O}_3 ,
\end{align*}
so that
\begin{align*}
\mathtt{M}^{(1)} &= \mathtt{m}^{(1)} \mathbb{I}_3,
\end{align*}
where 
\begin{align*}
\mathtt{m}^{(1)}\coloneqq \frac{1}{2} \left[  - \mathrm{i} \check{\bmb} \cdot \bmxi +  \sqrt{ - ( \check{\bmb} \cdot \bmxi )^2 + 4 \check{\mathtt{a}} \, |\bmxi|^2  } \right] .
\end{align*}
Note that 
\begin{align*}
\mathtt{m}^{(1)}  &= \delta \, \frac{ \mathrm{i} \, \bmxi \cdot \nabla\eta + \lambda^{(1)} }{ 1+|\nabla\eta|^2 }, \;\; \lambda^{(1)} \, \coloneqq \, \sqrt{ (1+|\nabla\eta|^2)|\bmxi|^2 - (\bmxi \cdot \nabla\eta)^2 } ,
\end{align*}
where $\lambda^{(1)}$ is the leading order symbol of the classical Dirichlet--Neumann operator.

\item
The subprincipal symbol is found from the equation
\begin{align*}
&\mathrm{i} \, ( \nabla \check{\mathtt{a}} ) \cdot 2\bmxi \, \mathbb{I}_3 - \check{\mathtt{c}} \, \mathtt{M}^{(1)}  +\mathrm{i} \check{\bmb}\cdot\bmxi \mathtt{M}^{(0)} +  \mathtt{M}^{(0)}\mathtt{M}^{(1)} +  \mathtt{M}^{(1)}\mathtt{M}^{(0)} \\
&\;\;\;\; +( \check{\bmb}\cdot\grad)\mathtt{M}^{(1)} - \mathrm{i}  \partial_{\xi_1} \mathtt{M}^{(1)} \partial_x \mathtt{M}^{(1)} -\mathrm{i} \partial_{\xi_2} \mathtt{M}^{(1)} \partial_y \mathtt{M}^{(1)} = \mathbb{O}_3,
\end{align*}
which yields
\begin{align*}
\mathtt{M}^{(0)} &= \mathtt{m}^{(0)}\mathbb{I}_3 ,
\end{align*}
where
\begin{align*}
\mathtt{m}^{(0)} &= \frac{1}{ \mathrm{i} \check{\bmb} \cdot \bmxi + 2\mathtt{m}^{(1)} } \left[ \mathrm{i} \, \nabla_{\bmxi} \mathtt{m}^{(1)} \cdot \nabla\mathtt{m}^{(1)} - ( \check{\bmb} \cdot \nabla ) \mathtt{m}^{(1)} + \check{\mathtt{c}} \, \mathtt{m}^{(1)} - \mathrm{i} \, (\nabla \check{\mathtt{a}}) \cdot 2 \bmxi  \right] .
\end{align*}

\item
Suppose that $\mathtt{M}^{(j)}$ has been calculated for $j=1,0,\ldots -J$ for some $0 \leq J < 1-(n-4)$. The term $\mathtt{M}^{(-J-1)}$ can be found from the equation
$$(2 \mathtt{m}^{(1)} + \mathrm{i} \check{\bmb} \cdot \bmxi)\mathtt{M}^{(-J-1)}=\mathtt{G}^{(-J)},$$
where $\mathtt{G}^{(-J)} \in M_3( \Sigma^{(-J)}_{n-2}({\mathbb R}^2) )$ depends on $\check{\mathtt{a}}$, $\check{\bmb}$, $\check{\mathtt{c}}$ and $(\mathtt{M}^{(j)})_{j=1,\ldots,-J}$.
\end{itemize}
The construction is completed by noting that there exists a symbol
\begin{equation*}
\mathtt{M} \in M_3( \Sigma^1_{n-2} ({\mathbb R}^2) )
\end{equation*}
such that
\begin{equation*}
\mathtt{M} = \sum_{ 1-(n-4) < j \leq 1 } \mathtt{M}^{(j)}
\end{equation*}
(see Shubin \cite[\S3.3]{shubin2001pseudodifferential}). Defining $N$ by equation \eqref{eq:Defn of N}, we can compute the asymptotic expansion 
$$\mathtt{N} = \sum_{ 1-(n-4) < j \leq 1 } \mathtt{N}^{(j)}$$
of the operator $N$; in particular, we find that
\begin{align*}
\mathtt{N}^{(1)} =  \mathtt{n}^{(1)} \mathbb{I}_3, &\qquad \mathtt{n}^{(1)}\coloneqq \frac{1}{2} \left[  - \mathrm{i} \check{\bmb} \cdot \bmxi -  \sqrt{ - ( \check{\bmb} \cdot \bmxi )^2 + 4 \check{\mathtt{a}} \, |\bmxi|^2  } \right] ,
\end{align*}
so that
\begin{align*}
\mathtt{n}^{(1)}  &= \delta \frac{ \mathrm{i} \bmxi \cdot \grad\eta - \lambda^{(1)} }{1+|\grad\eta|^2}.
\end{align*}
Finally, note that
$$\mathrm{Re}\, \mathtt{m}^{(1)} = - \mathrm{Re}\, \mathtt{n}^{(1)}  = \delta \frac{ \lambda^{(1)}}{1+|\grad \eta|^2}
\geq \delta \frac{ |\bmxi|}{1+\max |\grad \eta|^2} \gtrsim \delta \langle \bmxi \rangle$$
for sufficiently large $|\bmxi|$, so that $\mathtt{m}^{(1)}$ and $-\mathtt{n}^{(1)} $ are strongly elliptic.

\end{proof}

Using Lemma \ref{lem:factorisation} and arguing as in the proof of Lemma 4.13 of \cite{alazard2009paralinearization} we can deduce the following result.

\begin{lemma} \label{lem:ExpFactorization}

There exist two matrix-valued symbols
\begin{equation*}
\mathtt{M}, \mathtt{N} \in M_3(  \Sigma^1_{n-2}(\mathbb{R}^2) ) 
\end{equation*}
such that
\begin{align}
\left( \partial_w \mathbb{I}_3 - T_{ \mathtt{N} } \right) \left( \partial_w \mathbb{I}_3 - T_{ \mathtt{M} } \right) \hat{\bmB} = \hat{\bmF}_0 + \hat{\bmF}_1 ,
\end{align}
where $\hat{\bmF}_0$ is given by \eqref{eq:F0formula}, and where $\hat{\bmF}_1 \in ( C^0([-h,0];H^{2n-4}(\mathbb{R}^2)) )^3$ .
\end{lemma}

Recalling the definition \eqref{eq:paraGood} of the good unknown $\hat{\bmB}$ and the regularity assumption on $\hat{\bmA}$, we have that 
\begin{equation*}
\partial_w \hat{\bmB} - T_{\mathtt{M}} \hat{\bmB}|_{w=0}  \in \left(  H^{n-1}(\mathbb{R}^2) \right)^3 .
\end{equation*}

For the irrotational case the elliptic regularity result given by Proposition 4.13 in \cite{alazard2009paralinearization} can improve the regularity in the right-hand side of the (analogue of the) above equation up to $H^{2n-3-\varepsilon}(\mathbb{R}^2)$ for any fixed $\varepsilon >0$ (see Corollary 4.14 of \cite{alazard2009paralinearization}). However, this is not true if $\hat{\bmomega} \neq \bmzero$. 

\begin{corollary} \label{cor:EqGoodUnknown}

Let $\varepsilon >0$. Introduce $\hat{\bmW} \coloneqq (\partial_w-T_{\mathtt{M}})\hat{\bmB}$ and we write $N = N^{(1)} + N^{(\leq 0)}$, where $N^{(1)}$ is the principal part of $N$ obtained in Lemma \ref{lem:ExpFactorization}, and $ N^{(\leq 0)} \in M_3 \left(  \Gamma^0_0(\mathbb{R}^2)  \right)$. Then 
\begin{align*}
\partial_w \hat{\bmW} - T_{ N^{(1)} } \hat{\bmW} &= T_{ N^{(\leq 0)} } \hat{\bmW} +\hat{\bmF}_0 + \hat{\bmF}_1 ,
\end{align*}
where $\hat{\bmF}_0$ is given by \eqref{eq:F0formula}, and $\hat{\bmF}_1 \in ( C^0([-h,0];H^{2n-4}(\mathbb{R}^2)) )^3$ . 

Moreover, up to terms in $( H^{2n-3-\varepsilon}(\mathbb{R}^2) )^3$, the following equality holds true,
\begin{align}
&\partial_w \hat{\bmB} - T_{\mathtt{M}} \hat{\bmB}|_{w=0} = \int_{-h}^0 (T_{\mathscr{E} } \hat{\bmF}_0)(z)  \, \mathrm{d}z , \label{eq:ParaGoodEq} 
\end{align}
where
\begin{align*}
& \mathscr{E}(z;\bmx,\bmxi) = \exp(z \, N^{(1)}(\bmx,\bmxi))  - \mathrm{i} \,  \exp(z \, N^{(1)}(\bmx,\bmxi)) \, \frac{z^2}{2} \, \nabla_{\bmxi} N^{(1)}(\bmx,\bmxi) \cdot \nabla_{\bmx} N^{(1)}(\bmx,\bmxi) . \nonumber
\end{align*}

\end{corollary}

The above corollary can be proved by following the proof of Proposition 4.13 in \cite{alazard2009paralinearization}. 

Now, in order to simplify the notation we omit until the end of the section the fact that all functions depending on $w$ are evaluated on $w=0$.

First, from \eqref{eq:bdry1_vec} we have
\begin{align} \label{eq:bdry1_new}
\hat{A}_{3} &= \eta_x  \hat{A}_1 + \eta_y \hat{A}_2.
\end{align}

Next, we rewrite the boundary condition \eqref{eq:bdry2_vec} in the following way,
\begin{align}
& \hat{A}_{3y} - \frac{1}{\delta} \eta_y \hat{A}_{3w} -  \frac{1}{\delta} \hat{A}_{2w} + \eta_x \left(  \hat{A}_{2x} - \hat{A}_{1y} -  \frac{1}{\delta} \eta_x \hat{A}_{2w} +  \frac{1}{\delta} \eta_y \hat{A}_{1w} \right) =  -  \partial_y \Delta^{-1} ( \hat{\bmomega} \cdot \bmN ), \label{eq:bdry2_comp1} \\
& -\hat{A}_{3x} +  \frac{1}{\delta}  \eta_x \hat{A}_{3w} +  \frac{1}{\delta} \hat{A}_{1w} + \eta_y \left(  \hat{A}_{2x} - \hat{A}_{1y} -   \frac{1}{\delta} \eta_x \hat{A}_{2w} +  \frac{1}{\delta} \eta_y \hat{A}_{1w} \right) = \partial_x \Delta^{-1} ( \hat{\bmomega} \cdot \bmN ). \label{eq:bdry2_comp2}
\end{align}

Then if we multiply \eqref{eq:bdry2_comp1} by $\eta_x$ and \eqref{eq:bdry2_comp2} by $\eta_y$ and we sum them up, we obtain
\begin{align*}
& -  \frac{1}{\delta} \eta_x (1+|\nabla\eta|^2) \hat{A}_{2w} +  \frac{1}{\delta} \eta_y (1+|\nabla\eta|^2) \hat{A}_{1w} + (\eta_x \eta_{xy} - \eta_y \eta_{xx}) \hat{A}_1   \\
&\qquad + (\eta_x \eta_{yy} - \eta_y \eta_{xy}) \hat{A}_2 + \eta_x \eta_y \hat{A}_{2y} + \eta_x^2 \hat{A}_{2x} - \eta_x \eta_y \hat{A}_{1x} - \eta_y^2 \hat{A}_{1y} \\
& =   -  \nabla\eta \cdot \gradp \Delta^{-1} ( \hat{\bmomega} \cdot \bmN ),
\end{align*}
and if we add $\curl^{\varrho}\hat{\bmA} \cdot \bmN$ to both sides of the above equation we obtain
\begin{align} 
( \curl^{\varrho} \hat{\bmA} )_3 &= \partial_x^{\varrho} \hat{A}_2 - \partial_y^{\varrho} \hat{A}_1 .   \label{eq:curlA3new}
\end{align}

Moreover, from the incompressibility we have
\begin{align} \label{eq:IncomprSigma}
 \frac{1}{\delta} \hat{A}_{3w} &= - \partial_x^{\varrho} \hat{A}_1 - \partial_y^{\varrho} \hat{A}_2.
\end{align}

Moreover, plugging \eqref{eq:bdry1_new} into \eqref{eq:bdry2_comp1}-\eqref{eq:bdry2_comp2} we obtain
\begin{align}
 \frac{1}{\delta} ( 1+|\nabla\eta|^2 ) \hat{A}_{1w} &= \eta_{xx} \hat{A}_1 + \eta_{xy} \hat{A}_2 + 2 \eta_x \hat{A}_{1x} + \eta_y \hat{A}_{1y} + (\eta_x + \eta_y) \hat{A}_{2y}  \nonumber \\
&\qquad - \eta_y \hat{A}_{2x} +   \partial_x \Delta^{-1} ( \hat{\bmomega} \cdot \bmN ) , \label{eq:bdry3_comp1} \\
\frac{1}{\delta} ( 1+|\nabla\eta|^2 ) \hat{A}_{2w} &= \eta_{xy} \hat{A}_1 + \eta_{yy} \hat{A}_2 + (\eta_x + \eta_y) \hat{A}_{1x} - \eta_x \hat{A}_{1y} + 2 \eta_y \hat{A}_{2y}  \nonumber \\
&\qquad + \eta_x \hat{A}_{2x} + \partial_y \Delta^{-1} ( \hat{\bmomega} \cdot \bmN ) . \label{eq:bdry3_comp2}
\end{align}

Consider the formula \eqref{eq:rotGDNOdef} for the rotational part $\hat{\mathcal{G}}_{gen,II}[\eta]\hat{\bmomega}$ of the straightened version of the generalized Dirichlet--Neumann operator. If we plug the boundary condition \eqref{eq:bdry1_new} into \eqref{eq:rotGDNOdef} we get
\begin{align}
\hat{\mathcal{G}}_{gen,II}[\eta]\hat{\bmomega} &= (1+ \eta_y^2) \hat{A}_{2x} - (1+ \eta_x^2) \hat{A}_{1y} + (\eta_y \eta_{xx} - \eta_x \eta_{xy} ) \hat{A}_1 + \eta_x \eta_y \hat{A}_{1x} \nonumber \\
&\qquad + (\eta_y \eta_{xy} - \eta_x \eta_{yy}) \hat{A}_2 - \eta_x \eta_y \hat{A}_{2y}, \;\; \text{at} \; \; w=0 . \label{eq:GDNO2} 
\end{align}
Finally, we conclude with the paralinearization of the generalized Dirichlet--Neumann operator. By \eqref{eq:curlA3new}, we have 
\begin{align*}
\hat{\mathcal{G}}_{gen,II}[\eta]\hat{\bmomega} &= (1+|\nabla\eta|^2) ( \partial_x^{\varrho} \hat{A}_2 - \partial_y^{\varrho} \hat{A}_1 ) - \nabla\eta \cdot (\curl^{\varrho}\hat{\bmA})_{\rmh}  - (\curl^{\varrho}\hat{\bmA})_3 \, |\nabla\eta|^2 \;\; \text{at} \; \; w=0 ,
\end{align*}
and if we paralinearize the above expression we get
\begin{align}
&  T_{1+|\nabla\eta|^2} ( \partial_x^{\varrho} \hat{A}_2 - \partial_y^{\varrho} \hat{A}_1 ) +  T_{ \partial_x^{\varrho} \hat{A}_2 - \partial_y^{\varrho} \hat{A}_1 } \, |\nabla\eta|^2  - T_{\nabla\eta} \cdot ( \underline{ \hat{\mathbb{V}} }_{II}[\eta]\hat{\bmomega} ) \nonumber \\
&\qquad - T_{ \underline{ \hat{\mathbb{V}} }_{II}[\eta]\hat{\bmomega} } \cdot \nabla\eta   - T_{ \underline{ \hat{\mathbb{W}} }_{II}[\eta]\hat{\bmomega} } |\nabla\eta|^2 - T_{ |\nabla\eta|^2 } \, ( \underline{ \hat{\mathbb{W}} }_{II}[\eta]\hat{\bmomega} )  - \hat{\mathcal{G}}_{gen,II}[\eta]\hat{\bmomega} \in H^{2n-3}(\mathbb{R}^2),  \label{eq:ParalinGDNO}
\end{align}
where we recall that $\underline{ \hat{\mathbb{V}} }_{II}[\eta]\hat{\bmomega}=(\curl \hat\bmA)_\rmh |_{w=0}$ and $\underline{ \hat{\mathbb{W}} }_{II}[\eta]\hat{\bmomega}=(\curl \hat{\bmA})_3 |_{w=0}$.

Now, observe that $\partial_y^{\varrho} \hat{A}_1 = \hat{A}_{1y} -  \frac{1}{\delta} \eta_y \hat{A}_{1w}$ and $\partial_x^{\varrho} \hat{A}_2 = \hat{A}_{2x} -  \frac{1}{\delta} \eta_x \hat{A}_{2w}$, hence the good unknown $\hat{\bmB}$ given by \eqref{eq:paraGood} satisfies
\begin{align*}
\hat{B}_{1y} +  \frac{1}{\delta} T_{ \hat{A}_{1yw} } \eta -  \frac{1}{\delta} T_{\eta_y} \hat{B}_{1w} -  \frac{1}{\delta^2} T_{\eta_y} T_{ \hat{A}_{1ww} } \eta - \partial_y^{\varrho} \hat{A}_1   &\in H^{2n-3}(\mathbb{R}^2), \\
\hat{B}_{2x} +  \frac{1}{\delta} T_{ \hat{A}_{2xw} } \eta -  \frac{1}{\delta} T_{\eta_x} \hat{B}_{2w} - \frac{1}{\delta^2} T_{\eta_x} T_{ \hat{A}_{2ww} } \eta - \partial_x^{\varrho} \hat{A}_2   &\in H^{2n-3}(\mathbb{R}^2).
\end{align*}
If we insert the above formulae in \eqref{eq:ParalinGDNO} we get
\begin{align}
\hat{\mathcal{G}}_{gen,II}[\eta]\hat{\bmomega} &=  T_{ 1+|\nabla\eta|^2 } \bigg( \hat{B}_{2x} - \hat{B}_{1y} + \frac{1}{\delta} T_{ \hat{A}_{2xw} } \eta - \frac{1}{\delta} T_{ \hat{A}_{1yw} } \eta - \frac{1}{\delta}  T_{\eta_x} \hat{B}_{2w} + \frac{1}{\delta}  T_{\eta_y} \hat{B}_{1w} \nonumber \\
&\qquad \qquad \quad - \frac{1}{\delta^2}  T_{\eta_x} T_{ \hat{A}_{2ww} } \eta + \frac{1}{\delta^2} T_{\eta_y} T_{ \hat{A}_{1ww} } \eta \bigg)  \nonumber \\
&\qquad - T_{\nabla\eta} \cdot ( \underline{ \hat{\mathbb{V}} }_{II}[\eta]\hat{\bmomega} )  - T_{ \underline{ \hat{\mathbb{V}} }_{II}[\eta]\hat{\bmomega} } \cdot \nabla\eta - T_{ |\nabla\eta|^2 } \, ( \underline{ \hat{\mathbb{W}} }_{II}[\eta]\hat{\bmomega} )  + \mathcal{R}_{II}, \nonumber \\
&\phantom{} \label{eq:ParalinGDNO2} \\
\mathcal{R}_{II} &\in \;  H^{2n-3}(\mathbb{R}^2), \nonumber 
\end{align}
where by evaluating the interior equation \eqref{eq:main_eq_sys} at the boundary we have
\begin{align*}
\hat{A}_{1ww}  &= \frac{\delta^2}{1+|\nabla\eta|^2} \, \left[ - \Delta \hat{A}_1 + \frac{2}{\delta} \nabla\eta \cdot \nabla \hat{A}_{1w} + \frac{1}{\delta}  \Delta\eta \, \nabla \hat{A}_{1w} - \hat{\omega}_1 \right]  , \\
\hat{A}_{2ww}  &= \frac{\delta^2}{1+|\nabla\eta|^2} \, \left[ - \Delta \hat{A}_2 + \frac{2}{\delta} \nabla\eta \cdot \nabla \hat{A}_{2w} + \frac{1}{\delta} \Delta\eta \, \nabla \hat{A}_{2w} - \hat{\omega}_2 \right]  . 
\end{align*}

Now, recall that Corollary \ref{cor:EqGoodUnknown} implies that for $\varepsilon >0$
\begin{equation*}
\begin{cases}
\hat{B}_{1w} - \left( T_{M_{11}} \hat{B}_1 + T_{M_{12}} \hat{B}_2  + T_{M_{13}} \hat{B}_3  \right) - \int_{-h}^0 (T_{\mathscr{E} } \hat{\bmF}_0)_1(z)  \, \mathrm{d}z   \in H^{2n-3-\varepsilon}(\mathbb{R}^2) , \\ \\
\hat{B}_{2w} - \left( T_{M_{21}} \hat{B}_1 + T_{M_{22}} \hat{B}_2  + T_{M_{23}} \hat{B}_3  \right) - \int_{-h}^0 (T_{\mathscr{E} } \hat{\bmF}_0)_2(z)  \, \mathrm{d}z  \in H^{2n-3-\varepsilon}(\mathbb{R}^2) ,
\end{cases}
\end{equation*}
so that if we introduce $\lambda_{II} \in M_{1,3}(\Sigma^1_{n-2}(\mathbb{R}^2))$ as
\begin{align} 
\lambda_{II} &\coloneqq -
\begin{pmatrix}
1 & 1 & 0
\end{pmatrix} \; \; (1+|\nabla\eta|^2) \times \nonumber \\
&\qquad \times \; \left[ 
\begin{pmatrix}
0 & \mathrm{i} \xi_1 & 0 \\
-\mathrm{i} \xi_2 & 0 & 0 \\
0 & 0 & 0
\end{pmatrix}
+ \frac{1}{\delta}
\begin{pmatrix}
T_{\eta_y} T_{M_{11}} &  T_{\eta_y} T_{M_{12}} &  T_{\eta_y} T_{M_{13}} \\
- T_{\eta_x} T_{M_{21}} & -  T_{\eta_x} T_{M_{22}} & -  T_{\eta_x} T_{M_{23}} \\
0 & 0 & 0
\end{pmatrix}
\right]  
, \label{eq:lambdaII}
\end{align}
we obtain from \eqref{eq:ParalinGDNO2} that
\begin{align}
&\hat{\mathcal{G}}_{gen,II}[\eta]\hat{\bmomega} \nonumber \\
&=  T_{\lambda_{II}} \hat{\bmB} + \frac{1}{\delta} T_{1+|\nabla\eta|^2} \, T_{ \gradp\eta } \cdot \left[ \int_{-h}^0 (T_{\mathscr{E} } \hat{\bmF}_0)_{\rmh}(z)  \, \mathrm{d}z \right] \nonumber \\ 
&\;\;\;\; + T_{1+|\nabla\eta|^2} \left( \frac{1}{\delta}  T_{\hat{A}_{2xw}}\eta - \frac{1}{\delta} T_{\hat{A}_{1yw}}\eta - \frac{1}{\delta^2} T_{\eta_x} T_{\hat{A}_{2ww}}\eta + \frac{1}{\delta^2} T_{\eta_y} T_{\hat{A}_{1ww}}\eta \right) \nonumber \\
&\;\;\;\; - T_{\nabla\eta} \cdot ( \underline{ \hat{\mathbb{V}} }_{II}[\eta]\hat{\bmomega} ) - T_{ \underline{ \hat{\mathbb{V}} }_{II}[\eta]\hat{\bmomega} } \cdot \nabla\eta - T_{ |\nabla\eta|^2 } \, ( \underline{ \hat{\mathbb{W}} }_{II}[\eta]\hat{\bmomega} )  + \mathcal{R}_{II}, \nonumber \\
&\phantom{} \label{eq:ParalinGDNO3}  
\end{align}
where $\mathcal{R}_{II} \in \;  H^{2n-3-\varepsilon}(\mathbb{R}^2) $.

\section*{Acknowledgements}

The author would like to warmly thank Thomas Alazard and Erik Wahl\'en for useful comments and suggestions. This research was supported by the Swedish Research Council under grant no.\ 2021-06594 while the author was in residence at Institut Mittag-Leffler in Djursholm, Sweden in Autumn 2023. S. Pasquali is supported by PRIN 2022 ``Turbulent effects vs Stability in Equations from Oceanography" (TESEO), project number: 2022HSSYPN. S. Pasquali would like to thank INdAM-GNAMPA.

\begin{appendix}

\section[Proof of Proposition $\ref{prop:CLSolBVP}$]{Proof of Proposition \ref{prop:CLSolBVP} } \label{sec:proofDivCurl}

We now prove Proposition \ref{prop:CLSolBVP}. In order to solve the div-curl problem \eqref{eq:BVP}, recall that
\begin{align*}
H(\Div,D_\eta) &\coloneqq \{ \bmF \in (L^{2}(D_\eta))^3 : \Div \bmF \in L^{2}(D_\eta)   \} ,  \\
H(\curl,D_\eta) &\coloneqq \{ \bmF \in (L^{2}(D_\eta))^3 : \curl \bmF \in (L^{2}(D_\eta))^3   \} . 
\end{align*}
It follows that if $\bmomega \in (L^{2}(D_\eta)^3$ and $\bmU$ solves the first two equations of \eqref{eq:BVP}, one can take the normal and the tangential traces at the boundary, from which we derive the boundary conditions of \eqref{eq:BVP}.

Without loss of generality, we can reduce \eqref{eq:BVP} to the case $\Phi=0$. Indeed, let $\varphi \in \dot{H}^2(D_\eta)$ be the solution of the boundary value problem
\begin{equation} \label{eq:BVPups}
\begin{cases}
\Delta_{\bmx,z} \varphi = 0 , & \text{in} \; \; D_\eta, \\
\varphi = \Phi , & \text{at} \; \; z=\eta, \\
\partial_n \varphi = 0 , & \text{at} \; \; z=-h, \\
\end{cases}
\end{equation}
(its existence is guaranteed by the results in Sec. 2.1.3 and Sec. 2.5.2 of \cite{lannes2013water}), then, if we define $\bmv = \bmU - \nabla_{\bmx,z} \varphi$, we have
\begin{equation} \label{eq:BVPv}
\begin{cases}
\curl \bmv = \bmomega, & \text{in} \; \; D_\eta, \\
\Div \bmv = 0 & \text{in} \; \; D_\eta, \\
\bmv_{\parallel} = \gradp\Psi & \text{at} \; \; z=\eta, \\
\bmv \cdot \bme_3 =0 , & \text{at} \; \; z=-h,
\end{cases}
\end{equation}
which is equivalent to \eqref{eq:BVP} with $\Phi=0$. We look for a solution of \eqref{eq:BVPv} of the form $\bmv=\curl \bmA$, where the vector potential $\bmA$ satisfies
\begin{equation} \label{eq:BVPA}
\begin{cases}
\curl \, \curl \bmA = \bmomega , & \text{in} \; \; D_\eta , \\
\bmA \times \bme_3 = \bmzero , & \text{at} \; \; z=-h ,  \\
\bmA \cdot \bmN = 0 , & \text{at} \; \; z=\eta ,  \\
(\curl \bmA)_{\parallel} = -\gradp \Delta^{-1} (\bmomega \cdot \bmN) , & \text{at} \; \; z=\eta .
\end{cases}
\end{equation}
Before proving the existence of a vector potential $\bmA$ solving \eqref{eq:BVPA}, we prove some technical lemmas.

\begin{lemma} \label{lem:TecLemma}

Let $k > \frac{5}{2}$, let $\eta \in H^{k}(\mathbb{R}^2)$ satisfies condition \eqref{eq:StrConnected}. Let also $\bmA \in H(\Div,D_\eta) \cap H(\curl,D_\eta)$ be such that
\begin{equation*}
\begin{cases}
\bmA \times \bme_3 = \bmzero , & \text{at} \; \; z=-h ,  \\
\bmA \cdot \bmN = 0 , & \text{at} \; \; z=\eta .
\end{cases}
\end{equation*}
Then there exists $C>0$ such that
\begin{align*}
& \| \bmA \|_{(L^2(D_\eta))^3} + \|\nabla_{\bmx,z}  \bmA \|_{ (L^2(D_\eta))^{3 \times 3} }^2 \\
&\leq C( h,\|\eta\|_{H^{k}(\mathbb{R}^2)} ) \, \left( \| \curl\bmA \|_{(L^2(D_\eta))^3}^2 + \| \Div\bmA \|_{L^2(D_\eta)}^2  \right) .
\end{align*}

\end{lemma}

\begin{proof}

We just prove the estimate
\begin{align*}
\|\nabla_{\bmx,z}  \bmA \|_{(L^2(D_\eta))^{3 \times 3} }^2 &\leq \| \curl\bmA \|_{(L^2(D_\eta))^3}^2 + \| \Div\bmA \|_{L^2(D_\eta)}^2 +C \; \|\eta\|_{H^{k}(\mathbb{R}^2)  } \; \| \underline{\bmA} \|_{(L^2( \mathbb{R}^2 ))^3}^2 ,
\end{align*}
since the thesis can be deduced by combining Lemma 3.1 of \cite{castro2015well}, together with the above estimate.

Here we adopt the convention of summing over repeated indices and we write $(\partial_i)_{i=1,2,3} =: \nabla_{\bmx,z}$. By arguing as in the proof of Lemma 3.2 of  \cite{castro2015well}, we obtain
\begin{align} 
 \| \nabla_{\bmx,z}\bmA \|_{( L^2(D_{\eta}) )^{3 \times 3} }^2 &=  \sum_{i,j=1}^3 \, \int_{D_\eta} |\partial_j A_i|^2 = \| \curl\bmA \|_{ (L^2(D_\eta))^3 }^2 +  \int_{ D_{\eta} } \partial_j A_i \, \partial_i A_j . \label{eq:nablaA}
\end{align}

Following Sec. 3.1.1 of \cite{alazard2014cauchy}, we rewrite the last integral exploiting a regularizing diffeomorphism in \eqref{eq:straight},
\begin{align*}
\int_{ D_{0} } \partial_j^{\Sigma} \tilde{A}_i \, \partial_i^{\Sigma} \tilde{A}_j \ (1+\sigma_w)  &= -  \int_{ D_{0} }  \tilde{A}_i \, \partial_j^{\Sigma} \partial_i^{\Sigma} \tilde{A}_j \, (1+\sigma_w)   - \sigma_{wj} \, \tilde{A}_i \, \partial_i^{\Sigma} \tilde{A}_j \\
&\;\; + \int_{ \partial D_0 } \left[   n_j \, \tilde{A}_i \, \partial_i^{\Sigma}\tilde{A}_j - \frac{ \sigma_j }{ 1+\sigma_w } \, \tilde{A}_i \, \partial^{\Sigma}_i \tilde{A}_j \right] (1+\sigma_w) ,
\end{align*}
where we are using the notation of Sec. \ref{sec:strBVP}. Notice that
\begin{align*}
\left| \int_{ D_{0} }  \tilde{A}_i \, \partial_j^{\Sigma} \partial_i^{\Sigma} \tilde{A}_j \ (1+\sigma_w)  \right| &\stackrel{ \eqref{eq:smoothing} }{\leq} \| \Div\bmA \|_{L^2(D_\eta)}^2 + \|\eta\|_{H^k(\mathbb{R}^2)} \, \|  \partial_j^{\Sigma} \tilde{A}_j \|_{ L^2(D_0) } \, \| \tilde{A}_i \|_{ L^2(D_0) } \\
&\qquad + C( \|\sigma\|_{C^1(D_0)} ) \left[ \left| \int_{\partial D_0 } n_i \tilde{A}_i \partial_j^{\Sigma} \tilde{A}_j \right| + \| \tilde{\bmA} \|_{ ( L^2(D_0) )^3 }^2 \right] .
\end{align*}
Hence,
\begin{align*}
&\| \nabla_{\bmx,z}\bmA \|_{( L^2(D_{\eta}) )^{3 \times 3} }^2 \\
&\leq  \| \curl\bmA \|_{ (L^2(D_\eta))^3 }^2 + \| \Div\bmA \|_{L^2(D_\eta)}^2  \\
&\;\; + 2 \|\eta\|_{H^k(\mathbb{R}^2)} \, \|  \partial_j^{\Sigma} \tilde{A}_j \|_{ L^2(D_0) } \, \| \tilde{A}_i \|_{ L^2(D_0) } \\
&\;\; + C( \|\sigma\|_{C^1(D_0)} ) \left[ \left| \int_{\partial D_0 } n_i \tilde{A}_i \partial_j^{\Sigma} \tilde{A}_j \right| + \left| \int_{ \partial D_0 } n_j \, \tilde{A}_i \, \partial_i^{\Sigma}\tilde{A}_j  \right| +  \| \tilde{\bmA} \|_{ ( L^2(D_0) )^3 }^2 \right]
\end{align*}

Therefore, by exploiting that $\bmA \in H(\Div,D_{\eta})$ and Lemma 3.1 in \cite{castro2015well} 
\begin{align*}
\| \nabla_{\bmx,z}\bmA \|_{( L^2(D_{\eta}) )^{3 \times 3} }^2 &\leq \left( 1+ C( h,  \|\sigma\|_{C^1(D_0)} ) \right) \times \\
&\times  \bigg\{  \| \curl\bmA \|_{ (L^2(D_\eta))^3 }^2 + \| \Div\bmA \|_{L^2(D_\eta)}^2 \\
&\qquad   + C( \|\sigma\|_{C^1(D_0)} ) \left[ \left| \int_{\partial D_0 } n_i \tilde{A}_i \partial_j^{\Sigma} \tilde{A}_j \right| + \left| \int_{ \partial D_0 } n_j \, \tilde{A}_i \, \partial_i^{\Sigma}\tilde{A}_j  \right| \right] \bigg\} .
\end{align*}

Now we evaluate the boundary integral on the right-hand side of \eqref{eq:nablaA}. When we consider the bottom contribution, we observe that 
\begin{align*}
(\bmn \cdot \tilde{\bmA}) \; \Div^{\Sigma} \tilde{\bmA} |_{w=-h} &= \tilde{\bmA} \cdot \partial_n^{\Sigma} \tilde{\bmA} ,
\end{align*}
because the bottom is flat. Since
\begin{align*}
\int_{ \{ w=-h \} } n_i \,\tilde{A}_j \, (\partial_j^{\Sigma} \tilde{A}_i) &= \int_{ \{ w=-h \} } \tilde{\bmA} \cdot \partial_n^{\Sigma} \tilde{\bmA }
\end{align*}
because $\tilde{\bmA}$ is normal to the bottom, we obtain that 
\begin{align*}
\int_{ \{w=-h\}   } n_i \, \tilde{A}_j \, (\partial_j^{\Sigma} \tilde{A}_{i}) - (\bmn \cdot \tilde{\bmA}) \, \Div^{\Sigma} \tilde{\bmA} &= 0 .
\end{align*}
If we consider the contribution at the free surface, by using \eqref{eq:smoothing} we have that there exists $C>0$ such that
\begin{align*}
\int_{ \{ w=0 \} } n_j \, \tilde{A}_i \, (\partial_i^{\Sigma} \tilde{A}_j) &\leq C \, \|\eta\|_{H^{k}(\mathbb{R}^2)} \, \| \tilde{\bmA} |_{w=0} \|_{(L^2( \mathbb{R}^2 ))^3}^2 .
\end{align*}

Hence, we can deduce that there exists $C>0$ such that
\begin{align*}
\int_{ \{w=0\}   } n_i \, \tilde{A}_j \, (\partial_j^{\Sigma} \tilde{A}_{i}) - (\bmn \cdot \tilde{\bmA}) \, \Div^{\Sigma} \tilde{\bmA} &\leq C \;  \, \|\eta\|_{H^{k}(\mathbb{R}^2)} \, \| \tilde{\bmA} |_{w=0} \|_{(L^2( \mathbb{R}^2 ))^3}^2 .
\end{align*}

\end{proof}

We now construct a solution to \eqref{eq:BVPA}, imposing in addition that $\bmA$ is divergence free. More explicitly, we study the boundary value problem
\begin{equation} \label{eq:newBVPA}
\begin{cases}
\curl \, \curl \bmA = \bmomega , & \text{in} \; \; D_\eta , \\
\Div \, \bmA = 0 , & \text{in} \; \; D_\eta , \\
\bmA \times \bme_3 = \bmzero , & \text{at} \; \; z=-h ,  \\
\bmA \cdot \bmN = 0 , & \text{at} \; \; z=\eta ,  \\
(\curl \bmA)_{\parallel} = \gradp \Psi , & \text{at} \; \; z=\eta ,
\end{cases}
\end{equation}
with $\Psi \in \dot{H}^{1/2}(\mathbb{R}^2)$. Instead of adopting an argument analogous to Definition 3.4 - Lemma 3.6 in \cite{castro2015well},  we restate the boundary value problem \eqref{eq:systA} in an equivalent way, along the lines of the approach outlined in Sec. 4 of \cite{groves2020variational} for Beltrami flows; this will allow us to obtain a solution of \eqref{eq:systA} by solving (the flattened version of) an equivalent problem, whose solution can be explicitly written in terms of a Green matrix. \\

Let $k> \frac{5}{2}$, and let $\eta \in H^{k}(\mathbb{R}^2)$ be such that \eqref{eq:StrConnected} holds true. Let us assume that $\bmomega \in H_b( \Div_0,D_{\eta})$ and that $\Phi \in \dot{H}^{3/2}(\mathbb{R}^2)$; we introduce the following subspace of $( H^1(D_{\eta}) )^3$,
\begin{align} \label{eq:Xeta}
\mathfrak{X}_{\eta} &\coloneqq \{ \bmF \in ( H^1(D_{\eta}) )^3 : \bmF \times \bme_3|_{z=-h} = \bmzero, \bmF \cdot \bmn = 0 \}.
\end{align}
By Corollary \ref{cor:Xeta} we introduce the following notion: we call a \emph{weak solution} of \eqref{eq:systA} a function $\bmA \in \mathfrak{X}_{\eta}$  such that
\begin{align} \label{eq:WeakSol}
\int_{D_{\eta}} \curl\bmA \cdot \curl\bmB + \Div\bmA \, \Div\bmB &= \int_{D_{\eta}} \bmomega \cdot \bmB + \int_{\mathbb{R}^2}  \nabla \, \Delta^{-1}(\underline{\bmomega} \cdot \bmN)  \cdot \bmB_{\parallel}, \;\;  \forall \bmB \in \mathfrak{X}_{\eta}, 
\end{align}
while we call a \emph{strong solution} of \eqref{eq:systA} a weak solution such that $\bmA \in ( H^2(D_{\eta}) )^3$ is solenoidal and satisfies 
\begin{align*}
\curl \curl \bmA &= \bmomega , \; \; \text{in} \; \; D_{\eta}, 
\end{align*}
in $(L^2(D_{\eta}))^3$, and
\begin{align*}
(\curl\bmA)_{\parallel} &= - \gradp \, \Delta^{-1} (\underline{\bmomega} \cdot \bmN) 
\end{align*}
in $(H^{1/2}(\mathbb{R}^2))^2$. \\

We mention an immediate consequence of Lemma \ref{lem:TecLemma} regarding a characterization of a closed subspace of $( H^1(D_{\eta}) )^3$; such result will be used in Sec. \ref{sec:GDNOhom} for the homogeneous expansion of the generalized Dirichlet--Neumann operator (see also Proposition 4.2 in \cite{groves2020variational} for an analogous result for Beltrami flows).

\begin{corollary} \label{cor:Xeta}

Let $k > \frac{5}{2}$, and let $\eta \in H^{k}(\mathbb{R}^2)$ be such that \eqref{eq:StrConnected} holds true. Then
\begin{align*}
\mathfrak{X}_{\eta} &= \big\{ \bmF \in ( L^2(D_{\eta}) )^3 : \curl \bmF \in ( L^2(D_{\eta}) )^3, \Div\bmF \in L^2(D_{\eta}), \\
&\qquad \qquad \qquad \qquad \quad \bmF \times \bme_3|_{z=-h} = \bmzero, \bmF \cdot \bmn = 0 \big\} ,
\end{align*}
and the function $\bmF \mapsto \left( \| \curl \bmF \|_{( L^2(D_{\eta}) )^3 } + \| \Div \bmF \|_{ L^2(D_{\eta})  } \right)^{1/2}$ is equivalent to its usual norm.
\end{corollary}

Now we prove state two result about the existence of weak and strong solutions for \eqref{eq:systA}. These results can be proved as Lemma 4.5 and Lemma 4.7 of \cite{groves2020variational}.

\begin{proposition} \label{prop:ExistWeak}

Let $\bmomega \in H_b( \Div_0,D_{\eta})$ and $\Phi \in \dot{H}^{1/2}(\mathbb{R}^2)$, then the boundary value problem \eqref{eq:systA} admits a weak solution $\bmA$. Moreover, $\bmA$ is solenoidal, and it satisfies
\begin{align*}
\curl \curl \bmA &= \bmomega , \; \; \text{in} \; \; D_{\eta}, 
\end{align*}
in the sense of distributions, and it satisfies
\begin{align*}
(\curl\bmA)_{\parallel} &= - \gradp \, \Delta^{-1} (\underline{\bmomega} \cdot \bmN)
\end{align*}
in $(H^{-1/2}(\mathbb{R}^2))^2$.
\end{proposition}

\begin{proposition} \label{prop:WeakStrong}

Let $\bmomega \in H_b( \Div_0,D_{\eta})$ and $\Phi \in \dot{H}^{3/2}(\mathbb{R}^2)$, then any weak solution $\bmA$ of the boundary value problem \eqref{eq:systA} is a strong solution. 
\end{proposition}

Finally, we state a more quantitative reformulation of Proposition \ref{prop:ExistWeak} (for the proof see Lemma 3.5 in \cite{castro2015well}).

\begin{lemma} \label{lem:TecLemma41}

Let $k > \frac{5}{2}$, and assume that $\eta \in H^{k}(\mathbb{R}^2)$ satisfies condition \eqref{eq:StrConnected}, and let $\bmomega \in H_b( \Div_0,D_{\eta})$ and $\Psi \in \dot{H}^{1/2}(\mathbb{R}^2)$. Then there exists a unique $\bmA \in \mathfrak{X}_{\eta}$ to \eqref{eq:newBVPA}. Moreover, there exists $C>0$ be such that 
\begin{align*}
\| \curl\bmA \|_{(L^2(D_\eta))^3} &\leq C( h,\|\eta\|_{H^{k}(\mathbb{R}^2)} ) \, \left( \| \bmomega \|_{(L^2(D_\eta))^3} + \| \nabla\Psi \|_{ ( H^{-1/2}(\mathbb{R}^2) )^2 }  \right) .
\end{align*}

\end{lemma}

Now consider $\bmA \in \mathfrak{X}_\eta$ as the solution given by Proposition \ref{prop:ExistWeak} and Proposition \ref{prop:WeakStrong}, and set $\bmv \coloneqq \curl\bmA$. We have that $\bmv \in H(\Div,D_{\eta}) \cap H(\curl,D_{\eta})$, and from Lemma \ref{lem:TecLemma41} we get
\begin{align} \label{eq:bfvEst}
\| \bmv \|_{ (L^2(D_\eta))^3 } &\leq C(h,\|\eta\|_{H^{k}(\mathbb{R}^2)} ) \, \left[ \|\bmomega\|_{(L^2(D_\eta))^3} + \|\nabla\Psi\|_{ ( H^{-1/2}(\mathbb{R}^2) )^2 }  \right] .
\end{align}

We conclude the proof of Proposition \ref{prop:CLSolBVP} by setting $\bmU = \bmv + \nabla_{\bmx,z} \varphi$, where $\varphi$ solves \eqref{eq:BVPups}, and arguing as in Sec. 3 of \cite{castro2015well}. The uniqueness of the solution $\bmU$ follows from Lemma 3.8 of \cite{castro2015well}, which also leads to the estimate
\begin{align*}
\| \nabla_{\bmx,z} \varphi \|_{( H^1(D_{\eta}) )^3} &\leq C( h_0^{-1} , \|\eta\|_{H^{k}(\mathbb{R}^2)} ) \, \| \nabla\Phi \|_{ ( H^{1/2}(\mathbb{R}^2) )^2 } .
\end{align*}

\section{Paradifferential calculus} \label{sec:Paradiff}

Here we recall some standard rules of paradifferential calculus (see Sec. 5 of \cite{metivier2008paradifferential} and Sec. 4 of \cite{alazard2009paralinearization}). The paradifferential calculus was introduced by Bony in order to deal with the quantization of symbols $a(\bmx,\bmxi)$ of degree $m$ with respect to $\bmxi$ and limited regularity in the $\bmx$-variable, to which are associated operators of order $\leq m$ denoted by $T_a$.

For any $u \in \mathcal{S}'(\mathbb{R}^d)$ we denote its Fourier transform by $\mathscr{F}u$. If $p: \mathcal{S}(\mathbb{R}^d) \to \mathcal{S}'(\mathbb{R}^d)$, then 
\begin{align*}
p(\bfD)u &:= \mathscr{F}^{-1}(p \; \mathscr{F}u).
\end{align*}
In the following we also write $\mathbb{R}^d_{\neq} := \mathbb{R}^d\setminus\{ \bmzero \}$.

\begin{definition} \label{def:symbolsPara}
Let $\varrho \geq 0$ and $m \in \mathbb{R}$, we denote by $\Gamma^m_\varrho(\mathbb{R}^d)$ the space of locally bounded functions $a(\bmx,\bmxi) \in \mathbb{R}^d \times \mathbb{R}^d_{\neq} $ which are of class $C^\infty$ with respect to $\bmxi$ for $\bmxi \neq \bmzero$ and such that for all $\alpha \in \mathbb{N}^2$ and all $\bmxi \neq \bmzero$ the map $\bmx \mapsto \pd_{\bmxi}^\alpha a(\bmx,\bmxi)$ belongs to $C^\varrho(\mathbb{R}^d)$ and there exists $C_\alpha >0$ such that
\begin{align}
\forall |\bmxi| \geq 1/2, \; \; \; \|\pd_{\bmxi}^\alpha a(\cdot,\bmxi)\|_{ C^\varrho(\mathbb{R}^d) } &\leq C_\alpha (1 + |\bmxi|)^{m-|\alpha|}. \label{eq:symbolsEst}
\end{align}
\end{definition}

Note that in the above definition we considered symbols which are not necessarily smooth for $\bmxi = \bmzero$. The motivation comes from the principal symbol $\lambda^{(1)}(\bmx,\bmxi)$ of the classical Dirichlet--Neumann operator. One can notice that this symbol is not of class $C^\infty$ with respect to the variable $\bmxi$; moreover, if $\eta \in C^s(\mathbb{R}^d)$, then $\lambda^{(1)} \in \Gamma^1_{s-1}(\mathbb{R}^d)$. 

We now introduce the space of pluri-homogeneous symbols.

\begin{definition} \label{def:symbolsHomPara}
Let $\varrho \geq 1$ and $m \in \mathbb{R}$, we denote by $\Sigma^m_\varrho(\mathbb{R}^d)$ the space of symbols of the form
\begin{align*}
a(\bmx,\bmxi) &= \sum_{0 \leq j < \varrho} a_{m-j}(\bmx,\bmxi),
\end{align*}
where $a_{m-j} \in \Gamma^{m-j}_{\varrho-j}(\mathbb{R}^d)$ is homogeneous of degree $m-j$ in the variable $\bmxi$ and is of class $C^\infty$ in the variable $\bmxi$ for $\bmxi \neq \bmzero$, and with regularity $C^{\varrho -j}$ in the variable $\bmx$. We say that $a_m$ is the principal symbol of $a$.
\end{definition}

In order to define rigorously paradifferential operators we first introduce a fixed cutoff functions $\chi$. Fix $\varepsilon_1$, $\varepsilon_2$ such that $0 < 2 \varepsilon_1 < \varepsilon_2 < 1/2$ and a function $f \in C^{\infty}_0(\mathbb{R}^d)$ satisfying
$f (t) = f (-t)$, $f (t) = 1$ for $|t| \leq 2 \varepsilon_1$ and $f (t) = 0$ for $|t| \geq \varepsilon_2$; then set
\begin{align*}
\chi(\bmxi_1 , \bmxi_2 ) = (1 - f (\bmxi_2 )) \, f( |\bmxi_1| / |\bmxi_2| ) .
\end{align*}
Observe that the cutoff function satisfies the following properties:
\begin{itemize}
\item[i.] there exists $\varepsilon_1$, $\varepsilon_2$ satisfying $0 < 2 \varepsilon_1 < \varepsilon_2 < 1/2$ such that
\begin{align*}
\chi(\bmxi_1 , \bmxi_2 ) &= 1, \; \; \text{if} \; \; |\bmxi_1 | \leq \varepsilon_1 |\bmxi_2 | \; \; \text{and} \; \; |\bmxi_2 | \geq 2, \\
\chi(\bmxi_1 , \bmxi_2 ) &= 0, \; \; \text{if} \; \; |\bmxi_1 | \geq \varepsilon_2 |\bmxi_2 | \; \; \text{or} \; \; |\bmxi_2 | \geq 1 ;
\end{align*}
\item[ii.] for all $(\alpha, \beta) \in \mathbb{N}^d$, there exists $C_{\alpha,\beta}>0$ such that
\begin{align*}
| \partial_{\bmxi_1}^{\alpha} \partial_{\bmxi_2}^{\beta}  \chi(\bmxi_1 , \bmxi_2 ) | &\leq C_{\alpha,\beta} \, (1 + |\bmxi_2 |)^{ -|\alpha|-|\beta|} ;
\end{align*}
\item[iii.] $\chi$ satisfies the following symmetry condition,
\begin{align*}
\chi(\bmxi_1,\bmxi_2) &= \chi(-\bmxi_1,-\bmxi_2) \, = \, \chi(-\bmxi_1,\bmxi_2) .
\end{align*}
\end{itemize}

Given a symbol $a \in \Sigma^m_\varrho(\mathbb{R}^d)$, we define the paradifferential operator $T_a$ by
\begin{align}
\mathscr{F}(T_au)(\bmxi) &:= (2\pi)^{-d} \int_{\mathbb{R}^d} \chi(\bmxi-\bmxi',\bmxi') \; \; \mathscr{F}a(\bmxi-\bmxi',\bmxi') \; \; \mathscr{F}u(\bmxi') \; \mathrm{d}\bmxi', \label{eq:ParadiffOp}
\end{align}
where $\mathscr{F}a$ is the Fourier transform of $a$ with respect to the first variable.  \\

We notice that if $Q(D)$ is a Fourier multiplier with symbol $q(\bmxi)$, we do not have $Q(D) = T_q$, because of the cut-off function used in the definition \eqref{eq:ParadiffOp}; however, we have $Q(D) = T_q + R$, where $R: H^t(\mathbb{R}^d) \to H^\infty(\mathbb{R}^d)$ for all $t \in \mathbb{R}$.

\begin{definition} \label{eq:OrdOp}
Let $m \in \mathbb{R}$, then an operator $T$ is of order $m$ if, for all $s \in \mathbb{R}$, it is bounded from $H^{s+m}(\mathbb{R}^d)$ to $H^{s}(\mathbb{R}^d)$.
\end{definition}

\begin{proposition} \label{prop:OrdParadiffOp}
Let $m \in \mathbb{R}$ and let $a \in \Gamma^m_0(\mathbb{R}^d)$, then $T_a$ is of order $m$.
\end{proposition}

Recall the following properties
\begin{proposition} \label{prop:CompParadiffOp}
Let $m_1,m_2 \in \mathbb{R}$ and let $a \in \Gamma^{m_1}_1(\mathbb{R}^d)$, $b \in \Gamma^{m_2}_1(\mathbb{R}^d)$, then $T_a T_b - T_{ab}$ is of order $m_1+m_2-1$.
\end{proposition}

\begin{proposition} \label{prop:AsympParadiffOp}
Let $\varrho >0$, $m_1,m_2 \in \mathbb{R}$ and let $a \in \Gamma^{m_1}_\varrho(\mathbb{R}^d)$, $b \in \Gamma^{m_2}_\varrho(\mathbb{R}^d)$. Set
\begin{align*}
a \sharp b (\bmx,\bmxi) &:= \sum_{|\alpha| < \varrho} \frac{ (-\mathrm{i})^\alpha }{\alpha!} \; \; \pd_{\bmxi}^\alpha a(\bmx,\bmxi) \; \; \pd_{\bmx}^\alpha b(\bmx,\bmxi) \in \sum_{j < \varrho} \Gamma^{m_1+m_2-j}_{\varrho-j}(\mathbb{R}^d),
\end{align*}
then $T_a T_b - T_{a \sharp b}$ is of order $\leq m_1+m_2-\varrho$.
\end{proposition}

If $a=a(\bmx)$ does not depend on the variable $\bmxi$, then the paradifferential operator $T_a$ is called \emph{paraproduct}. From Proposition \ref{prop:OrdParadiffOp} we have that if $b \in H^\beta(\mathbb{R}^d)$ with $\beta > \frac{d}{2}$, then $T_b$ is of order $0$. Moreover, we recall the following properties

\begin{lemma} \label{lem:Paraprod1}
Let $m_1 \in \mathbb{R}$, $m_2 < d/2$, $a \in H^{m_1}(\mathbb{R}^d)$, $b \in H^{m_2}(\mathbb{R}^d)$. Then $T_b a \in H^{m_1+m_2-d/2}(\mathbb{R}^d)$.

Moreover, if $a \in L^\infty(\mathbb{R}^d)$, then $T_a$ is an operator of order $\leq 0$, and there exists $C>0$ such that
\begin{align*}
\| T_a u \|_{ H^s(\mathbb{R}^d) } &\leq C \, \|a\|_{L^\infty(\mathbb{R}^d)} \, \|  u \|_{ H^s(\mathbb{R}^d) } , \; \; \forall u \in H^s(\mathbb{R}^d), \; \; \forall s \in \mathbb{R}.
\end{align*}
Similarly, if $a \in L^\infty(\mathbb{R}^d)$, then for all $s \in (0,+\infty) \setminus \mathbb{N}$ there exists $C>0$ such that
\begin{align*}
\| T_a u \|_{ C^s(\mathbb{R}^d) } &\leq C \, \|a\|_{L^\infty(\mathbb{R}^d)} \, \|  u \|_{ C^s(\mathbb{R}^d) } , \forall u \in C^s(\mathbb{R}^d)  .
\end{align*}
\end{lemma}

\begin{lemma} \label{lem:Paraprod2}
Let $m > d/2$, $a \in H^{m}(\mathbb{R}^d)$, $F \in C^{\infty}(\mathbb{R})$. Then $F(a)-T_{F'(a)} a \in H^{2m-d/2}(\mathbb{R}^d)$. 

Moreover, let $m_1,m_2 \in \mathbb{R}$ be such that $m_1 + m_2 >0$, and let $a \in H^{m_1}(\mathbb{R}^d)$, $b \in H^{m_2}(\mathbb{R}^d)$. Then 
\begin{align*}
R(a,b) &:= ab - T_a b - T_b a \in H^{m_1 + m_2 - d/2}(\mathbb{R}^d) ,
\end{align*}
and there exists $K>0$ such that
\begin{align*}
\| R(a,b) \|_{ H^{m_1 + m_2 - d/2}(\mathbb{R}^d) } &\leq K \, \| a \|_{H^{m_1}(\mathbb{R}^d)} \, \| b \|_{H^{m_2}(\mathbb{R}^d)} .
\end{align*}
Similarly, let $m_1 \in \mathbb{R}_+$, $m_2 \in \mathbb{R}$ be such that $m_1 + m_2 >0$, and let $a \in C^{m_1}(\mathbb{R}^d)$, $b \in H^{m_2}(\mathbb{R}^d)$. Then 
\begin{align*}
R(a,b) &= ab - T_a b - T_b a \in H^{m_1 + m_2}(\mathbb{R}^d) ,
\end{align*}
and there exists $K>0$ such that
\begin{align*}
\| R(a,b) \|_{ H^{m_1 + m_2 }(\mathbb{R}^d) } &\leq K \, \| a \|_{C^{m_1}(\mathbb{R}^d)} \, \| b \|_{H^{m_2}(\mathbb{R}^d)} .
\end{align*}

\end{lemma}

\begin{lemma} \label{lem:Paraprod3}
Let $m > \frac{d}{2}$ be such that $m - \frac{d}{2} \notin \mathbb{N}$. If $a \in H^{m}(\mathbb{R}^d)$, $b \in H^{m}(\mathbb{R}^d)$, then $T_a T_b - T_{ab}$ is of order $ - \left( m - \frac{d}{2} \right)$.
\end{lemma}

Next, we mention some product estimates (see  (A.120)-(A.121) of \cite{alazard2015sobolev} for the case $d=2$; see Proposition 9.10 of \cite{taylor2011partial3} for \eqref{eq:ADProduct} and Proposition 8.6.8 of \cite{hormander1997lectures} for \eqref{eq:ZygProduct}).

\begin{lemma} \label{lem:Product}

For any $(r,s,s') \in \mathbb{R}_+^3$ such that $s' > s \geq r$, then there exists a constant $C>0$ such that
\begin{align*}
\| T_a u \|_{ H^{s-r}(\mathbb{R}^2) } &\leq C \, \|a\|_{H^{-r}(\mathbb{R}^2)} \, \|  u \|_{ C^{s'}(\mathbb{R}^2) } , \; \; \forall a \in H^{-r}(\mathbb{R}^2), \; \; \forall u \in C^{s'}(\mathbb{R}^2) .
\end{align*}
Moreover, let $d \geq 1$, $s \in \mathbb{R}$ and $s' \in \mathbb{R}_+$ be such that $s' > |s|$, then there exists a constant $C>0$ such that
\begin{align} \label{eq:ADProduct}
\| f \, g \|_{ H^{s}(\mathbb{R}^d) } &\leq C \, \|f\|_{H^{s}(\mathbb{R}^d)} \, \|  g \|_{ C^{s'}(\mathbb{R}^d) } , \; \; \forall f \in H^{s}(\mathbb{R}^d), \; \; \forall g \in C^{s'}(\mathbb{R}^d) .
\end{align}
Finally, let $d \geq 1$, $s_1,s_2 \in \mathbb{R}$ be such that $s_1 + s_2 >0$, and set $s := \min(s_1,s_2)$; then there exists a constant $C=C(s_1,s_2)>0$ such that
\begin{align} \label{eq:ZygProduct}
\| f \, g \|_{ C^{s}(\mathbb{R}^d) } &\leq C \, \|f \|_{C^{s_1}(\mathbb{R}^d)} \, \|  g \|_{ C^{s_2}(\mathbb{R}^d) } , \; \; \forall f \in C^{s_1}(\mathbb{R}^d), \; \; \forall g \in C^{s_2}(\mathbb{R}^d) .
\end{align}

\end{lemma}

\begin{remark} \label{rem:ProductRem}

We recall the following product estimates (see Theorem 8.3.1 in \cite{hormander1997lectures}). 

Let $d \geq 1$, and let $r, r_1, r_2 \in \mathbb{R}$ be such that $r_1+r_2 \geq 0$, $r \leq \min(r_1,r_2)$, $r \leq r_1 +r_2 - d/2$ ($r <  r_1 + r_2 -d/2$ if either $r_1$, $r_2$ or $-r$ is equal to $d/2$), then 
\begin{align} \label{eq:SobProduct}
\| f \, g \|_{ H^{r}(\mathbb{R}^d) } &\leq C \, \|f\|_{H^{r_1}(\mathbb{R}^d)} \, \|  g \|_{ H^{r_2}(\mathbb{R}^d) } , \; \; \forall f \in H^{r_1}(\mathbb{R}^d), \; \; \forall g \in H^{r_2}(\mathbb{R}^d) .
\end{align}

Moreover, let $d \geq 1$, $r>0$ and $f , g \in H^{r}(\mathbb{R}^d) \cap L^{\infty}(\mathbb{R}^d)$, then
\begin{align} \label{eq:TameProduct}
\| f \, g \|_{ H^{r}(\mathbb{R}^d) } &\leq C \, \left( \|f\|_{H^{r}(\mathbb{R}^d)} \, \|  g \|_{ L^{\infty}(\mathbb{R}^d) } + \|f\|_{L^{\infty}(\mathbb{R}^d)} \, \|  g \|_{ H^{r}(\mathbb{R}^d) }\right) .
\end{align}

\end{remark}

We also mention a result concerning the action of pseudodifferential operators on H\"older spaces, see (1.1) and Proposition 1.1 in \cite{taylor2015commutator}.

\begin{lemma} \label{lem:TaylorHolder}

Let $d \geq 1$, $m,r \in \mathbb{R}$, $P \in OPS^m_{1,0}(\mathbb{R}^d)$, then $P : C^r_{\ast}(\mathbb{R}^d) \to C^{r-m}_{\ast}(\mathbb{R}^d)$.

\end{lemma}

Next, we mention the following result (see  Lemma 2.19 in \cite{alazard2014cauchy} for a slightly more general statement; see also Proposition 2.12 in \cite{lannes2013water}).

\begin{lemma} \label{lem:trace}

Let $h>0$, $s \in \mathbb{R}$, and let $\tilde{f}:D_0 \to \mathbb{R}$ be such that 
\begin{equation*}
\tilde{f} \in L^2([-h,0];H^{s+1/2}(\mathbb{R}^2)) , \;\; \partial_w \tilde{f} \in L^2([-h,0];H^{s-1/2}(\mathbb{R}^2)) .
\end{equation*}
Then $\tilde{f} \in C^{0}([-h,0];H^{s}(\mathbb{R}^2))$, and there exists a constant $C > 0$ such that
\begin{align*}
\| \tilde{f} \|_{ C^{0}([-h,0];H^{s}(\mathbb{R}^2)) } &\leq C \left[ \| \tilde{f} \|_{ L^{2}([-h,0];H^{s+1/2}(\mathbb{R}^2)) } + \| \partial_w \tilde{f} \|_{ L^{2}([-h,0];H^{s-1/2}(\mathbb{R}^2)) }  \right] .
\end{align*}

\end{lemma}

Finally, we state the following elliptic regularity result.

\begin{proposition} \label{prop:MaxRegularity}

Let $\tau>0$, $0 \leq r < 1$, $A \in M_3( \Gamma^1_{1+r}(\mathbb{R}^2)  )$, and $B \in M_3(  \Gamma^0_0(\mathbb{R}^2) )$. Assume that there exists $c>0$ such that
\begin{align*}
\forall (\bmx, \bmxi) \in \mathbb{R}^2 \times \mathbb{R}^2, &\; \; |Re \, A(\bmx,\bmxi) | \geq c |\bmxi| .
\end{align*}
If $\bmU \in C^1([-\tau,0]; ( H^{-\infty}(\mathbb{R}^2) )^3 )$ solves 
\begin{align*}
\partial_w \bmU + T_A \bmU &= T_B \bmU + \bmF,
\end{align*}
where $\bmF \in C^0([-\tau,0]; ( H^{s}(\mathbb{R}^2) )^3 )$ for some $s \in \mathbb{R}$, then 
\begin{equation*}
\bmU(0) \in ( H^{s+r}(\mathbb{R}^2) )^3 . 
\end{equation*}

Moreover, if $\bmU \in C^1([-\tau,0]; ( H^{s}(\mathbb{R}^2) )^3 ) \cap C^0([-\tau,0]; (H^{s-1}(\mathbb{R}^2))^3  )$, then 
\begin{equation*}
\bmU \in C^0([-\tau,0]; ( H^{s+r}(\mathbb{R}^2) )^3 ) ,
\end{equation*}
and there exists $K>0$ such that
\begin{align*}
& \| \bmU \|_{ L^{\infty}([-\tau,0];H^{s+r}(\mathbb{R}^2,\mathbb{R}^3))  } \leq K \times \\
&\times \left[   \| \bmU(0) \|_{( H^{s+r}(\mathbb{R}^2) )^3  } +  \| \bmF \|_{ L^{\infty}([-\tau,0]; ( H^{s}(\mathbb{R}^2) )^3 )  } + \| \bmU \|_{ L^{\infty}([-\tau,0]; ( H^{s}(\mathbb{R}^2) )^3 )  } \right].
\end{align*}
\end{proposition}

\begin{proof}

The proof is analogous to the one of Proposition 4.10 of \cite{alazard2009paralinearization} (see also Lemma 2.2.7 of \cite{alazard2015sobolev}).
\end{proof}

\end{appendix}

\bibliography{Water_Waves_Vorticity}
\bibliographystyle{alpha}

\end{document}